\documentclass[12pt]{report}

\usepackage{amsmath}
\usepackage{amssymb}
\usepackage{amsfonts}
\usepackage{amsthm}
\usepackage[all]{xy}
\usepackage{tikz}
\usepackage{tikz-cd}
\usepackage[T5,T1]{fontenc}
\usepackage{hyperref}
\usepackage{setspace}

\usetikzlibrary{matrix,arrows,decorations.pathmorphing}

\theoremstyle{plain}
\newtheorem{thm}{Theorem}[section]
\newtheorem{lem}[thm]{Lemma}
\newtheorem{cor}[thm]{Corollary}
\newtheorem{prop}[thm]{Proposition}

\newtheorem{conj}[thm]{Conjecture}

\theoremstyle{definition}
\newtheorem{defn}[thm]{Definition}
\newtheorem{ex}[thm]{Example}
\newtheorem{examples}[thm]{Examples}

\newtheorem*{claim}{Claim}
\newtheorem{rmk}[thm]{Remark}
\newtheorem{rmks}[thm]{Remarks}

\newcommand{\gr}{{\rm gr}}

\newcommand{\surj}{\twoheadrightarrow}
\newcommand{\inj}{\hookrightarrow}

\newcommand{\Br}{{\rm Br}}

\newcommand{\rank}{{\rm rank}}

\newcommand{\Hom}{{\rm Hom}}
\newcommand{\im}{{\rm im}}

\newcommand{\Tr}{{\rm Tr}}

\newcommand{\Gal}{{\rm Gal}}

\newcommand{\sA}{{\mathcal A}}
\newcommand{\sB}{{\mathcal B}}
\newcommand{\sC}{{\mathcal C}}

\newcommand{\sE}{{\mathcal E}}

\newcommand{\sG}{{\mathcal G}}

\newcommand{\sI}{{\mathcal I}}

\newcommand{\sN}{{\mathcal N}}
\newcommand{\sO}{{\mathcal O}}

\newcommand{\sS}{{\mathcal S}}

\newcommand{\sZ}{{\mathcal Z}}

\newcommand{\F}{{\mathbb F}}

\renewcommand{\H}{{\mathbb H}}

\newcommand{\N}{{\mathbb N}}

\newcommand{\Q}{{\mathbb Q}}
\newcommand{\R}{{\mathbb R}}

\newcommand{\U}{{\mathbb U}}

\newcommand{\Z}{{\mathbb Z}}

\newcommand{\fN}{{\mathfrak N}}

\newcommand{\fx}{{ F^\times}}
\newcommand{\fxs}{{ (F^\times)^2 }}
\newcommand{\kx}{{ K^\times}}
\newcommand{\kxs}{{ (K^\times)^2 }}

\makeatletter

{\cleardoublepage
 \begin{center}
  \section*{Acknowledgements}
 \end{center}
 \begingroup
}{\newpage\endgroup}

{\cleardoublepage 
 \begin{center}
  \section*{Dedication}
 \end{center}
 \begingroup
}{\newpage\endgroup}

\newenvironment{preliminary}%
{\pagestyle{plain}\pagenumbering{roman}}%
{\pagenumbering{arabic}}

\if@twoside 
\def\ps@thesis{\let\@mkboth\markboth
   \def\@oddfoot{}
   \let\@evenfoot\@oddfoot
   \def\@oddhead{
      {\sc\rightmark} \hfil \rm\thepage
      }
   \def\@evenhead{
      \rm\thepage \hfil {\sc\leftmark}
      }
   \def\chaptermark##1{\markboth{\ifnum \c@secnumdepth >\m@ne
      Chapter\ \thechapter. \ \fi ##1}{}}
   \def\sectionmark##1{\markright{\ifnum \c@secnumdepth >\z@
      \thesection. \ \fi ##1}}}
\else 
\def\ps@thesis{\let\@mkboth\markboth
   \def\@oddfoot{}
   \def\@oddhead{
      {\sc\rightmark} \hfil \rm\thepage
      }
   \def\chaptermark##1{\markright{\ifnum \c@secnumdepth >\m@ne
      Chapter\ \thechapter. \ \fi ##1}}}
\fi

\newcommand\isco[1]{%
  \edef\@tempa{#1}%
  \def\@tempb{}%
  \ifx\@tempa\@tempb
	\else \\\underline{Co-Supervisor:}\vspace{0.35in}\\\dots\dots\dots\dots\dots\dots\dots\\{#1}\\
  \fi
}

\newcommand\isjoint[1]{%
  \edef\@tempa{#1}%
  \def\@tempb{}%
  \ifx\@tempa\@tempb
	\else \\\underline{Joint Supervisor:}\vspace{0.35in}\\\dots\dots\dots\dots\dots\dots\dots\\{#1}\\
  \fi
}

\newcommand\isalt[1]{%
  \edef\@tempa{#1}%
  \def\@tempb{}%
  \ifx\@tempa\@tempb
	\else \\\underline{Alternate Supervisor:}\vspace{0.35in}\\\dots\dots\dots\dots\dots\dots\dots\\{#1}\\
  \fi
}

\newcommand\isdefinedsig[1]{%
  \edef\@tempa{#1}%
  \def\@tempb{}%
  \ifx\@tempa\@tempb
	\else \\ \dots\dots\dots\dots\dots\dots\dots\\{#1}\\
  \fi
}
\newcommand\isdefinedspinetitle[1]{%
  \edef\@tempa{#1}%
  \def\@tempb{}%
  \ifx\@tempa\@tempb
	\else (Spine title: #1)\\
  \fi
}
\newcommand\coauthor[1]{%
  \edef\@tempa{#1}%
  \def\@tempb{}%
  \ifx\@tempa\@tempb
	\else \newpage \Large Co-Authorship Statement\normalsize\\\indent\\#1\\
  \fi
}

\newcommand\acknowlege[1]{%
  \edef\@tempa{#1}%
  \def\@tempb{}%
  \ifx\@tempa\@tempb
	\else \newpage \Large Acknowlegements\normalsize\\\indent\\#1\newpage
  \fi
}

\newcommand{\department}{Mathematics}
\newcommand{\degree}{Doctor of Philosophy}
\newcommand{\firstname}{Michael}
\newcommand{\middlename}{L.}
\newcommand{\lastname}{Rogelstad}
\newcommand{\authorname}{{\firstname} {\middlename} {\lastname}}
\newcommand{\titl}{Combinatorial Techniques in the Galois Theory of $p$-Extensions}
\newcommand{\spinetitle}{}
\newcommand{\thesisformat}{Monograph} 
\newcommand{\gyear}{\number\year}
\newcommand{\makecoauthor}{
\Large\begin{center}\textbf{Co-Authorship Statement}\end{center}\normalsize

\vspace{.5 in}
\begin{spacing}{1.5}

\noindent Chapters 3, 4 and 5 of this thesis incorporate material that is the result of joint research with Professor J\'an Min\'a\v{c} and Dr. Nguy\~{\^e}n Duy T\^an.  Chapter 4 is based on the paper \cite{MRT}: \\

Dimensions of Zassenhaus Filtration Subquotients of Some Pro-$p$ Groups, 

\textit{arXiv:1405.6980v2}, Mar 2015,
\\

\noindent which is to appear in the Israel Journal of Mathematics.

\end{spacing} 
}
\newcommand{\makeacknowlege} {
\Large\begin{center}\textbf{Acknowledgements}\end{center}\normalsize
\vspace{.5in}
\begin{spacing}{1.5}
I would like to extend a most sincere thank you to my supervisor, Professor J\'an Min\'a\v{c}, for the opportunity to pursue this research, for many interesting and valuable discussions and for his constant encouragement, enthusiastic assistance and unending patience in seeing this work to completion.  My friend and colleague Dr. Nguy\~{\^e}n Duy T\^an provided a tremendous amount of help along the way and for all of his efforts I am also exceedingly grateful.  This thesis would not have been possible without the kind and generous assistance of these two extraordinary mathematicians.

I very much appreciate the helpful suggestions and guidance provided by supervisory committee members, Professors Dan Christensen and Nicole Lemire, and examiners, Professors Tatyana Barron, Sunil Chebolu, Masoud Khalkhali and Marc Moreno Maza.  A special thank you to Professor Chebolu for making the trip to London, Ontario for my thesis examination and for subsequent discussions regarding the thesis material which I found very useful and encouraging.  I would like to thank Janet Williams for her kind help with various administrative tasks over the past few years.

My thanks are also extended to the faculty of the Department of Mathematics for many excellent courses and seminars.  I owe a particular debt of gratitude to Professor Stuart Rankin and the late Professors Richard Kane and Andr\'e Boivin.  Dedicated and outstanding teachers, they were an early source of inspiration leading to my pursuit of graduate studies in mathematics.  

And last but not least, to my family, for the many sacrifices and unwavering support that made this fascinating journey possible, I will be forever grateful.  
   
\end{spacing}
}

\renewcommand{\maketitle}
{\begin{titlepage}
   \setcounter{page}{1}
   \begin{large}
   \begin{center}
      \mbox{}
      \vfill
      {\MakeUppercase{\titl}}\\
      \isdefinedspinetitle{\spinetitle}
      (Thesis format: \thesisformat)\\
      \vfill
      by \\
      \vfill
      {\authorname}\\
      \vfill
      Graduate Program in {\department}\\
      \vfill
		A thesis submitted in partial fulfilment\\
		of the requirements for the degree of\\
		\degree\\
		\vfill
		The School of Graduate and Postdoctoral Studies\\
		The University of Western Ontario\\
		London, Ontario, Canada\\
		\vfill
      {\copyright} {\authorname} {\gyear}  \\
      \vspace*{.2in}
   \end{center}
   \end{large}
   \end{titlepage}

}

\makeatother

\begin{document}

\begin{preliminary}

\maketitle
\setcounter{page}{2}
\addcontentsline{toc}{chapter}{Abstract}
\Large\begin{center}\textbf{Abstract}\end{center}\normalsize

\vspace{.5 in}
\begin{spacing}{1.5}

A major open problem in current Galois theory is to characterize those profinite groups which appear as absolute Galois groups of various fields.  Obtaining detailed knowledge of the structure of quotients and subgroup filtrations of Galois groups of $p$-extensions is an important step toward a solution.  We illustrate several techniques for counting Galois $p$-extensions of various fields, including pythagorean fields and local fields.  An expression for the number of extensions of a formally real pythagorean field having Galois group the dihedral group of order 8 is developed.  We derive a formula for computing the $\F_p$-dimension of an $n$-th graded piece of the Zassenhaus filtration for various finitely generated pro-$p$ groups, including free pro-$p$ groups, Demushkin groups and their free pro-$p$ products.  Several examples are provided to illustrate the importance of these dimensions in characterizing pro-$p$ Galois groups.  We also show that knowledge of small quotients of pro-$p$ Galois groups can provide information regarding the form of relations among the group generators.

\vfill
\noindent
\textbf{Keywords:} Galois theory, $p$-extension, pro-$p$ group, absolute Galois group, local field, formally real pythagorean field, Zassenhaus filtration, Demushkin group.
\end{spacing}
 
\newpage

\addcontentsline{toc}{chapter}{Co-Authorship Statement}
\makecoauthor 
\newpage
\addcontentsline{toc}{chapter}{Acknowledgements}
\makeacknowlege
\newpage

\tableofcontents\newpage
\newpage
\end{preliminary}


\begin{spacing}{1.5}

\chapter{Introduction}
\label{ch:chapter1}

Nearly 200 years after his untimely death in a duel at the age of 20, the legacy of \'Evariste Galois lives on in the theory that bears his name.  While remarkable progress has certainly been made, a number of open questions remain.  One major problem is, given a field $F$ with separable closure $F_s$, to characterize, among other profinite groups, the absolute Galois group $G_F:=\Gal(F_s/F)$ of $F$.  One means of approaching this question is to study the structure of subgroups and quotients of certain profinite groups.  Such fundamental problems in current Galois theory also have important implications in other areas of mathematics.  For example, knowing how much information Galois groups carry about base fields is closely related to the problem in algebraic geometry of determining how much information about algebraic varieties is contained in the knowledge of fundamental groups. 

Given a prime $p$, one can consider the Galois group $G_F(p)$ of the maximal $p$-extension of $F$.  This is a pro-$p$ group which is the maximal pro-$p$ quotient of $G_F$.  To shed some light on the structure of this group, one can then ask whether, for a given pro-$p$ group $G$, there exists a normal extension $K/F$ with Galois group isomorphic to $G$.  This is the inverse Galois problem and the solution is closely related to properties of the base field $F$.  The next question which naturally arises is to determine the number of such extensions. 

Similarly, the structure of subgroups, particularly certain central subgroup filtrations, of profinite groups has a close connection with Galois theory.  For example, in 1947, I. R. Shafarevich \cite{Sha} showed that $G_F(p)$, for local fields $F$ not containing a primitive $p$-th root of unity, was a free pro-$p$ group simply by determining the cardinality of some of its filtration quotients.  Zassenhaus filtrations of groups were introduced in \cite{Zas} and, in the case of absolute Galois groups, these filtrations and their subquotients have recently been investigated in \cite{CEM,Ef1,Ef2, EM,MT,MTE}.  Other filtrations of absolute Galois groups, such as the descending $p$-central series, are also of interest (see, for example, \cite{La2,EfMin}). 

Our focus is on the Galois theory of $p$-extensions, which are Galois extensions $K/F$ of a base field $F$ whose Galois group is a pro-$p$ group, and our goal here is twofold.  First, we endeavour to illustrate several methods for counting finite $p$-extensions.  We compare these methods to show that, by employing various algebraic tools, one can develop relatively efficient counting techniques.  Our second goal, and main result, is to develop a method for determining the $\F_p$-dimension of subquotients of the Zassenhaus filtration of finitely generated pro-$p$ groups.  We derive an explicit formula for these dimensions in a number of specific cases and point out several examples of their importance in the Galois theory of $p$-extensions.

These various techniques are dependent upon results from a broad range of subjects, including Galois cohomology, the theory of quadratic forms and quaternion algebras, and the general theory of M\"{o}bius functions.  The necessary background is presented in the next chapter.  We begin that chapter with an overview of the theory of profinite groups and discuss several results that will be needed in the study of filtration quotients in chapter \ref{ch:dimensions}.  For example, in section \ref{sec:filtrations} the descending $q$-central series and the Zassenhaus $p$-filtrations are defined and  the connections between filtrations of a finitely generated pro-$p$ group $G$ and the structure of both the completed group algebra $\F_p[[G]]$ of $G$ and the Magnus algebra with coefficients in $\F_p$ are described.  These connections are important in developing the technique for counting $\F_p$-dimensions of Zassenhaus filtration subquotients in section \ref{sec:hilbert series}.

In chapter \ref{ch:extensions} we turn to the problem of counting Galois $p$-extensions of a field $F$.  Beginning with the case $p=2$, we point out the connection between a small quotient of the group $G_F(2)$, called the W-group of $F$, and the number of Galois extensions of $F$ having Galois group isomorphic to the dihedral group $D_4$ of order 8, which are referred to as $D_4$-extensions.  The W-group is the quotient $G_F^{[3]}=G_F/G_F^{(3)}$ in the descending 2-central series $(G_F^{(i)})_{i\geq 1}$ of $G_F$, and is considered a Galois-theoretic analog of the cohomology ring $H^*(G_F,\F_2)$ \cite{CEM,MS2}.  This further suggests that enumerating $p$-extensions and filtration quotient dimensions can play an important role in elucidating the structure of absolute Galois groups. 

We first consider the case of $p$-extensions of local fields.  Section \ref{sec:constructing ext} describes the method of directly constructing $D_4$-extensions of the field of $p$-adic numbers, $\Q_p$, for both $p$ odd and $p=2$, due to H. Naito \cite{Naito}.  In the proof of Proposition \ref{prop:D4Qp}, we illustrate an alternative, group-theoretic approach based on knowledge of the W-group of $\Q_p$ and in section \ref{sec:quat alg Qp} we outline an approach based on the theory of quaternion algebras over $\Q_p$.  

Turning, in section \ref{sec:mobius}, to the more general case of a local field $K$ which is a finite extension of $\Q_p$, we describe an interesting method of counting $p$-extensions of $K$ using M\"{o}bius functions and complex characters, due to M. Yamagishi \cite{yam} which relies on lemma \ref{lem:muG} from the general theory of M\"{o}bius functions.  In example \ref{ex:shaf}, we show that if $K$ does not contain a primitive $p$-th root of unity, this method can be used to obtain an earlier result of Shafarevich \cite{Sha}.  The case in which $K$ is a finite extension of $\Q_2$ of odd degree not containing a primitive 4-th root of unity is shown in detail in example \ref{ex:Q2}.  This example illustrates that the method produces a simple expression for the number of $D_4$-extensions of $K$ which depends only on $n$, but requires detailed knowledge of the classification of Demushkin groups as well as the complex character theory of $D_4$ and all of its subgroups.

In section \ref{sec:cup}, we assume that $K$ is a finite extension of $\Q_p$ containing a primitive $p$-th root of unity and develop a method based on Galois cohomology and the solution of embedding problems to count the number of $\U_3(\F_p)$-extensions of $K$, where $\U_3(\F_p)$ is the group of unipotent three by three matrices over $\F_p$.  This is the `cup product analogue' of a technique using Massey products to compute the number of $\U_4(\F_p)$-extensions of $K$ developed by J. Min\'a\v{c} and N. D. T\^an \cite{MT1}.  Since $D_4\cong \U_3(\F_2)$, this provides another method of enumerating the $D_4$-extensions of certain local fields.

We introduce the theory of formally real pythagorean fields in section \ref{sec:pyth} and use properties of the set of orderings, cup products and quaternion algebras to derive an expression for the number of $D_4$-extensions of a formally real pythagorean SAP field $F$ with finite square class group $\fx/\fxs$.  We also characterize the group $G_F(2)$ for both pythagorean SAP fields and superpythagorean fields.  The structure of these groups is then studied further in chapter \ref{ch:dimensions}.

The main results appear in chapter \ref{ch:dimensions}.  We introduce the Hilbert-Poincar\'e series and define $c_n(G)$ to be the $\F_p$ dimension of the $n$-th graded piece $G_{(n)}/G_{(n+1)}$ of the Zassenhaus filtration of a finitely generated pro-$p$ group $G$.  Using a beautiful theorem of Jennings and Lazard, we develop an explicit formula for $c_n(G)$ for various families groups $G$, including finitely generated free pro-$p$-groups,  Demushkin groups, and free pro-2 products of finitely many copies of the cyclic  group $C_2$ of order 2.  In Proposition \ref{prop:wn}, we point out a relationship between $c_n(G)$ and the $\Z_p$-rank of the $n$-th graded piece of the descending central series of $G$.

Section \ref{sec:free} deals with free pro-$p$ groups.  The Magnus homomorphism introduced in Theorem \ref{thm:magnusiso} allows us, in Lemma \ref{lemHPfree}, to characterize a finitely generated free pro-$p$ group by its Hilbert-Poincar\'e series and in Remark \ref{rmk:fgGalois}, we show that determining finitely generated free pro-$p$ groups within the family of all Galois groups of the maximal $p$-extensions of fields containing a primitive $p$-th root of unity actually requires only the two numbers, $c_1(G)$ and $c_2(G)$.  We observe also that, for a free pro-$p$ group $S$, the numbers $c_n(S)$ determine the minimal number of generators of the Zassenhaus subgroups of $S$ and we give an explicit $\F_p$-basis for $S_{(n)}/S_{(n+1)}$, for each $n$ in terms of Hall commutators.  In Lemma \ref{lem:coefficient} and Corollary~\ref{cor:unipotent}, we again meet the group $\U_n(\F_p)$ and provide an interesting, purely group theoretical result based on the formula for $c_n(S)$.

Following up on the characterization of the Galois groups of maximal $p$-extensions of pythagorean fields given in section \ref{sec:pythGF2}, we study, in section \ref{sec:freeprod}, groups $G$ which are free products of a finite number of cyclic groups of order 2.  We show that each such group $G$ contains a free pro-2 subgroup $H$ of index 2 and in Corollary \ref{cor:quotient} we obtain, for each $n\geq 2$, the interesting relation $H_{(n)} = H\cap G_{(n)}$ using knowledge of the numbers $c_n(G)$ and $c_n(H)$.  In Remarks \ref{rmks:SAP} we observe that $c_1(G_F(2))$ and $c_2(G_F(2))$ are sufficient to determine the group $G_F(2)$ if $F$ is either a pythagorean SAP field or a superpythagorean field.  These examples illustrate the fact that the numbers $c_n(G)$ can be very useful in group theory and Galois theory.     

We conclude chapter \ref{ch:dimensions} by looking at Demushkin groups as well as some other groups.  We observe again that, for a Demushkin group $G$, the numbers $c_n(G)$ determine the minimal number of generators of the Zassenhaus subgroups of $G$.    

Finally in chapter \ref{ch:relations}, we consider relations among the generators of finitely generated pro-$p$ Galois groups and show that knowledge of small quotients of these groups can be useful in determining the form of these relations.

The following is a list of the theorems/propositions/lemmas/corollaries which constitute the main results of this thesis:  3.4.10, 3.4.12, 3.4.15, 4.1.13, 4.1.14, 4.2.1, 4.2.9, 4.2.13, 4.3.3, 4.3.4, 4.3.5, 4.4.1, 5.1.2, 5.2.1.

\chapter{Background}
\label{ch:background}

In this chapter we review some of the basic theory of profinite groups and introduce the tools that will be needed in subsequent chapters.  We outline a connection between central filtrations of profinite groups and filtrations of completed group algebras which will be important in developing a technique for computing filtration subquotient dimensions in chapter \ref{ch:dimensions}.  Galois cohomology as well as the theories of quaternion algebras and quadratic forms lead to interesting methods for counting Galois $p$-extensions.  We review the pertinent background and relevant connections as a prelude to illustrating several combinatorial techniques in chapter \ref{ch:extensions}.  Many problems of enumeration, including those that we study in chapters \ref{ch:extensions} and \ref{ch:dimensions}, are closely related to the theory of M\"obius functions.  The final section of this chapter outlines that general theory.

\section{Profinite Groups}

\begin{defn}
A {\it profinite group} is a compact Hausdorff topological group whose open subgroups form a neighbourhood basis at the identity.
\end{defn}

There are several other equivalent definitions, the most important of which is based on the concept of an {\it inverse} (or {\it projective}) {\it limit}.  We briefly outline this construction.  A {\it directed set} is a non-empty partially ordered set $(\Lambda,\leq)$ such that for every $\lambda,\mu \in \Lambda$ there exists $\nu \in \Lambda$ with $\nu \geq \lambda$ and $\nu \geq \mu$.  An {\it inverse system} of groups over $\Lambda$ is a family of groups $(G_\lambda)_{\lambda \in \Lambda}$ together with homomorphisms $\pi_{\lambda\mu} \colon G_\lambda \to G_\mu$ whenever $\lambda \geq \mu$, satisfying the conditions
\[
\pi_{\lambda\lambda}=\textnormal{Id}_{G_\lambda} \hspace{.1in}\textnormal {and}\hspace{.1in}  \pi_{\lambda\nu}=\pi_{\mu\nu}\pi_{\lambda\mu} \textnormal { whenever } \lambda \geq \mu \geq \nu.
\]   
  The {\it inverse limit}
\[
\varprojlim G_\lambda = \varprojlim (G_\lambda)_{\lambda \in \Lambda}  
\]
is the subgroup of the direct product $\prod_{\lambda \in \Lambda}G_\lambda$ consisting of all elements $(g_\lambda)_{\lambda \in \Lambda}$ such that $\pi_{\lambda\mu}(g_\lambda)=g_\mu$ whenever $\lambda \geq\ \mu$.  An analogous construction can also be applied to other structures such as sets, rings or topological spaces. 

A profinite group can then be defined as a topological group that can be realized as an inverse limit of finite groups endowed with the discrete topology.  It can be shown (see, for example \cite[Proposition 1.3]{DDMS}) that these two definitions are equivalent.

\begin{ex}
\begin{enumerate}

\item Given a group $G$, let $\Lambda$ be the set of all normal subgroups of finite index in $G$, directed by reverse inclusion.  Then the family of quotients $(G/N)_{N\in\Lambda}$ forms an inverse system of finite groups.  The inverse limit
\[
\hat G=\varprojlim(G/N)_{N\in\Lambda}
\]  
is a profinite group, called the {\it profinite completion} of $G$.

\item Profinite groups arise in this way as the Galois groups of algebraic field extensions.  If $K/F$ is a Galois extension, its Galois group
\[
\textnormal{Gal}(K/F)\cong \varprojlim \textnormal{Gal}(M/F)_{M\in\Gamma},
\]
where $\Gamma$ is the set of all finite Galois sub-extensions $M/F$ of $K/F$.  In the case that $K=F_s$ is the separable closure of $F$, then $G_F:=\Gal(F_s/F)$ is the \textit{absolute Galois group} of $F$.

\item If $F$ is a finite field with algebraic closure $\bar F$, then Gal$(\bar F/F)\cong \hat \Z$, the profinite completion of $\Z$.
\end{enumerate}
\end{ex}

\subsection{Pro-$p$ groups}

\begin{defn}
Let $p$ be a prime number.  A {\it pro-p group} is a profinite group in which every open normal subgroup has index equal to some power of $p$.
\end{defn}  

In an analogous way to the case of profinite groups, it can be shown that a topological group $G$ is a pro-$p$ group if and only if $G$ is topologically isomorphic to an inverse limit of finite $p$-groups.

\begin{ex}
\begin{enumerate}

\item Given a group $G$, let $\Lambda$ be the set of all normal subgroups of $G$ whose index is a power of $p$, directed by reverse inclusion.  Then 
\[
\hat G_p=\varprojlim(G/N)_{N\in\Lambda}
\]  
is a pro-$p$ group, called the {\it pro-p completion} of $G$.

\item Historically, the subject began with what is regarded as the prototype of all pro-$p$ groups, the additive group of $p$-adic integers,
\[
\Z_p=\varprojlim (\Z/p^n\Z)_{n\in \N} = \{(x_n)_{n\in \N}\mid x_i \equiv x_j (\textnormal{mod} p^j) \textnormal{ if } i\geq j\}
\]

\item The {\it maximal p-extension} $F(p)$ of a field $F$ is the compositum of all finite Galois sub-extensions $K/F$ of $F_s/F$ with $[K:F]$ a power of $p$.  The group $G_F(p)=$Gal$(F(p)/F)$ is the maximal pro-$p$ quotient of the absolute Galois group $G_F$ of $F$. 
\end{enumerate}
\end{ex}

\begin{defn}
Let $G$ be a pro-$p$ group.  A {\it system of generators} of $G$ is a subset $X$ of $G$ with the following properties:
\begin{enumerate}
\item $G$ is the smallest (closed) subgroup containing $X$;
\item every neighbourhood of the identity in $G$ contains almost all (i.e. all but finitely many) elements of $X$.
\end{enumerate}
\end{defn}

\begin{defn}
A system $X$ of generators of the pro-$p$ group $G$ is called {\it minimal} if no proper subset of $X$ is a system of generators of $G$.  The cardinality of a minimal system of generators of $G$ will be denoted $d(G)$ and is often referred to as the \textit{rank} of $G$.  $G$ is said to be {\it finitely generated} if $d(G)< \infty$.
\end{defn}

\begin{defn}
Let $G$ be a profinite group.  The {\it Frattini subgroup} of $G$ is 
\[
\Phi(G)=\cap\{M\mid M  \textnormal { is a maximal proper open subgroup of }  G\}.
\]
\end{defn}

The Frattini subgroup plays an important role in the study of pro-$p$ groups.  In particular, we have

\begin{thm} [Burnside's Basis Theorem]
Let G be a pro-p group and $X=\{x_i\mid i\in I\}$ a subset of G such that every neighbourhood of $1\in G$ contains almost all elements of X.  Then X is a system of generators of G if and only if $\{x_i\Phi (G)\mid i\in I\}$ generates $G/ \Phi(G)$.
\end{thm}

\begin{proof}
See \cite[Theorem 4.10]{Ko}.
\end{proof}

We will be primarily interested in finitely generated pro-$p$ groups and in this case we have the following useful results. 

\begin{prop}
If G is a pro-p group then G is finitely generated if and only if $\Phi (G)$ is open in G.
\end{prop}

\begin{proof}
See \cite[Proposition 1.14]{DDMS}.
\end{proof}

\begin{thm}[Serre]
If G is a finitely generated pro-p group then every subgroup of finite index in G is open.
\end{thm}

\begin{proof}
See \cite[Theorem 1.17]{DDMS}.
\end{proof}

\begin{cor}
If G is a finitely generated pro-p group then $\Phi (G)=G^p[G,G]$, where $[G,G]$ is the subgroup generated by commutators $g^{-1}h^{-1}gh$ with $g, h\in G$, and $G^p=<g^p \mid g\in G>$.
\end{cor}

\begin{proof}
See \cite[Corollary 1.20]{DDMS}.
\end{proof}

Note that $G/G^p[G,G]$ is an $\F_p$-vector space, so if $G$ is a finitely generated pro-$p$ group, then $d(G)=\dim_{\F_p}(G/\Phi(G))$.

\begin{cor}
\label{cor:finitegen}
The topology of a finitely generated pro-p group is determined by its group structure.
\end{cor}

\begin{proof}
See \cite[Corollary 1.21]{DDMS}.
\end{proof}

An important method of describing a pro-$p$ group is via a presentation by generators and relations.
 
\begin{defn}
Let $I$ be an index set and let $S_I$ be the free discrete group on the generators $\{s_i\mid i \in I\}$.  Let $\fN$ be the set of normal subgroups $N$ of $S_I$ such that:
\begin{enumerate}
\item $S_I/N$ is a finite $p$-group;
\item $N$ contains almost all elements of $\{s_i\mid i \in I\}$.
\end{enumerate}  

Then $\{S_I/N \mid N\in \fN\}$ is an inverse system and $S(I)=\varprojlim(S_I/N)_{N\in\fN}$ is a pro-$p$ group called the {\it free pro-p group with system of generators} $\{s_i\mid i \in I\}$.

When $I=\{1,\ldots, n \}$ one often writes $S(n)$ instead of $S(I)$ and refers to $S(n)$ as the {\it free pro-p group of rank n}. 
\end{defn}

\begin{defn}
Let G be a pro-$p$ group.  An exact sequence
\[
\begin{tikzcd}
1 \arrow{r} &R \arrow{r} &S \arrow{r}{\varphi} &G \arrow{r} &1,
\end{tikzcd}
\]
where $S$ is a free pro-$p$ group with system of generators $\{s_i\mid i \in I\}$ is called a {\it presentation} of $G$ by $S$.  

If $\{\varphi(s_i)\mid i \in I\}$ is a minimal system of generators of $G$, then the presentation is called {\it minimal}.

A subset $E\subseteq R$ is called a {\it system of relations} of $G$ if
\begin{enumerate}
\item $R$ is the smallest closed normal subgroup of $S$ containing $E$;
\item every open normal subgroup of $R$ contains almost all elements of $E$.
\end{enumerate}
We say $E$ is {\it minimal} if no subset of $E$ is a system of relations for $G$. 
\end{defn}

\begin{ex}
Demushkin groups (see Definition \ref{def:Demushkin}) are an interesting example of finitely generated pro-$p$ groups having only a single relation among a minimal system of generators.  These groups play an important role in the Galois theory of local fields and we will have much more to say about them in subsequent chapters.
\end{ex}

\subsection{Completed group algebras}

In order to study the structure of pro-$p$ groups in greater detail, we introduce two related objects; the completed group algebra of a pro-$p$ group and the Magnus algebra.

Let $ \varLambda $ be a compact commutative ring with identity, let $G$ be a profinite group and let $\fN_G$ be the set of all open normal subgroups of $G$.  For $N, N' \in \fN_G$ with $N\supseteq N'$ the natural map
\[
G/N'\longrightarrow G/N
\]
induces an epimorphism
\[
\varLambda[G/N']\longrightarrow \varLambda[G/N]
\]
of group algebras.  These maps define an inverse system $\{ \varLambda[G/N] \mid N\in \fN_G \}$ of compact rings.

\begin{defn}
The {\it completed group algebra} $\varLambda [[G]]$ of the profinite group $G$ over the compact ring $\varLambda$ is the inverse limit of the system $\{ \varLambda[G/N] \mid N\in \fN_G \}$.
\end{defn}

By the map
\[
g\longmapsto \prod_{N \in \fN_G} gN,
\]
$G$ embeds into $\varLambda [[G]]$ and the subring $\varLambda [G]$, which is given the subspace topology, is dense in $\varLambda [[G]]$.

\begin{thm}
\begin{enumerate}
\item[(i)] Let A be a compact $\varLambda$-algebra.  Every morphism $\phi \colon G \to A^\times$ from G into the group of units $A^\times$ of A can be extended uniquely to a morphism $\varLambda [[G]] \to A$.
\item[(ii)] Let $\phi \colon G \to G'$ be a morphism of profinite groups with kernel N.  The kernel of the induced morphism $\varLambda[[G]] \to \varLambda[[G']]$ is the closed ideal I(N) generated by $\{h-1\mid h \in N\}$.
\end{enumerate}
\end{thm}

\begin{proof}
See \cite[Theorem 7.2]{Ko}.
\end{proof}

\begin{thm}
Let $\varLambda$ be a finite ring and G a profinite group.  The system $\{I(N)\mid N\in\fN_G \}$ is a neighbourhood basis at $0\in \varLambda[[G]]$.
\end{thm}

\begin{proof}
See \cite[Theorem 7.3]{Ko}.
\end{proof}

\begin{defn}
Let K be a commutative ring with identity, let I be an index set and let $U=\{u_i \mid i\in I \}$.  The {\it Magnus algebra} $K \langle\langle U \rangle\rangle$ is the associative algebra of formal power series in the non-commuting indeterminates $u_i$, $i\in I$, with coefficients in $K$.
\end{defn}

Let $\sI$ be the ideal of $K \langle\langle U \rangle\rangle$ consisting of all formal power series having constant term 0.  If $u\in \sI$, then $1+u$ is invertible and
\[
(1+u)^{-1} = 1-u+u^2-u^3+\ldots + (-1)^nu^n+\ldots
\] 
The invertible elements $1+u_i$ generate a subgroup of the group of units of $K \langle\langle U \rangle\rangle$.  This group is the image of the free group $S_I$ under the homomorphism $\psi \colon S_I \to K \langle\langle U \rangle\rangle^\times$ given by $s_i \mapsto 1+u_i$ and is referred to as the {\it Magnus group}. 

\begin{lem}
The map $\psi$ is injective.
\end{lem}

\begin{proof}
See \cite[Lemma 4.4]{Ko}.
\end{proof}

Identifying $S_I$ with its image in $K \langle\langle U \rangle\rangle$, we have $s_i =1+u_i$ and if $s\in S_I, s\neq 1$, we have $s=1+u$ with $u\in \sI^n, u\notin \sI^{n+1}$, for some $n\geq 1$.  Magnus called $n$ the {\it dimension} of $s$. 

\subsection{Filtrations}
\label{sec:filtrations}

In attempting to elucidate the structure of Galois groups, the study of certain filtrations can be very useful.  We are particularly interested in filtrations of pro-$p$ groups and the corresponding filtrations on the completed group algebra and the Magnus algebra.

Let $G$ be a group.  A {\it central filtration} (or {\it central series} or {\it central sequence}) of $G$ is a sequence $(G_n)_{n\geq 1}$ of subgroups of $G$ such that
\[
G_1=G,\qquad G_{i+1}\leq G_i,\qquad [G_i,G_j]\leq G_{i+j},
\]  
where $[H,K]$ denotes the subgroup of $G$ generated by the commutators $[h,k]=h^{-1}k^{-1}hk$ with $h\in H$, $k\in K$.  The name arises from the fact that for any such filtration and any $i\geq 1$, $G_i \trianglelefteq G$ and $G_i/G_{i+1}$ lies in the centre of $G/G_{i+1}$.

Let $G$ be a profinite group and let $q$ be either a $p$-power or 0.  The {\it descending} (or {\it lower}) {\it q-central series} $(G^{(n,q)})_{n\geq 1}$ of $G$ is defined inductively by
\[
G^{(1,q)} =G,\qquad G^{(i+1,q)}=(G^{(i,q)})^q[G^{(i,q)},G],\hspace{.1in} i=1,2,\ldots,
\]
where, given closed subgroups $H$ and $K$ of $G$, $[H,K]$ (respectively $H^q$, $HK$) denotes the closed subgroup topologically generated by all commutators $[h,k]$ (respectively $q$-th powers, products $hk$) with $h\in H$, $k\in K$.  Note that $G^{(i,q)}$ is normal in $G$.  For $i\geq 1$, let $G^{[i,q]}=G/G^{(i,q)}$.  If $q$ is understood, we generally abbreviate
\[
G^{(i)}=G^{(i,q)},\qquad G^{[i]}=G^{[i,q]}.
\]

When $q=0$ the series $G^{(i,0)}$ is called the {\it descending} (or {\it lower}) {\it central series} of $G$.  In this case we will adopt the commonly used notation
\[
G_i=G^{(i,0)}.
\]

We can define another central filtration $(G_{(n)})_{n\geq 1}$ on the profinite group $G$ by
\[
G_{(1)}=G,\qquad G_{(n)}=G_{(\lceil n/p\rceil)}^p \prod_{i+j=n} [G_{(i)},G_{(j)}],\hspace{.1in} n=2,3,\ldots,
\]
where $p$ is a fixed prime and $\lceil n/p\rceil$ is the least integer $r$ such that $pr\geq n$.  Then $[G_{(i)}, G_{(j)}]\leq G_{(i+j)}$ and $G_{(i)}^p\leq G_{(pi)}$ for all $i,j\geq 1$ and $(G_{(n)})_{n=1,2,\ldots}$ is the fastest descending sequence of closed subgroups of $G$ having these properties.  This sequence is the {\it Zassenhaus p-filtration} of $G$.  For quotients we again adopt the notation $G_{[i]}=G/G_{(i)}$.  We will usually be considering the Zassenhaus $p$-filtration of a pro-$p$ group $G$ and will often refer to this simply as the \textit{Zassenhaus filtration} of $G$.

We now turn our attention to the special case in which $G$ is a finitely generated pro-$p$ group and look at the connections between the filtrations of $G$ and the structure of both the completed group algebra $\F_p[[G]]$ of $G$ and the Magnus algebra with coefficients in $\F_p$.  

\begin{defn}
The {\it augmentation ideal} $I(G)$ of $\F_p[[G]]$ is the closed two-sided ideal generated by the elements $g-1$, $g\in G$ and $I^n(G)$ is the closure of the $n$-th power of $I(G)$ in $\F_p[[G]]$.
\end{defn}

\begin{thm}
Let G be a finitely generated pro-p group.  Then $\{I^n(G)\mid n=1,2,\ldots\}$ is a neighbourhood basis at $0\in \F_p[[G]]$.
\end{thm}

\begin{proof}
See \cite[Theorem 7.8]{Ko}.
\end{proof}

There is a close relationship between the filtration $I^n(G)$ of $\F_p[[G]]$ and the Zassenhaus $p$-filtration $G_{(n)}$ of $G$.  

\begin{thm}
Let $G$ be a finitely generated pro-$p$ group and $I(G)$ the augmentation ideal of $\F_p[[G]]$.  Then
\[
G_{(n)}=(1+I^n(G))\cap G,\qquad n\geq 1
\]
\end{thm}

\begin{proof}
Since the topology of a finitely generated pro-$p$ group is determined by its group structure, the result follows from \cite[Theorem 12.9]{DDMS}. 
\end{proof}

So for each $n\geq 1$, $G_{(n)}$ is the kernel of the natural homomorphism of $G$ into the group of units of $\F_p[[G]]/I^n(G)$; that is, 
\[
G_{(n)}=\{g\mid g-1\in I^n(G) \}.
\]

\begin{rmk}
In view of this theorem, the Zassenhaus filtration has also been referred to as the dimension series, with the subgroups $G_{(n)}$ being called  the dimension subgroups in characteristic $p$. 
\end{rmk}

Now let $U=\{u_1,\ldots,u_d \}$ and consider the Magnus algebra $\F_p\langle\langle U \rangle\rangle$.  Let $\sI^n$ denote the ideal of all power series in $\F_p\langle\langle U \rangle\rangle$ whose homogeneous components have degree at least $n$.  Then $\sI^n/\sI^{n+1}$ is the submodule of $\F_p\langle\langle U \rangle\rangle$ generated by the monomials of degree $n$. Also, $\{\sI^n\mid n=1,2,\ldots \}$ is a basis of open neighbourhoods of $0\in \F_p\langle\langle U \rangle\rangle$ and with this topology, $\F_p\langle\langle U \rangle\rangle$ is the direct product of its homogeneous components, $\sI^n/\sI^{n+1}$, hence compact.

Let $S=S(d)$ be the free pro-$p$ group with system of generators $\{ s_1,\ldots, s_d\}$.  We have the Magnus embedding $S\hookrightarrow \F_p\langle\langle U \rangle\rangle ^\times$ given by $s_i\mapsto 1+u_i$, $i=1,\ldots,d$.  Furthermore, we have

\begin{thm}
\label{thm:magnusiso}
The map 
\[
s_i \longmapsto 1+u_i,\qquad i=1,\ldots,d,
\]
can be extended to an isomorphism $\F_p[[S]]\cong \F_p\langle\langle U \rangle\rangle $.
\end{thm}

\begin{proof}
See \cite[Theorem 7.16]{Ko}.
\end{proof}

Identifying $\F_p\langle\langle U \rangle\rangle$ and $\F_p[[S]]$, we then have

\begin{thm}
Let G be a finitely generated pro-p group with presentation
\[
\begin{tikzcd}
1 \arrow{r} &R \arrow{r} &S \arrow{r} &G \arrow{r} &1.
\end{tikzcd}
\]
Let $\{ r_i\mid i\in I \}$ be a system of relations with respect to this presentation.  Then the kernel of the induced map $\F_p[[S]]\rightarrow \F_p[[G]]$ is generated as an ideal of $\F_p[[S]]$ by $\{r_i-1\mid i\in I \}$.
\end{thm}

\begin{proof}
See \cite[Theorem 7.17]{Ko}.
\end{proof}

Hence we have the following commutative diagram:
\begin{center}
\begin{tikzcd}
S \arrow{d} \arrow{r} &\F_p\langle\langle U \rangle\rangle \arrow{dd}\\ 
G\arrow{d} \\
\F_p[[G]] \arrow{r} &\F_p\langle\langle U \rangle\rangle /J  
\end{tikzcd}
\end{center}
where $J$ is the ideal of $\F_p\langle\langle U \rangle\rangle$ generated by the set $\{r_i\mid i\in I \}$ and all other maps are defined in the obvious way.

\section{Galois Cohomology}
\label{sec:cohom}

Given a field $F$ and a Galois extension $K/F$ with Galois group $G=\Gal(K/F)$, the cohomology groups $H^n(G,A)$, where $A$ is a $G$-module, often contain important arithmetic information.  The study of these groups is referred to as Galois cohomology.  Since Galois groups are profinite groups, we begin by looking at the more general case of the cohomology of profinite groups.  Subsequently, we look in more detail at the groups $H^n(G,A)$ for $n=0,1,2$ and consider some specific choices of $G$ and $A$ which will be of interest in the chapters that follow.

\subsection{Cohomology of profinite groups} 

Let $G$ be a profinite group and $A$ a discrete $G$-module.  An (inhomogeneous) \textit{n-cochain} of $G$ with coefficients in $A$ is a continuous function $y:G^n\to A$.  The set of all such functions is an abelian group denoted $C^n(G,A)$ and these groups form a complex  
\[
\begin{tikzcd}
C^0(G,A) \arrow{r}{d^1} &C^1(G,A) \arrow{r}{d^2} &C^2(G,A) \arrow{r} &\cdots,
\end{tikzcd}
\]
where the coboundary operator $d^{n+1}:C^n(G,A)\to C^{n+1}(G,A)$ is given by
\[
\begin{aligned}
(d^{n+1}y)(g_1,\ldots ,g_{n+1})
&=g_1y(g_2,\ldots ,g_{n+1})\\
&+\sum_{i=1}^{n} (-1)^iy(g_1,\ldots ,g_{i-1},g_ig_{i+1},g_{i+2},\ldots ,g_{n+1})\\
&+(-1)^{n+1}y(g_1,\ldots ,g_n)
\end{aligned}
\]
for $y\in C^n(G,A)$ and $n\geq 0$.  The group of \textit{n-cocycles} is  
\[
\begin{tikzcd}
Z^n(G,A):=\ker (C^n(G,A)\arrow{r}{d^{n+1}} &C^{n+1}(G,A)).
\end{tikzcd}
\]
The group of \textit{n-coboundaries} is the subgroup of $Z^n(G,A)$ defined by
\[
\begin{tikzcd}
B^n(G,A):=\im (C^{n-1}(G,A)\arrow{r}{d^{n}} &C^n(G,A))
\end{tikzcd}
\]
with $B^0(G,A):=0$.  Then, for $n\geq 0$, the quotient group $H^n(G,A)=Z^n(G,A)/B^n(G,A)$ is the \textit{n-dimensional cohomology group} of $G$ with coefficients in $A$.

If $f:A\to B$ is a $G$-module homomorphism, then we have the induced homomorphism
\[
f:C^n(G,A)\to C^n(G,B),\quad x(g_1,\ldots,g_n)\mapsto fx(g_1,\ldots,g_n)
\]
and the commutative diagram
\[
\begin{tikzcd}
\cdots \arrow{r} &C^n(G,A) \arrow{r}{d^{n+1}} \arrow{d}{f}  &C^{n+1}(G,A) \arrow{r} \arrow{d}{f} &\cdots\\
\cdots \arrow{r} &C^n(G,B) \arrow{r}{d^{n+1}} &C^{n+1}(G,B) \arrow{r} &\cdots.
\end{tikzcd}
\]
So $f:A\to B$ induces a homomorphism $f:C^\bullet (G,A)\to C^\bullet (G,B) $ of complexes and hence we obtain homomorphisms $ f:H^n(G,A)\to H^n(G,B) $.  

In addition, using the Snake lemma of homological algebra, one can show (see for example {\cite[\S 1.3]{NSW}}) that every exact sequence $0\to A\to B\to C\to 0$ of $G$-modules gives rise to a canonical \textit{connecting homomorphism}
\[
\delta :H^n(G,C)\to H^{n+1}(G,A)
\]  
and we obtain the \textit{long exact cohomology sequence} 
\[
\begin{tikzcd}
0\arrow{r} &A^G \arrow{r} &B^G \arrow{r} &C^G \arrow{r}{\delta} &H^1(G,A) \arrow{r} &\cdots\\ \cdots \arrow{r} &H^n(G,A) \arrow{r} &H^n(G,B) \arrow{r} &H^n(G,C) \arrow{r}{\delta} &H^{n+1}(G,A) \arrow{r} &\cdots.
\end{tikzcd}
\]

There are also several important maps between cohomology groups involving a ``change of groups'' rather than a ``change of modules''.  Given two profinite groups $G$ and $G'$, a $G$-module $A$, a $G'$-module $A'$, we say that two homomorphisms
\[
f:A\to A',\quad \varphi:G'\to G
\]
are a \textit{compatible pair} if $f(\varphi (g')a)=g'f(a)$ for $g'\in G',\ a\in A$.  From such a pair of homomorphisms we obtain a homomorphism
\[
C^n(G,A)\to C^n(G',A'),\quad a\mapsto f\circ a\circ \varphi.
\]
This commutes with $d$ and hence induces a homomorphism of cohomology groups
\[
H^n(G,A)\to H^n(G',A').
\]

\begin{ex}
Let $H$ be a closed normal subgroup of $G$ and $A$ a $G$-module.  Then $A^H$ is a $G/H$-module.  The projection and injection
\[
G\surj G/H,\quad A^H\inj A
\]
form a compatible pair of homomorphisms, which induce the \textit{inflation} homomorphism
\[
inf:H^n(G/H,A^H)\to H^n(G,A), 
\]
given by
\[
(inf\ x)(g_1,\ldots,g_n)=x(\bar{g_1},\ldots,\bar{g_n}),\quad g_i\in G,
\]
where the image of $g\in G$ in $G/H$ is denoted $\bar{g}$.
\end{ex}
\begin{ex}
For an arbitrary closed subgroup $H$ of $G$, the $G$-module $A$ is also an $H$-module.  The compatible homomorphisms
\[
H\inj G,\quad id:A\to A
\]
induce the \textit{restriction} homomorphism
\[
res:H^n(G,A)\to H^n(H,A), 
\]
given by
\[
(res\ x)(h_1,\ldots,h_n)=x(h_1,\ldots,h_n),\quad h_i\in H.
\]
\end{ex}

If $H$ is an open subgroup of $G$, then in addition to the restriction, we have a map in the opposite direction called the \textit{corestriction}
\[
cor:H^n(H,A)\to H^n(G,A).
\]
This is a kind of norm map and for $n=0$, is the usual norm map
\[
N_{G/H}:A^H\to A^G,\quad a\mapsto \sum_{\sigma \in G/H} \sigma a.
\]
One of the main properties of the corestriction map is that $cor\circ res =[G:H]$.

We also have a bilinear map defined on the cohomology groups of $G$ which plays an important role in Galois cohomology.  It arises as follows.  Let $A,\ B,\ C$ be $G$-modules and suppose there exists a continuous bilinear map $A\times B\to C$, $(a,b)\mapsto a\cdot b$, such that $g(a\cdot b)=(ga)\cdot (gb)$ for $g\in G,\ a\in A,\ b\in B$.  For any pair $p,\ q\geq 0$, we can define a bilinear map, called the \textit{(cochain) cup product}
\[
\cup :C^p(G,A)\times C^q(G,B)\to C^{p+q}(G,C)
\] 
by
\[
(x\cup y) (g_1,\ldots ,g_p,h_1,\ldots,h_q)=x(g_1,\ldots,g_p)\cdot g_1g_2\ldots g_py(h_1,\ldots,h_q).
\]
For this map, we have the formula
\[
d(x\cup y)=dx\cup y + (-1)^px\cup dy.
\]
It follows that if $x$ and $y$ are cocycles then $x\cup y$ is a cocycle, and if one of the cochains $x$ and $y$ is a coboundary and the other a cocycle, then $x\cup y$ is a coboundary.  Hence we obtain a cup product on cohomology
\[
\cup :H^p(G,A)\times H^q(G,B)\to H^{p+q}(G,C).
\]
 
The cohomology groups of a profinite group $G$ can be built up from those of the finite quotient groups of $G$. Let $U$, $V$ run through the open normal subgroups of $G$.  If $V\subseteq U$, the projections
\[
\begin{tikzcd}
G^n \arrow{r} & (G/V)^n \arrow{r} &(G/U)^n
\end{tikzcd}
\]
induce homomorphisms
\[
\begin{tikzcd}
C^n(G/U,A^U) \arrow{r} &C^n(G/V,A^V) \arrow{r} &C^n(G,A),
\end{tikzcd}
\]
which commute with the operators $d^n$.  Hence, we obtain homomorphisms
\[
\begin{tikzcd}
H^n(G/U,A^U) \arrow{r} &H^n(G/V,A^V) \arrow{r} &H^n(G,A).
\end{tikzcd}
\]
The groups $ H^n(G/U,A^U) $ form a direct system giving a canonical homomorphism 
\[
\varinjlim _{U} H^n(G/U,A^U)\to H^n(G,A).
\]

\begin{prop}
\label{prop:finite}
The above homomorphism is an isomorphism:
\[
H^n(G,A) \cong \varinjlim _{U} H^n(G/U,A^U). 
\]
\end{prop}

\begin{proof}
See \cite[Proposition 1.2.5]{NSW}.
\end{proof}

Hence, in what follows, we may assume that $G$ is finite.  We will be interested primarily in the cases $n=0,1,2$.

\subsection{The lower dimensional cohomology groups}

\noindent \textbf{The group $H^0(G,A)$:}  We identify $C^0(G,A)$ with $A$ via the natural isomorphism $y\mapsto y(1)$.  Then, for $a\in A$, $(d^1a)(g)=ga-a$, so $H^0(G,A)=A^G$, the set of elements of $A$ left invariant by $G$.\\

\noindent \textbf{The group $H^1(G,A)$:}  A 1-cocycle is a continuous function $y:G\to A$ such that
\[
y(gh)=y(g)+gy(h) \quad \textnormal{for all} \ g,h\in G.
\]                
Such a function is also called a \textit{crossed homomorphism}.  It is a coboundary if there exists an element $a\in A$ such that $y(g)=ga-a$ for all $g\in G$.  A crossed homomorphism of this type is called \textit{principal}.  

Let $F$ be a field and $K/F$ a Galois extension with $G=\Gal (K/F)$.  Both the additive group $K^+$ and the multiplicative group $\kx$ are $G$-modules.  We have
\begin{thm}[Hilbert's Satz 90]
$H^1(G,\kx)=1$
\end{thm}
\begin{proof}
By Proposition~\ref{prop:finite}, we may assume that $K/F$ is finite.  Let $a:G\to \kx$ be a 1-cocycle.  Since the automorphisms $\sigma \in G$ are linearly independent over $K$, there exists $c\in \kx$ such that 
\[
b=\sum_{\sigma \in G} a(\sigma)\sigma c \neq 0.
\]
Then, for $\tau \in G$, we have
\[
\tau b=\sum_{\sigma \in G} \tau (a(\sigma))\tau \sigma c = \sum_{\sigma \in G} a(\tau)^{-1} a(\tau \sigma)\tau \sigma c = a(\tau)^{-1}b.
\]
Thus $ a(\tau)=b\tau (b)^{-1} $, so $a$ is a 1-coboundary.
\end{proof}
This in turn leads to the following important result, called \textit{Kummer theory}.  Choose $n\in \N$ prime to the characteristic of $F$.  Denote the group of $n$-th roots of unity in the separable closure $F_s$ of $F$ by $ \mu_n $ and the absolute Galois group $\Gal (F_s/F)$ of $F$ by $G_F$.  From the exact sequence of $G_F$-modules
\[
\begin{tikzcd}
1 \arrow{r} &\mu_n \arrow{r} &F_s^\times \arrow{r}{a\mapsto a^n} &F_s^\times \arrow{r} &1.
\end{tikzcd}
\]
we obtain, by Hilbert's Satz 90, the exact sequence
\[
\begin{tikzcd}
H^0(G_F,F_s^\times)=\fx \arrow{r}{a\mapsto a^n} &\fx \arrow{r}{\delta} &H^1(G_F,\mu_n) \arrow{r} &H^1(G_F,F_s^\times)=1,
\end{tikzcd}
\]
and hence
\[
H^1(G_F,\mu_n)\cong \fx/(\fx)^n.
\]
Note that if $p$ is a prime such that $F$ contains a primitive $p$-th root of unity $\zeta_p$, then $\mu_p$ is a trivial $G_F$-module and we have
\[
\fx/(\fx)^p\cong H^1(G_F,\F_p)=\Hom(G_F,\F_p),
\]
where we consider $\F_p$ to be a $G_F$-module with trivial action.  For each $a\in \fx$, we have an element $\chi_a\in H^1(G_F,\F_p)$ defined by $\sigma(\sqrt[p]{a})=\zeta_p^{\chi_a(\sigma)}\sqrt[p]{a}$, for all $\sigma\in G_F$.

We obtain another important result by considering the case in which $G$ is a pro-$p$ group and $\F_p$ is a $G$-module with trivial $G$ action.  $H^1(G, \F_p)$ is then the $\F_p$-vector space of all continuous homomorphisms of $G$ into the discrete group $\F_p$.  Since each such homomorphism vanishes on $G^p[G,G]$, $H^1(G, \F_p)\cong H^1(G/G^p[G,G], \F_p)=\Hom (G/G^p[G,G], \F_p)$.  This implies that $H^1(G, \F_p)$ and $G/G^p[G,G]$ are dual to each other.  By Burnside's Basis Theorem, if $\dim_{\F_p} H^1(G, \F_p) =n<\infty$, $G$ is a finitely generated pro-$p$ group with $n$ as the minimal number of generators.\\

\noindent \textbf{The group $H^2(G,A)$:}  A 2-cocycle is a continuous function $x:G\times G\to A$ such that $d^3x=0$, that is
\[
x(gh,k)+x(g,h)=x(g,hk)+gx(h,k) \quad \textnormal{for all} \ g,h,k\in G.
\]
Such a function is a 2-coboundary if
\[
x(g,h)=y(g)-y(gh)+gy(h)
\]
for a 1-cochain $y:G\to A$.

The 2-cocycles occur in connection with group extensions.  Let
\[
\begin{tikzcd}
1 \arrow{r} &A \arrow{r}{\iota} &E \arrow{r}{\pi} &G \arrow{r} &1.
\end{tikzcd}
\]
be a group extension with abelian kernel $A$ such that $\iota (\sigma a)=\hat{\sigma} \iota (a) \hat{\sigma}^{-1}$ for $\sigma \in G$ and $a\in A$, where $\hat{\sigma}\in E$ is a pre-image of $\sigma$.  We refer to this as an extension of $G$ by the $G$-module $A$.  An extension of $G$ by a trivial $G$-module is called \textit{central}.  Two extensions
\[
\begin{tikzcd}
1 \arrow{r} &A \arrow{r} &E \arrow{r} &G \arrow{r} &1.
\end{tikzcd}
\] \
and
\[
\begin{tikzcd}
1 \arrow{r} &A \arrow{r} &F \arrow{r} &G \arrow{r} &1.
\end{tikzcd}
\]
are said to be \textit{equivalent} if there exists a homomorphism $\varphi :E\to F$ making the diagram
\[
\begin{tikzcd}
1 \arrow{r} &A \arrow{r} \arrow{d}{id} &E \arrow{r} \arrow{d}{\varphi} &G \arrow{d}{id} \arrow{r} &1\\
1 \arrow{r} &A \arrow{r} &F \arrow{r} &G \arrow{r} &1
\end{tikzcd}
\]
commute.  In this case $\varphi $ is an isomorphism and both extensions induce the same $G$-module structure on $A$.

If $s:G\to E$ is a section, we obtain a map $c:G\times G \to A$ by 
\[
s_\sigma s_\tau = \iota (c_{\sigma ,\tau})s_{\sigma\tau}
\]
for $\sigma,\tau \in G$, where, for convenience of notation, we write $s_\sigma$ for $s(\sigma)$ and similarly for $c$.  

Then, by associativity in $E$, we have
\[
c_{\rho\sigma,\tau}+ c_{\rho,\sigma} = c_{\rho,\sigma\tau} + \rho c_{\sigma,\tau}  \quad \textnormal{for all} \ \rho,\sigma,\tau \in G.
\]
Such a map $c$ is called a \textit{factor system}.  The set of factor systems forms an abelian group $\sZ^2(G,A)$ under point-wise addition.

Let $a:G\to A$ be an arbitrary function.  The map $(\sigma ,\tau) \mapsto a_\sigma + \sigma a_\tau - a_{\sigma\tau}$ is a factor system, referred to as \textit{split}, and the map $a$ is called a set of \textit{splitting factors} for the factor system. The split factor systems form a subgroup $\sB^2(G,A)$ of $\sZ^2(G,A)$ and 
\[
H^2(G,A)\cong \sZ^2(G,A)/\sB^2(G,A).
\]

Given another section $t:G\to E$, we must have $t_\sigma = \iota (a_\sigma)s_\sigma$ for some map $a:G\to A$, and the factor system given by $t$ is 
\[
(\sigma ,\tau)\mapsto c_{\sigma,\tau} + (a_\sigma + \sigma a_\tau -a_{\sigma\tau}).
\]
Hence, the cohomology class of $c$ is well defined and is referred to as the cohomology class of the extension
\[
\begin{tikzcd}
1 \arrow{r} &A \arrow{r}{\iota} &E \arrow{r}{\pi} &G \arrow{r} &1.
\end{tikzcd}
\]  
We then have

\begin{thm}
Two extensions of G by the G-module A are equivalent if and only if they have the same cohomology class.  Furthermore, for $\gamma \in H^2(G,A)$ there exists an extension with cohomology class $\gamma$. 
\end{thm}

\begin{proof}
See \cite[Theorem 2.3.1]{Ledet}.
\end{proof}

\begin{ex}
The extension with cohomology class $0\in H^2(G,A)$ is
\[
\begin{tikzcd}
1 \arrow{r} &A \arrow{r}{a\mapsto (a,1)} &A\rtimes G \arrow{r}{(a,\sigma)\mapsto \sigma} &G \arrow{r} &1,
\end{tikzcd}
\]
where $A\rtimes G$ is the \textit{semi-direct product}, i.e. the set $A\times G$ equipped with the composition $(a,\sigma)(b,\tau)=(a+\sigma b, \sigma\tau)$.
\end{ex}

\begin{ex}
In the case where $G=C_n$ is a cyclic group of order $n$ generated by $\sigma$, we can choose the section $G\to E$ as $\sigma^i \mapsto s^i,\ i=0,1,\ldots,n-1$, where $s\in E$ is a pre-image of $\sigma$.  Then $s^n=\iota(a)$ for some $a\in A$ and since $\iota(\sigma a)=s\iota(a)s^{-1}$, we have $a\in A^{C_n}$.  So every cohomology class in $H^2(C_n,A)$ is represented by a factor system of the form
\[
\begin{aligned}
c_{\sigma^i,\sigma^j}&= 
\begin{cases}
0,\quad i+j<n\\
a,\quad i+j\geq n 
\end{cases}\\
\end{aligned}
\]
for $i,j\in \{0,1,\ldots,n-1\}$, where $a\in A^{C_n}$.  Conversely, for any such $a$ this is a factor system.  A set of splitting factors for $c$ is given by $a_1=0,\ a_\sigma=b,\ a_{\sigma^2}=b+\sigma b,\ldots,a_{\sigma^{n-1}}=b+\sigma b+\cdots +\sigma^{n-2}b$, where $b\in A$ with $\Tr_{C_n}(b)=a$.  Hence
\[
H^2(C_n,A)\cong A^{C_n}/\Tr_{C_n}A.
\]
\end{ex}

We obtain another important result when we consider the case in which $G$ is a finitely generated pro-$p$ group.  Suppose $\{x_1,\dots,x_n\}$ is a minimal system of generators of $G$.  Then $G$ has a minimal presentation
\[
\begin{tikzcd}
1 \arrow{r} &R \arrow{r} &S(n) \arrow{r} &G \arrow{r} &1,
\end{tikzcd}
\]
where $S(n)$ is the free pro-$p$ group of rank $n$.  The \textit{rank} of the (closed normal) subgroup $R$ is the number of relations between the $x_i$'s.  

\begin{prop}
The following two conditions are equivalent:
\begin{enumerate}
\item[(i)] The subgroup $R$ is of finite rank (as a closed normal subgroup of $S(n)$
\item[(ii)] $H^2(G,\F_p)$ is of finite dimension
\end{enumerate}
If these conditions are satisfied, one has the equality
\[
r=n-h_1+h_2,
\]
where $r$ is the rank of the normal subgroup $R$ and $h_i=\dim_{\F_p}H^i(G,\F_p)$.
\end{prop} 

\begin{proof}
See \cite[Proposition 27]{Se2}
\end{proof}

Since $\{x_1,\dots,x_n\}$ is a minimal system of generators of $G$, $\dim_{\F_p}H^1(G,\F_p)=n$.  Hence $r=\dim_{\F_p}H^2(G,\F_p)$ is the minimal number of relations between the generators of $G$.

\section{Central Simple Algebras and the Brauer Group}
\label{sec:csa}
 
\begin{defn}
A \textit{central simple algebra} (CSA) $A/F$ over a field $F$ is a finite dimensional $F$-algebra having  center $F$ and no proper nontrivial two-sided ideals.  We say that an extension field $K/F$ is a \textit{splitting field} for $A/F$, or that $A/F$ \textit{splits} over $K$, if $A/F\otimes_F K \cong M_n(K)$ for some $n$.  Two CSA's $A/F$ and $B/F$ are called \textit{similar} if $A/F\otimes_{F} M_r(F) \cong B/F\otimes_{F} M_s(F)$ for some $r, s$ and we write $A/F \sim B/F$.  If $F$ is understood, we will abbreviate $A/F$ to $A$.
\end{defn}

We have the following characterization and properties of central simple algebras.

\begin{prop}
\label{prop:csa}
For a finite dimensional $F$-algebra $A$ the following are equivalent:
\begin{enumerate}
\item[(i)] $A$ is a CSA.
\item[(ii)] If $F_s$ is the separable closure of $F$, then $A$ splits over $F_s$.
\item[(iii)] There exists a finite Galois extension $K/F$ such that $A$ splits over $K$.
\item[(iv)] $A\cong M_n(D)$ for some $n$, where $D$ is a skew field over $F$ of finite degree. 
\end{enumerate}
\end{prop}

\begin{proof}
See \cite[Proposition 6.3.1]{NSW}.
\end{proof}

\begin{prop}
Let $A/F$ be a CSA.  Then $A\otimes_F A^{op}\cong M_n(F)$, where $n=\dim_F A$, and $A^{op}$ is the opposite algebra, i.e., $A$ equipped with the multiplication $a\cdot b= ba$.
\end{prop}

\begin{proof}
See \cite[Proposition 3.3.1]{Ledet}.
\end{proof}

The $F$-algebra $A^e=A\otimes_F A^{op}$ is sometimes known as the \textit{enveloping algebra} of $A$.  Also, by \cite[Corollary 3.3.4]{Ledet}, the tensor product $A\otimes_F B$ of two central simple $F$-algebras is again central simple and $\sim $ respects tensor products, so we have a well defined composition on the similarity classes of CSA's over $F$ given by
\[
[A][B]=[A\otimes_F B].
\]
This is associative, commutative, has identity $1=[F]$ and $[A]^{-1}=[A^{op}]$, so the similarity classes constitute an abelian group, known as the \textit{Brauer group} $\Br(F)$ of $F$.

\begin{ex}
Let $K/F$ be a Galois extension of degree $n=[K:F]$ with Galois group $G=\Gal(K/F)$.  Let $x:G\times G \to \kx$ be a normalized (i.e. $x(\sigma ,1)=x(1,\sigma)=1$) 2-cocycle and consider the $n$-dimensional $K$-vector space
\[
K^{(G)}=\bigoplus_{\sigma \in G}Ke_\sigma
\]
with coordinates indexed by $G$.  Define a multiplication on $K^{(G)}$ by
\[
(\sum_\sigma a_\sigma e_\sigma)(\sum_\tau b_\tau e_\tau)=\sum_{\sigma,\tau}a_\sigma \sigma b_\tau x(\sigma, \tau)e_{\sigma \tau}.
\]
This multiplication has identity $1=e_1$ and is associative due to the cocycle relation
\[
x(\sigma,\tau)x(\sigma\tau,\rho)=\sigma x(\tau,\rho)x(\sigma, \tau\rho),
\]
hence making $K^{(G)}$ an $n^2$-dimensional $F$-algebra, which is called the \textit{crossed product} of $K$ and $G$ by $x$, denoted $(K,G,x)$.
\end{ex}

\begin{prop}
\label{prop:crossed product}
Crossed product algebras have the following properties:
\begin{enumerate}
\item[(i)] $(K,G,x)$ is a central simple $F$-algebra which splits over $K$.
\item[(ii)] The normalized cocycles $x$ and $y$ are cohomologous if and only if $ (K,G,x)\cong (K,G,y) $.
\item[(iii)] $ (K,G,xy)\sim (K,G,x)\otimes_F (K,G,y)$.
\item[(iv)] Every central simple $F$-algebra which splits over $K$ is similar to a crossed product algebra.
\end{enumerate}
\end{prop}

\begin{proof}
See \cite[Proposition 6.3.3]{NSW}.
\end{proof}

If $L/F$ is a field extension, we have the \textit{restriction} homomorphism
\[
res_{L/F}:\Br(F)\to \Br(L),\quad [A]\mapsto [A\otimes_F L].
\]
The kernel of $res_{L/F}$ is the \textit{relative Brauer group}, $\Br(L/F)$, which is the group of central simple $F$-algebras which split over $L$.  If $K/F$ runs through the finite Galois subextensions of $F_s/F$, then by Proposition \ref{prop:csa}(3)
\[
\Br(F)=\bigcup_K \Br(K/F).
\]
Given a normalized 2-cocycle $x$, we can associate to the cohomology class $[x]\in H^2(\Gal(K/F),\kx)$ the class $[(K,G,x)]$ to obtain a map
\[
H^2(\Gal(K/F),\kx)\to \Br(K/F),
\]
which, by Proposition \ref{prop:crossed product}, is a group isomorphism.  If $F\subseteq K\subseteq L$ are two finite Galois extensions, then the diagram
\[
\begin{tikzcd}
H^2(\Gal(L/F),L^\times ) \arrow{r} &\Br(L/F)\\
H^2(\Gal(K/F),K^\times ) \arrow{r} \arrow{u}{inf} &\Br(K/F) \arrow[hook]{u}
\end{tikzcd}
\]
commutes and taking direct limits gives

\begin{thm}
For every Galois extension $K/F$ we have a canonical isomorphism
\[
H^2(\Gal(K/F),\kx)\cong \Br(K/F).
\]
In particular,
\[
H^2(G_F,F_s^\times)\cong \Br(F),
\]
so $\Br(F)$ is a torsion group.
\end{thm}

\begin{proof}
See \cite[Theorem 6.3.4]{NSW}.
\end{proof}

We can again consider, for $n$ prime to the characteristic of $F$, the exact sequence
\[
\begin{tikzcd}
1 \arrow{r} &\mu_n \arrow{r} &F_s^\times \arrow{r}{a\mapsto a^n} &F_s^\times \arrow{r} &1
\end{tikzcd}
\]
of $G_F$-modules.  The associated exact cohomology sequence, together with Hilbert's Satz 90 and the above theorem, then yields the isomorphism
\[
H^2(G_F,\mu_n)\cong \Br_n(F)=\{b\in \Br(F)\mid b^n=1\}.
\]

If $K/F$ is a cyclic extension of degree $n$ with Galois group $C_n=\langle \sigma \rangle $, we have
\[
\Br(K/F)\cong H^2(C_n,\kx)\cong \fx/N_{K/F}(\kx),
\]
where the class of $a\in \fx$ corresponds to the similarity class of the algebra $(K,C_n,x)=K[s]$ having relations $s^n=a$ and $sb=\sigma bs$ for $b\in K$.  In this case the crossed product algebra $(K,C_n,x)$ is often written $(K,\sigma,a)$ and is referred to as a \textit{cyclic algebra}.

\begin{ex}
Let $F$ be a field of characteristic $\neq 2$ and let $K=F(\sqrt{a})$ for some $a\in \fx\setminus \fxs$.  The cyclic algebra $(K,\sigma,b)=F[i,j]$, where $i=\sqrt{a},\ j^2=b\in \fx$, $\sigma$ generates $C_2=\Gal(K/F)$ and $ji=\sigma ij=-ij$.  This algebra is 4-dimensional over $F$ with basis $\{1,i,j,ij\}$ and is split if and only if $b$ is a norm in $K/F$.  
\end{ex}

\section{Quadratic Forms and Quaternion Algebras}
\label{sec:quad}

\begin{defn}
An algebra $Q/F$ in characteristic $\neq2$ is called a \textit{quaternion algebra} if $Q$ is generated over $F$ by elements $i$ and $j$ such that $i^2=a,\ j^2=b$ and $ji=-ij$ for some $a,b\in \fx$.
\end{defn}

\begin{ex}
$(F(\sqrt{a}),\sigma,b)$ is a quaternion algebra whenever $a\notin \fxs$.  In fact, for any $a,b\in \fx$ there is a corresponding quaternion algebra $Q$, denoted
\[
Q=(\frac{a,b}{F}) =(a,b)_F
\]
which is a 4-dimensional central simple algebra and is either split or a division algebra.
\end{ex}
 
Quaternion algebras have a close connection to quadratic forms.  We briefly recall some of the basic theory and notation.

Let $F$ be a field of characteristic $\neq2$ and let $V$ be a finite-dimensional $F$-vector space.  A symmetric bilinear form on $V$ is a map $B:V\times V \to F$, such that
\begin{enumerate}
\item $B(au+bv,w)=aB(u,w)+bB(v,w)$, and
\item $B(u,v)=B(v,u)$
\end{enumerate}
for all $u,v,w\in V$ and $a,b\in F$.  The associated \textit{quadratic form} is the diagonal map
\[
q:V\to F,\quad u\mapsto B(u,u).
\]
Since $q(u+v)=q(u)+q(v)+2B(u,v)$, one can recover $B$ from $q$.  The pair $(V,q)$ is called a \textit{quadratic space}.

\begin{ex}
Let $V=F^n$ and let $a_1,\ldots,a_n \in F$.  Then the map
\[
V\to F,\quad (x_1,\ldots,x_n)\mapsto a_1x_1^2+\cdots +a_nx_n^2
\]
is an $n$-ary quadratic form on $F^n$, called a \textit{diagonal form} and denoted by $\langle a_1,\ldots,a_n \rangle $.  
\end{ex}

Given two quadratic spaces $(V_1,q_1)$ and $(V_2,q_2)$ over $F$, where $q_1=\langle a_1,\ldots,a_n \rangle$ and $q_2=\langle b_1,\ldots,b_m \rangle$, one can define the \textit{orthogonal sum} $(V_1\oplus V_2,q_1\perp q_2):=(V_1,q_1)\perp (V_2,q_2)$ and \textit{tensor product} $(V_1\otimes V_2,q_1\otimes q_2):=(V_1,q_1)\otimes (V_2,q_2)$, where
\[
q_1\perp q_2:=\langle a_1,\ldots,a_n,b_1,\ldots,b_m \rangle
\]
and
\[
q_1\otimes q_2:=\langle a_1b_1,\ldots,a_1b_m,\ldots,a_nb_1,\ldots,a_nb_m \rangle.
\]

Two $n$-ary quadratic forms $q$ and $q'$ are said to be \textit{equivalent} $(\sim)$ if there exists an invertible matrix $C\in GL_n(F)$ such that $q'(Cv)=q(v)$ and two quadratic spaces $(V,q)$ and $(V',q')$ are said to be \textit{isometric} $(\cong)$ if there exists an $F$-isomorphism $\varphi:V\to V'$ such that $q'(\varphi(v))=q(v)$.  It can be shown that every quadratic form is equivalent to a diagonal form and that there is a one-one correspondence between the equivalence classes of $n$-ary quadratic forms and the isometry classes of $n$-dimensional quadratic spaces. 

\begin{defn}
Let $q$ be an $n$-ary quadratic form over $F$ and let $a\in \fx$.  We say that \textit{q represents a} if there exist $x_1,\ldots,x_n \in F$ such that $q(x_1,\ldots,x_n)=a$.  The set of values in $\fx$ represented by $q$, or the \textit{value set} of $q$, is denoted $D_F(q)$.  Note that this set depends only on the equivalence class of $q$.
\end{defn}

Let $(V,q)$ be a quadratic space and let $B$ be the symmetric bilinear form corresponding to $q$.  Two vectors $u,v \in V$ are said to be orthogonal, written $u\perp v$, if $B(u,v)=0$.  If $U$ is a subspace of $V$, the orthogonal complement of $U$ is the subspace
\[
U^\perp = \{v\in V\mid \forall u\in U: u\perp v\}.
\]
If $U\cap U^\perp =0$, we say that $U$ is a \textit{regular} subspace.  If $V^\perp =0$, we say that $q$ is regular or \textit{nonsingular}.  We say that a vector $v$ is \textit{isotropic} if $q(v)= 0$ and \textit{anisotropic} otherwise. The quadratic form $q$ is called isotropic, or is said to \textit{represent zero}, if $q$ has a  non-zero isotropic vector, otherwise it is called anisotropic.  An isotropic form can be regular, for example the binary form $\langle 1,-1 \rangle$. 

\begin{thm}
\label{thm:hyperbolic}
Let (V,q) be a 2-dimensional quadratic space.  The following are equivalent:
\begin{enumerate}
\item[(i)] V is regular and isotropic.
\item[(ii)] V is isometric to $\langle1,-1\rangle$.
\item[(iii)] V corresponds to the equivalence class of the binary quadratic form $x_1x_2$.
\end{enumerate}
\end{thm}

\begin{proof}
See \cite[Theorem I.3.2]{Lam}
\end{proof}

The isometry class of a 2-dimensional quadratic space satisfying the conditions of Theorem \ref{thm:hyperbolic} is called the \textit{hyperbolic plane}, denoted $\H$.  An orthogonal sum of hyperbolic planes is called a \textit{hyperbolic space}.  E. Witt \cite{Wit} showed that any regular quadratic space $(V,q)$ splits into an orthogonal sum $(V_h,q_h)\perp (V_a,q_a)$ where $V_h$ is hyperbolic, $V_a$ is anisotropic and the isometry types of $V_h$ and $V_a$ are uniquely determined.

\begin{defn}
The set of equivalence classes of anisotropic nonsingular quadratic forms over a field $F$, together with the binary operations $\perp$ and $\otimes$, form a commutative ring, known as the \textit{Witt ring} $W(F)$ of $F$.  
\end{defn}

Returning now to the quaternion algebra $ Q=(a,b)_F $ with basis $\{1,i,j,k=ij\}$, we can make $Q$ into a quadratic space as follows.  For an arbitrary quaternion $x=\alpha+\beta i+\gamma j+\delta k \in Q$, we define a map $N:Q\to F$ by $N(x)= \alpha^2-a\beta^2-b\gamma^2+ab\delta^2$.  Then $N=\langle 1,-a,-b,ab \rangle$ is a quadratic form, referred to as the norm form of $Q$, and $N(x)$ is called the norm of the quaternion $x$.

\begin{thm}
For $a,b,c,d\in \fx$, the following are equivalent:
\begin{enumerate}
\item[(i)] $\langle 1,-a,-b,ab \rangle \sim \langle 1,-c,-d,cd \rangle$.
\item[(ii)] $\langle -a,-b,ab \rangle \sim \langle -c,-d,cd \rangle$.
\item[(iii)] $(a,b)_F \cong (c,d)_F$.
\end{enumerate}
\end{thm}

\begin{proof}
See \cite[Ch. 2, Theorem 11.9]{scharlau}.
\end{proof}

Hence a quaternion algebra is completely determined by its norm form. This also leads to the following criteria for the splitting of a quaternion algebra.

\begin{cor}
The following are equivalent:
\begin{enumerate}
\item[(i)] $\langle -a,-b,ab \rangle$ is isotropic.
\item[(ii)] $(a,b)_F$ is split.
\item[(iii)] The binary form $\langle a,b \rangle$ represents 1.
\item[(iv)] $b\in N_{K/F}(K)$, where $K=F(\sqrt{a})$ and $N_{K/F}$ is the field norm.
\end{enumerate}
\end{cor}

We denote the equivalence class of the quaternion algebra $(a,b)_F$ in $\Br(F)$ by $(a,b)$ and call it a \textit{quaternion class}.  We have

\begin{thm}
\label{thm:quat}
Let $a,a',b,b',x,y \in \fx$.  Then
\begin{enumerate}
\item[(i)] $(a,b)=(b,a)$ and $(ax^2,by^2)=(a,b)$.
\item[(ii)] $ (a,-a)=1 $ and $(a,1-a)=1$ if $a\neq1$.
\item[(iii)] $(a,a)=(a,-1)$.
\item[(iv)] $(aa',b)=(a,b)(a',b)$ and $(a,bb')=(a,b)(a,b')$.
\end{enumerate}
\end{thm}

\begin{proof}
See \cite[Theorem 3.5.3]{Ledet}.
\end{proof}

\begin{defn}
A pairing $f:\fx \times \fx \to A$ into a multiplicative abelian group $A$ is said to be a \textit{Steinberg symbol} if $f$ is bimultiplicative and has the \textit{Steinberg property}; $f(a,b)=1$ whenever $a+b=1$.
\end{defn}

Every such symbol factors through a group $K_2(F)$, called the \textit{second $K$-group} of the field $F$, which is defined by
\[
K_2(F)=(\fx \otimes_\Z \fx)/ \langle a\otimes b\mid a+b=1 \rangle.
\]
The natural pairing $\varphi: \fx \times \fx \to K_2(F)$ is then a \textit{universal Steinberg symbol}; i.e. for any arbitrary Steinberg symbol $f:\fx \times \fx \to A$ there exists a unique group homomorphism $g:K_2(F)\to A$ such that $f=g\varphi$.

We define the group $k_2(F)$ to be the quotient $K_2(F)/(K_2(F))^2$.  Any Steinberg symbol into an abelian group $A$ with $A^2=1$ then factors uniquely through $k_2(F)$.  Theorem~\ref{thm:quat} shows that the quaternion map $(-,-):\fx \times \fx \to \Br(F)$ is a symmetric bilinear form defined on the square classes of $\fx$, which defines a Steinberg symbol into $\Br_2(F)$.  This symbol is induced by the unique group homomorphism $k_2(F)\to \Br_2(F)$ given by $[a]\otimes [b] \mapsto (a,b)$. 

\begin{thm}[Merkurjev]
\label{thm:Merk}
The map $k_2(F)\to \Br_2(F)$ is an isomorphism.
\end{thm}

\begin{proof}
See \cite{Merk}.
\end{proof}

This leads to the notion of a \textit{symbol} for any field $F$ as a multi-multiplicative map
\[
\fx \times \cdots \times \fx \to A, \quad (a_1,\ldots,a_n)\mapsto [a_1,\ldots,a_n],
\]
into a (multiplicatively written) abelian group $A$ such that $[a_1,\ldots,a_n]=1$ whenever $a_i+a_j=1$ for some $i\neq j$.  Every such symbol factors through a group $K_n^M(F)$, which is the universal target of symbols and is defined as follows.

\begin{defn}
The $n$-th \textit{Milnor $K$-group} of a field $F$ is the quotient
\[
K_n^M(F)=(\fx \otimes_\Z \cdots \otimes_\Z \fx)/I_n,
\]
where $I_n$ is the subgroup generated by the elements $a_1\otimes \cdots \otimes a_n$ such that $a_i+a_j=1$ for some $i\neq j$.
\end{defn}

The Milnor $K$-groups have a close connection to Galois cohomology, which arises as follows.

Let $k\in \N$ be prime to the characteristic of $F$.  Recall that the exact sequence
\[
\begin{tikzcd}
1 \arrow{r} &\mu_k \arrow{r} &F_s^\times \arrow{r}{a\mapsto a^k} &F_s^\times \arrow{r} &1
\end{tikzcd}
\]
gives a surjective homomorphism
\[
\delta_F:\fx \to H^1(G_F,\mu_k)
\]
with kernel $(\fx)^k$.  Also, for each $n\geq1$, we have the cup product
\[
\begin{tikzcd}
H^1(G_F,\mu_k) \times \cdots \times H^1(G_F,\mu_k) \arrow{r}{\cup} &H^n(G_F,\mu_k^{\otimes n}),
\end{tikzcd} 
\]
hence a map
\[
\fx \times \cdots \times \fx \to H^n(G_F,\mu_k^{\otimes n}), \quad (a_1,\ldots,a_n)\mapsto (a_1,\ldots,a_n)_F:=\delta_F a_1 \cup \cdots \cup \delta_F a_n.
\]
\begin{thm}[Tate]
The above map induces a homomorphism
\[
h_F:K_n^M(F)\to H^n(G_F,\mu_k^{\otimes n}),
\]
called the Galois symbol (or the norm residue map).
\end{thm}

\begin{proof}\cite[Theorem 6.4.2]{NSW}
The multiplicativity in each argument follows from the definition.  It remains to show that $(a_1,\ldots ,a_n)_F =1$ if $a_i+a_j=1$ for $i\neq j$ and it suffices to consider the case $n=2$ since if $n>2$ and, say $i=1,\ j=2$, then $(a_1,\ldots ,a_n)_F =(a_1,a_2)_F\cup (a_3,\ldots,a_n)_F$.

Let $n=2$ and let $a\in \fx,\ a\neq 1$.  Let $X^n-a=\prod_{i}f_i(X)$ with $f_i(X)$ monic and irrreducible in $F[X]$.  For each $i$, let $a_i$ be a root of $f_i(X)$ and let $F_i=F(a_i)$.  Then
\[
1-a=\prod_{i}f_i(1)=\prod_{i}N_{F_i/F}(1-a_i).
\]
Hence
\[
(1-a,a)_F=(\prod_{i}N_{F_i/F}(1-a_i),a)_F=\prod_{i}(N_{F_i/F}(1-a_i),a)_F.
\]
The formula $cor(\alpha \cup res \beta)=(cor \alpha)\cup \beta$ together with the fact that $cor$ is the norm on $H^0$ and commutes with $\delta$ gives
\[
\begin{aligned}
(N_{F_i/F}(1-a_i),a)_F &=cor(1-a_i,a)_{F_i}\\
&=cor (1-a_i, a_i^k)\\
&=cor (1-a_i,a_i)^k\\
&=1,
\end{aligned}
\]
hence $(1-a,a)=1$.
\end{proof}

For the Galois symbol we have

\begin{conj}[Bloch-Kato]
For every field $F$ and every $k\in \N$ prime to the characteristic of $F$, the Galois symbol yields an isomorphism
\[
h_F:K_n^M(F)/NK_n^M(F)\ \tilde{\longrightarrow} \ H^n(G_F,\mu_k^{\otimes n}).
\]
\end{conj}

This famous conjecture has recently been proved by V. Voevodsky with contributions by M. Rost and C. Weibel.  The result is often referred to as the \textit{norm residue isomorphism}.

\section{The Incidence Algebra and M\"{o}bius Functions}

Often a set of objects to be counted possesses a natural partial ordering.  As a result, many problems of enumeration are closely related to the theory of M\"{o}bius functions.  In this section we recall some pertinent aspects of that general theory (see \cite{Rota,Bend,Wall}).  

Consider a partially ordered set $P=(S,\leqq)$, where $\leqq$ is an order relation on the set $S$.  For any $x,y\in P$, the \textit{segment} $[x,y]:=\{z\in P\mid x\leqq z\leqq y\}$.  A partially ordered set $P$ is \textit{locally finite} if every segment in $P$ is finite.

Let $P$ be a locally finite partially ordered set.  The \textit{incidence algebra} of $P$ is defined as follows.  Consider the set of all real-valued functions of two variables $f(x,y)$, defined for $x,y\in P$, with the property that $f(x,y)=0$ if $x\nleqq y$.  The sum of two such functions as well as multiplication by scalars are defined as usual.  The product $h=fg$ is defined as follows:
\[
h(x,y)=\sum_{x\leqq z \leqq y}f(x,z)g(z,y).
\]
Since $P$ is locally finite, the sum on the right is well defined.  This is an associative algebra over the reals and has an identity element which is the Kronecker delta function, $\delta (x,y)$.

The \textit{zeta function} of $P$ is the element of the incidence algebra of $P$ given by $\zeta (x,y)=1$ if $x\leqq y$ and $\zeta (x,y)=0$ otherwise.  The function $n(x,y)=\zeta(x,y)-\delta(x,y)$ is called the \textit{incidence function}.

\begin{prop}
The zeta function of a locally finite partially ordered set $P$ is invertible in the incidence algebra.
\end{prop}

\begin{proof}
Let $\mu(x,y)$ be the function defined inductively over the elements in the segment [x,y] as follows.  Set $\mu(x,x)=1$ for all $x\in P$.  Now suppose that $\mu(x,z)$ has been defined for all $z$ such that $x\leqq z < y$ and set
\[
\mu(x,y)=-\sum_{x\leqq z < y}\mu(x,z).
\]
Then
\[
\begin{aligned}
(\zeta \mu)(x,y)&=\sum_{x\leqq z \leqq y}\zeta(x,z) \mu(z,y)\\
&=\sum_{x\leqq z \leqq y}\mu(z,y)\\
&=\delta(x,y),
\end{aligned}
\]
and similarly $(\mu\zeta)(x,y)=\delta(x,y)$.  The function $\mu$, the inverse of $\zeta$, is called the \textit{M\"{o}bius function} of the partially ordered set $P$.
\end{proof}

\begin{prop}[M\"{o}bius inversion formula I]
Let $f:P\to \R$ be defined for all $x$ in a locally finite partially ordered set $P$ and assume there exists an element $m\in P$ such that $f(x)=0$ unless $x\geqq m$.  Suppose that $g:P\to \R$ is given by
\[
g(x)=\sum_{y\leqq x}f(y).
\]
Then
\[
f(x)=\sum_{y\leqq x}g(y)\mu(y,x).
\]
\end{prop}

\begin{proof}
Since $P$ is locally finite, $\sum_{y\leqq x}f(y)=\sum_{m\leqq y\leqq x}f(y)$ is a finite sum.  Hence the function $g$ is well-defined.  Then
\[
\begin{aligned}
\sum_{y\leqq x}g(y)\mu(y,x)&=\sum_{y\leqq x} \sum_{z\leqq y}f(z)\mu(y,x)\\
&=\sum_{y\leqq x} \sum_{z}f(z)\zeta(z,y)\mu(y,x)\\
&=\sum_{z}f(z)\sum_{y\leqq x}\zeta(z,y)\mu(y,x)\\
&=\sum_{z}f(z)\delta(z,x)\\
&=f(x).
\end{aligned}
\]
\end{proof}

A similar argument establishes

\begin{prop}[M\"{o}bius inversion formula II]
Let $f:P\to \R$ be defined for all $x$ in a locally finite partially ordered set $P$ and assume there exists an element $M\in P$ such that $f(x)=0$ unless $x\leqq M$.  Suppose that $g:P\to \R$ is given by
\[
g(x)=\sum_{y\geqq x}f(y).
\]
Then
\[
f(x)=\sum_{y\geqq x}\mu(x,y)g(y).
\]
\end{prop}

\begin{cor}
The M\"{o}bius function $\mu$ of a locally finite partially ordered set can be computed recursively by either of the formulae
\[
\begin{aligned}
\mu(x,z)&=-\sum_{x\leqq y < z}\mu(x,y),\quad x<z,\\
\mu(x,z)&=-\sum_{x< y \leqq z}\mu(y,z),\quad x<z,
\end{aligned}
\]
together with $\mu(x,x)=1$.
\end{cor}

\begin{ex} \cite[Example 1]{Rota}
The classical M\"{o}bius function defined on the set of positive integers is given by $\mu(d)=(-1)^r$ if $d$ is a product of $r$ distinct primes and 0 otherwise.  The classical inversion formula, first derived by M\"{o}bius in 1832, is:
\[
g(m)=\sum_{n|m}f(n);\quad f(m)=\sum_{n|m}g(n)\mu(m/n).
\]
The set of positive integers is a locally finite partially ordered set with divisibility as the partial order.  In this case the incidence algebra has a distinguished subalgebra consisting of all functions $f$ of the form $f(n,m)=F(m/n)$.  The M\"{o}bius function of this partially ordered set is $\mu(n,m)=\mu(m/n)$.  The product $H=FG$ of two functions in this subalgebra can be written in the simpler form
\[
H(m)=\sum_{kn=m}F(k)G(n).
\]
If we associate the \textit{formal Dirichlet series} $\hat{F}(s)=\sum_{n=1}^{\infty}F(n)/n^s$ with the element $F$ of this subalgebra, then the above product corresponds to the product of two formal Dirichlet series considered as functions of $s$, $\hat{H}(s)=\hat{F}(s)\hat{G}(s)$.  Under this representation, the zeta function of the partially ordered set is the classical \textit{Riemann zeta function} $\zeta(s)=\sum_{n=1}^{\infty}1/n^s$, and the statement that the M\"{o}bius function is the inverse of the zeta function reduces to the classical identity $1/{\zeta(s)}=\sum_{n=1}^{\infty}\mu(n)/n^s$.
\end{ex}

The lattice of subgroups of a finite group $G$ is a locally finite partially ordered set.  The theory of M\"{o}bius functions in this case is of particular interest in counting Galois extensions.  A \textit{subgroup function} from $G$ to $\Z$ is a mapping of the lattice of subgroups of $G$ into $\Z$.  The equations
\[
\mu_G(G)=1 \quad \text{and}\quad \sum_{H\leq K}\mu_G(K)=0\ \quad \text{whenever}\ H<G,
\]
define the \textit{M\"{o}bius function} $\mu_G$ of $G$; $\mu_G$ is a subgroup function from $G$ to $\Z$.  Note that if $N\vartriangleleft G$, $N\leq H\leq G$, then $\mu_{G/N}(H/N)=\mu_G(H)$.

If two subgroup functions $g,h$ satisfy
\[
g(H)=\sum_{K\leq H}h(K)
\]
for all $H\leq G$, then by the M\"{o}bius inversion formula
\[
h(G)=\sum_{H\leq G}\mu_G(H)g(H).
\]

An explicit formula for $\mu_G(H)$ can be obtained as follows.  Let $M_1,\ldots ,M_r$ be the maximal subgroups of $G$.  If $S=\{i_1,\ldots,i_s\}$ is a subset of $I=\{1,\ldots,r\}$, let
\[
\begin{aligned}
(-1)^S&:=(-1)^s,\\
M_S&:=M_{i_1}\wedge M_{i_2}\wedge \cdots \wedge M_{i_s};
\end{aligned}
\]
so $M_\phi=G$, and $M_I$ is the Frattini subgroup $\Phi(G)$ of $(G)$.  Let $S_H$ denote the set of indices $i$ such that $H\leq M_i$.  Then
\[
\begin{aligned}
\sum_{H\leq M_S}(-1)^S&=\sum_{S\subseteq S_H}(-1)^S\\
&=
\begin{cases}
1 & \text{if}\ H=G\\
0 & \text{if}\ H<G,
\end{cases}
\end{aligned}
\] 
so
\[
\mu_G(H)=\sum_{M_S=H}(-1)^S.
\]

It follows that $\mu_G(H)=0$ unless $H$ is an intersection of maximal subgroups of $G$.  In particular, $\mu_G(H)=0$ unless $\Phi(G)\leq H$.

Now consider the case in which $G$ is a $p$-group.  Let $V_n(q)$ be an $n$-dimensional vector space over the field of $q$ elements and partially order the subspaces of $V_n(q)$ by inclusion.  Denote the resulting partially ordered set by $L(V_n(q))$.  The Gaussian coefficient $\binom{n}{k}_q$ is defined to be the number of $k$-dimensional subspaces of $V_n(q)$.  Hence
\[
\begin{aligned}
\binom{n}{k}_q&= \frac{\text{\# of sequences of \textit{k} independent vectors in}\ V_n(q)}{\text{\# of sequences of \textit{k} independent vectors in}\ V_k(q)}\\
&=\frac{(q^n-1)(q^n-q)\cdots (q^n-q^{k-1})}{(q^k-1)(q^k-q)\cdots (q^k-q^{k-1})}\\
&=\frac{(q^n-1)(q^{n-1}-1)\cdots (q^{n-k+1}-1)}{(q^k-1)(q^{k-1}-1)\cdots (q-1)}.
\end{aligned}
\]

For any two subspaces $S$ and $T$ of $V_n(q)$, the structure of the sublattice $[S,T]$ of $L(V_n(q))$ depends only on $\dim_{\F_q}T-\dim_{\F_q}S$, so computing the M\"{o}bius function of the partially ordered set $L(V_n(q))$ reduces to determining the M\"{o}bius function $\mu_k=\mu(0,V_k(q))$ for all $k\leq n$.

Let $X$ be a vector space over $\F_q$ with $|X|=x$.  For any subspace $U\in L(V_n(q))$ let $N_=(U)$ be the number of linear transformations $f:V_n(q)\to X$ whose kernel is $U$ and let $N_{\geq}(U)$ be the number of such maps whose kernel contains $U$.  Then
\[
N_{\geq}(U)=\sum_{U\leq W\in L(V_n(q))}N_=(W),
\]
and by M\"{o}bius inversion
\[
N_=(U)=\sum_{U\leq W\in L(V_n(q))}\mu(U,W)N_{\geq}(W).
\]
The number of injective maps is then
\[
N_=(0)=\sum_{W\in L(V_n(q))}\mu(0,W)N_{\geq}(W).
\]
Since any injective map from $V_n(q)\to X$ is specified by giving the image of an ordered basis of $V_n(q)$, the number of such maps is $(x-1)(x-q)\cdots (x-q^{n-1})$.  Also, if $W$ has $\F_q$-dimension $d(W)$, then $N_{\geq}(W)=x^{n-d(W)}$.  Hence
\[
(x-1)(x-q)\cdots (x-q^{n-1})=\sum_{W\in L(V_n(q))}\mu_{d(W)}x^{n-d(W)}=\sum_{k=0}^{n}\binom{n}{k}_q\mu_kx^{n-k}.
\]
Since this identity is true for infinitely many values of $x$, it is a polynomial identity.  Equating the constant terms gives
\[
\mu_n=(-1)(-q)\cdots (-q^{n-1})=(-1)^nq^{\frac{1}{2}n(n-1)}.
\]

Thus we have

\begin{lem}
\label{lem:muG}
If $G$ is a $p$-group and $H\leq G$ with $[G:H]=p^i$ then
\[
\mu_G(H)=
\begin{cases}
(-1)^ip^{\frac{1}{2}i(i-1)}\ &\text{if}\ G^p[G,G]\leq H\\
0\ &\text{otherwise}.
\end{cases}
\]
\qed
\end{lem}
This lemma will be important when we consider M. Yamagishi's method for determining the number of Galois $p$-extensions of certain local fields in section \ref{sec:mobius}.

\chapter{Counting Galois $p$-Extensions}
\label{ch:extensions}

The goal of this chapter is to illustrate several techniques for enumerating Galois $p$-extensions of various fields.  These numbers are important in the study of quotients and filtrations of absolute Galois groups.  In section \ref{sec:dihedral} we point out a connection between the number of Galois extensions of a field $F$ having Galois group isomorphic to $D_4$, the dihedral group of order 8, and a particular small quotient of $G_F$ referred to as the W-group of $F$.  

We turn to the problem of counting $D_4$-extensions of local fields, beginning, in section \ref{sec:local}, with a method of constructing extensions of the $p$-adic numbers, $\Q_p$, due to H. Naito \cite{Naito}.  We then provide an alternative, group-theoretic approach based on knowledge of the W-group as well as a method which utilizes the theory of quaternion algebras.  These are new techniques for determining the number of $D_4$-extensions of $\Q_p$ which are presented as an alternative to the direct construction approach of Naito.  

Section \ref{sec:mobius} describes a technique, due to M. Yamagishi \cite{yam}, using the theory of complex characters and M\"{o}bius functions to count finite Galois $p$-extensions of a local field $K$, where $K$ is a finite extension of $\Q_p$.  We illustrate this method in the case of $D_4$-extensions in example \ref{ex:Q2}.  In \cite{MT1} J. Min\'a\v{c} and N. D. T\^an develop a technique to compute the number of $\U_4(\F_p)$-extensions of $K$ using triple Massey products.  We closely follow this approach in section \ref{sec:cup}, with the necessary modifications, to show that cup products can be used to determine a formula for the number of $D_4$-extensions of $K$ based on the degree $n=[K:\Q_p]$.

In section \ref{sec:pyth} we consider formally real pythagorean fields.  It is interesting to note that in 1900, David Hilbert posed a famous list of twenty-three problems and it was the theory of formally real fields that led Emil Artin, in 1927, to a solution of Hilbert's seventeenth problem.  After reviewing the basic theory, we develop, in section \ref{sec:sap}, a formula for the number of $D_4$-extensions of a pythagorean SAP field.  In section \ref{sec:pythGF2} we characterize the group $G_F(2)$ for $F$ a pythagorean SAP field or a superpythagorean field.  These groups will be considered further when we study dimensions of Zassenhaus filtration subquotients in chapter \ref{ch:dimensions}. 

\section{The Inverse Galois Problem}

A central problem in modern Galois theory is the inverse Galois problem: given a field $F$ and a group $G$, is it possible to construct a Galois extension $K/F$ with Galois group isomorphic to $G$?  Such an extension $K/F$ is often referred to as a $G$-extension.  If such a construction is possible, then the closely related question of counting the number of $G$-extensions of $F$ naturally arises. 

The embedding problem in Galois theory generalizes the inverse problem and consists of finding the conditions under which one can construct a Galois extension $K/F$, with group $G$, such that $K$ extends a given Galois extension $L/F$ whose Galois group is a quotient of $G$.  If the group $G$ contains a normal subgroup $H$, then a natural approach to solving the inverse problem for the field $F$ and the group $G$ is to choose an extension $L/F$ with Galois group $G/H$ which can in turn be embedded in an appropriate extension $K$.  

Probably the simplest example of an embedding problem is the following well known result (see, for example \cite{Ishk}), which is also related to the notion of pythagorean fields.  We include a proof in order to illustrate that even this case is nontrivial.  

\begin{thm}
Let $F$ be a field with $\textnormal{char}(F)\ne 2$ and let $K=F(\sqrt{a})$ be a quadratic extension of $F$.  Then $K$ can be embedded in a cyclic extension $L/F$ of degree 4 if and only if $a$ is a sum of two squares in $F$.
\end{thm}

\begin{proof}
We assume that $-1\notin F^2$; otherwise every element in $F$ is a sum of two squares in $F$ and $L=F(\sqrt[4]{a})$ is a solution of the embedding problem.

Suppose that $L/F$ is a $C_4$-extension with $\Gal(L/F)=\langle g \rangle$ and that $K\subseteq L$.  Let $\alpha$ be a primitive element of $L$ and let $m=(\alpha-g^2(\alpha))/(g(\alpha)-g^3(\alpha))$.  Then $m$ is well defined and $m\neq 0$.  Also, $g(m)=(g(\alpha)-g^3(\alpha))/(g^2(\alpha)-\alpha)=-m^{-1}$ and $g^2(m)=-g(m)^{-1}=m$, so $m\in K=F(\sqrt{a})$.  If $m=x+y\sqrt{a}$, where $x,y\in F$, then $x^2-ay^2=N_{K/F}(m)=mg(m)=-1$, so $y\neq 0$ and $a=(x/y)^2+(1/y)^2$ is a sum of two squares in $F$.

Conversely, suppose $a=u^2+v^2$ with $u,v\in F,\ v\neq 0$ and let $m=(u+\sqrt{a})/v$.  Then, if $\bar{g}$ is the automorphism of $K=F\sqrt{a}$ defined by $\sqrt{a}\mapsto -\sqrt{a}$, we have $m\bar{g}(m)=-1$.  Now let $\lambda=1+m^2$.  Then $\lambda\neq 0$ and $\bar{g}(\lambda)=1+1/m^2=\lambda/m^2$.  Let $\theta=\sqrt{\lambda}$, let $L=K(\theta)$ and let $g$ be an automorphism of $L$ extending the automorphism $\bar{g}$ of $K$.  Then $g(\theta)^2=g(\lambda)=\lambda/m^2=(\theta/m)^2$, so, up to sign, $g(\theta)=\theta/m$.  Hence $g$ is an automorphism of $L/F$ and furthermore,
\[
\begin{aligned}
g^2(\theta)&=g(\theta)/\bar{g}(m)=\theta/(m\bar{g}(m))=-\theta;\\
g^3(\theta)&=-\theta/m;\\
g^4(\theta)&=\theta,
\end{aligned}
\]
so $g^4=1$.  Therefore $L=K(\theta)$ is normal over $F$ with $\Gal(L/F)=\langle g \rangle \cong C_4$.
\end{proof}
 
Embedding problems have close connections to Galois cohomology and quadratic forms.  They are also of considerable importance in the study of absolute Galois groups.  For example, from the Galois correspondence, an affirmative answer to the inverse problem is equivalent to the existence of a closed normal subgroup $H$ of the absolute Galois group $G_F$ of $F$ such that $G_F/H \cong G$. However, absolute Galois groups remain largely mysterious objects and determining which profinite groups are realizable as absolute Galois groups of various fields remains a significant open problem in Galois theory.

One means of approaching this problem is to study small quotients of absolute Galois groups.  The structure of these groups is, in turn, closely related to the problem of counting Galois extensions.  In the next section, for example, we look at the connection between the number of $D_4$-extensions of a field $F$ and the W-group of $F$.

\section{Dihedral Extensions and W-Groups}
\label{sec:dihedral}

We follow \cite{MS2} to define a special Galois extension of a base field $F$ and then summarize results which pertain to the determination of dihedral extensions of $F$.  

We fix the following notation:  $C_\textit{n}$ denotes the cyclic group of order \textit{n} and $ D_4 $ denotes the dihedral group of order 8.  We assume that all fields have characteristic different from 2 and we make no distinction between an element $a$ in a field $F$ and and its square class $a\fxs \in \fx/\fxs$.  An extension $K$ of the field $F$ is called a \textit{G-extension} if $K/F$ is Galois with Galois group $G$.

Let $F^{(2)}=F(\sqrt{a}\mid a\in F^\times) $; the compositum of all quadratic extensions of $F$, $\varGamma =\{b\in F^{(2)}\mid F^{(2)}(\sqrt{b})/F \text{ is Galois}\} $ and $F^{(3)}=F^{(2)}(\sqrt{b}\mid b\in \varGamma )$; the compositum of all quadratic extensions of $F^{(2)}$ which are Galois over $F$.  Due to its close connection with the Witt ring $W(F)$ of $F$, the field $F^{(3)}$ has been referred to as the \textit{Witt closure} of $F$ and the group Gal$(F^{(3)}/F)$ is called the \textit{W-group} of $F$.  Recall that the \textit{quadratic closure} or \textit{maximal 2-extension} of $F$, denoted $F(2)$, is the smallest extension of $F$ which is closed under taking of square roots, or alternatively, is the compositum of all 2-towers over $F$ (inside a fixed algebraic closure of $F$).  The group Gal$(F(2)/F)$ is the maximal pro-2 quotient, $G_F(2)$, of the absolute Galois group $G_F$ of $F$.  By \cite[Proposition 2.1]{MS2} we see that the W-group of $F$, Gal$(F^{(3)}/F)\cong G_F(2)^{[3]}$.

Let $\{a_i\mid i\in I\}$ be a basis of $F^\times/(F^\times)^2$.  The automorphisms $\sigma_i$ given by $\sigma_j(\sqrt{a_i})=(-1)^{\delta_{ij}}\sqrt{a_i}$, where $\delta_{ij}$ is the Kronecker delta function, form a minimal set of generators of Gal$(F^{(2)}/F)$ and they induce a natural isomorphism Gal$(F^{(2)}/F)\cong\prod_{i\in I}C_2$.  From Kummer theory, Gal$(F^{(2)}/F)$ is the Pontrjagin dual of the discrete group $F^\times/(F^\times)^2$ under the pairing $(\sigma,a)=\sigma(\sqrt{a})/\sqrt{a}$ with values in $C_2\cong\{\pm 1\}$.

We now look more closely at the structure of $F^{(3)}$.  Recall that quaternion algebras over $F$ are denoted $(a,b)_F$, or simply $(a,b)$ when the field $F$ is clear.  By Merkurjev's Theorem [\ref{thm:Merk}], the subgroup of the Brauer group Br$(F)$ generated by the isomorphism classes of quaternion algebras over $F$ is $\textnormal{Br}_2(F)$, the subgroup generated by elements of order $\leq 2$.  The operation in Br$(F)$ will be written multiplicatively, so $(a,b)_F=1$ means $(a,b)$ splits over $F$.  For $a\in F^\times$, $N_a$ denotes the \textit{norm group} of $a$, i.e., the group of values of the quadratic form $\langle 1,-a \rangle$ over $F$.  The norm of an element $y\in F(\sqrt{a})$ for $a\in F^\times \setminus  (F^\times)^2$ will be written $Ny$.

If $a\in \fx/\fxs$ then by a $C_4^a$\textit{-extension} of $F$ we mean a $C_4$-extension $K$ of $F$ such that $F(\sqrt{a})\subset K$.  For these we have the following well known result (see, for example, \cite[Chapter 2, Section 2]{JLY} or \cite[Proposition 2.3]{MS2}).

\begin{prop}
\label{prop:c4ext}
Let $a\in \fx/\fxs$.  Then there exists a $C_4^a$-extension of $F$ if and only if $ (a,a)_F=1 $.  Furthermore, $K$ is a $C_4^a$-extension of $F$ if and only if $K=F(\sqrt{a})(\sqrt{y})$ where $y\in F(\sqrt{a})$ is such that $Ny=a\in \fx/\fxs$.
\end{prop} 

Two elements $a,b\in \fx$ are called \textit{independent modulo squares} if $a$ and $b$ are linearly independent in $\fx/\fxs$.  If $a,b\in \fx$ are independent modulo squares then by a $D_4^{a,b}$\textit{-extension} of $F$ we mean a $D_4$-extension $K$ of $F$ such that $F(\sqrt{a},\sqrt{b})\subset K$ and Gal$(K/F(\sqrt{ab}))\cong C_4$.  This next proposition is also well known (see, for example, \cite[Chapter 2, Section 2]{JLY}).  For some history and discussion of more general types of Galois extensions related to these extensions see also \cite[7.7]{Fr}, \cite{Ma} or \cite{MN}. 

\begin{prop}
\label{prop:d4ext}
Let $a,b\in \fx$ be independent modulo squares.  Then there exists a $D_4^{a,b}$-extension of $F$ if and only if $ (a,b)_F=1 $.  Furthermore, $K$ is a $D_4^{a,b}$-extension of $F$ if and only if $K=F(\sqrt{a},\sqrt{b})(\sqrt{y})$ where $y\in F(\sqrt{a})$ is such that $Ny=b\in \fx/\fxs$.
\end{prop}        

The following diagram shows the lattice of subfields of a $D_4^{a,b}$-extension $K$ of $F$:
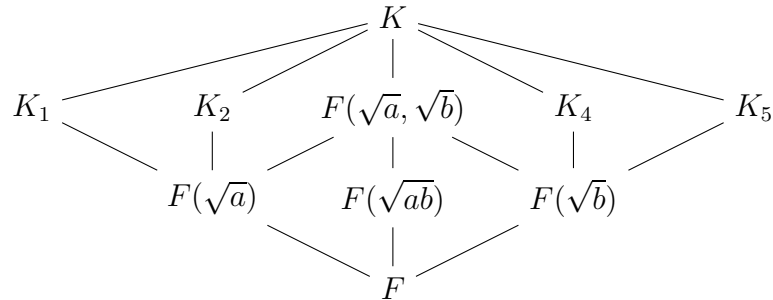
\begin{center}
\begin{tikzpicture}[scale=1.2]
\node (f) at (0,0) {$F$};
\node (f1) at (-2,1) {$F(\sqrt{a})$};
\node (f2) at (0,1) {$F(\sqrt{ab})$};
\node (f3) at (2,1) {$F(\sqrt{b})$};
\node (k1) at (-4,2) {$K_1$};
\node (k2) at (-2,2) {$K_2$};
\node (k3) at (0,2) {$F(\sqrt{a},\sqrt{b})$};
\node (k4) at (2,2) {$K_4$};
\node (k5) at (4,2) {$K_5$};
\node (k) at (0,3) {$K$};
\path [-]
(f) edge (f1)
(f) edge (f2)
(f) edge (f3)
(f1) edge (k1)
(f1) edge (k2)
(f1) edge (k3)
(f2) edge (k3)
(f3) edge (k3)
(f3) edge (k4)
(f3) edge (k5)
(k1) edge (k)
(k2) edge (k)
(k3) edge (k)
(k4) edge (k)
(k5) edge (k);
\end{tikzpicture}
\end{center}

J. Min\'a\v{c} and M. Spira have shown that $F^{(3)}$ is the compositum of all quadratic, $C_4$- and $D_4$-extensions of $F$ \cite[Corollary 2.18]{MS2}.  Furthermore, they observe that if $y,z \in F(\sqrt{a})$ both satisfy the statement of Proposition \ref{prop:c4ext}, then $ F^{(2)}(\sqrt{y})=F^{(2)}(\sqrt{z}) $, and that a similar remark holds for $D_4$-extensions.  They also show that $F^{(3)}$ can be described as the Galois closure over $F$ of the compositum of all extensions $L$ of $F$ such that $F\subseteq L \subseteq F(2)$ and $[L:F]\leq 4$ \cite[Corollary 2.19]{MS2}.

Counting the number of $D_4$-extensions of a given field $F$ is therefore important in determining the W-group, $\Gal(F^{(3)}/F)$, and thereby gaining a better understanding of the absolute Galois group of $F$.  We now consider various examples of fields $F$ in order to illustrate several methods of counting these extensions.

\section{Local Fields}
\label{sec:local}

\subsection{Constructing extensions}
\label{sec:constructing ext}

We begin by considering, as the base field, the field of $p$-adic numbers, $\Q_p$.  One means of determining the number of $D_4$-extensions of $\mathbb{Q}_p$ is, of course, by actually constructing all such extensions.  Following \cite{Naito}, we outline this technique and then look at alternative methods of counting these extensions.  

\textbf{1. The case} $\mathbf{p\neq 2.}$   Any element $x\in \Q_p$ can be written uniquely in the form $x=up^n$, where $u$ is a unit in $\Z_p$.  For $p$ odd, $x=up^n\in \mathbb{Q}_p^\times$ is a square if and only if $n$ is even and the image of $u$ in the residue field $\Z_p/p\Z_p \cong \F_p$ is a square mod $p$.  Hence, $\Q_p^\times/(\Q_p^\times)^2\cong C_2\times C_2$ with representatives $\{1,p,u,up\}$ where $(\frac{u}{p})=-1$. 

The lattice of subfields of a $D_4$-extension $L/\Q_p$ is shown in the following diagram.

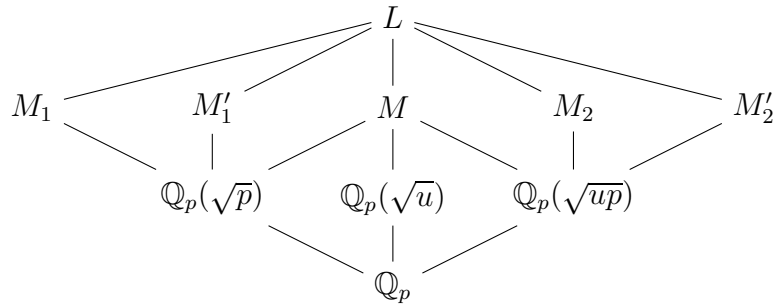
\begin{center}
\begin{tikzpicture}[scale=1.2]
\node (f) at (0,0) {$\Q_p$};
\node (f1) at (-2,1) {$\Q_p(\sqrt{p})$};
\node (f2) at (0,1) {$\Q_p(\sqrt{u})$};
\node (f3) at (2,1) {$\Q_p(\sqrt{up})$};
\node (k1) at (-4,2) {$M_1$};
\node (k2) at (-2,2) {$M_1^\prime$};
\node (k3) at (0,2) {$M$};
\node (k4) at (2,2) {$M_2$};
\node (k5) at (4,2) {$M_2^\prime$};
\node (k) at (0,3) {$L$};
\path [-]
(f) edge (f1)
(f) edge (f2)
(f) edge (f3)
(f1) edge (k1)
(f1) edge (k2)
(f1) edge (k3)
(f2) edge (k3)
(f3) edge (k3)
(f3) edge (k4)
(f3) edge (k5)
(k1) edge (k)
(k2) edge (k)
(k3) edge (k)
(k4) edge (k)
(k5) edge (k);
\end{tikzpicture}
\end{center}

The three quadratic extensions of $\mathbb{Q}_p$ are $\mathbb{Q}_p(\sqrt{p})$, $\mathbb{Q}_p(\sqrt{u})$ and $\mathbb{Q}_p(\sqrt{up})$ and $L/\mathbb{Q}_p$ has four intermediate fields of degree 4 which are not Galois over $\mathbb{Q}_p$.  These are the extensions labelled $ M_1, M^\prime_1, M_2, M^\prime_2 $ in the above lattice diagram.  For each $n\in \N$, any given local field has exactly one unramified extension of degree $n$.  Since $\Q_p(\sqrt{p})/Q_p$ and $\Q_p(\sqrt{up})/Q_p$ are ramified, $\Q_p(\sqrt{u})/Q_p$ is unramified.  Hence, $M/\Q_p(\sqrt{p})$ and $M/\Q_p(\sqrt{up})$ are also unramified.  So $M_i, M_i^{\prime},\ i=1,2$ are totally, and since $(p,4)=1$, tamely ramified extensions of $\Q_p$.  By Serre's mass formula, $\mathbb{Q}_p$ has exactly four totally and tamely ramified extensions of degree 4.  One such extension, say $M_1$, is $\Q_p(\sqrt[4]{p})/\Q_p$.  Hence $L=\Q_p(\sqrt{u}, \sqrt[4]{p})/\Q_p$. 

If $ p\equiv 1$ mod 4, then $\mathbb{Q}_p$ contains the $4^{th}$ roots of unity, so $\mathbb{Q}_p(\sqrt[4]{p})/\mathbb{Q}_p$ is a Galois extension of degree 4.  Hence $\mathbb{Q}_p$ can have no $D_4$-extension in this case.

If $ p\equiv 3$ mod 4, then $ -1 $ is not a square in $ \Q_p $, so $\Q_p(\sqrt[4]{p})/\Q_p$ is not  Galois.  In this case we see that $ \Q_p(\sqrt{-1}, \sqrt[4]{p}) $ is a $D_4$-extension of $\Q_p$.  

\textbf{2. The case} $\mathbf{p= 2.}$  We now consider the field of 2-adic numbers, $\Q_2$.  Let $L/\mathbb{Q}_2$ be a Galois extension of degree 8.  The Galois group of $L$, Gal$(L/\mathbb{Q}_2)\cong D_4$ if and only if $L$ contains an intermediate field of degree 4 which is not Galois over $\Q_2$.  Hence, in order to determine the $D_4$-extensions of $\Q_2$, it is sufficient to construct all quadratic extensions of $K_i$ which are not Galois over $\Q_2$, where $K_i$ is a quadratic extension of $\Q_2$.

The lattice of subfields of a $D_4$-extension $L/\Q_2$ is shown below.  We denote by $K$ the quadratic extension of $\Q_2$ for which $L/K$ is cyclic of degree 4.  The other two quadratic extensions of $\Q_2$ in $L$ are denoted $K_1$ and $K_2$.  For $i=1,2$, $M_i$ and $M_i^\prime$ are the quadratic extensions of $K_i$ in $L$ which are not Galois over $\Q_2$. 

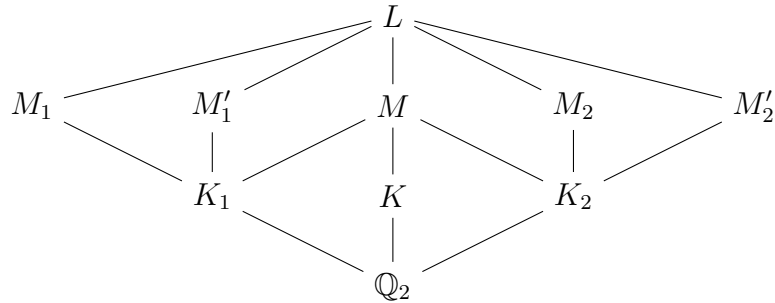
\begin{center}
\begin{tikzpicture}[scale=1.2]
\node (f) at (0,0) {$\Q_2$};
\node (f1) at (-2,1) {$K_1$};
\node (f2) at (0,1) {$K$};
\node (f3) at (2,1) {$K_2$};
\node (k1) at (-4,2) {$M_1$};
\node (k2) at (-2,2) {$M_1^\prime$};
\node (k3) at (0,2) {$M$};
\node (k4) at (2,2) {$M_2$};
\node (k5) at (4,2) {$M_2^\prime$};
\node (k) at (0,3) {$L$};
\path [-]
(f) edge (f1)
(f) edge (f2)
(f) edge (f3)
(f1) edge (k1)
(f1) edge (k2)
(f1) edge (k3)
(f2) edge (k3)
(f3) edge (k3)
(f3) edge (k4)
(f3) edge (k5)
(k1) edge (k)
(k2) edge (k)
(k3) edge (k)
(k4) edge (k)
(k5) edge (k);
\end{tikzpicture}
\end{center}

Let $\sigma$ be the generator of the Galois group of $K_i/\mathbb{Q}_2$.  Then $M_i= K_i(\sqrt{\alpha})$ for an $\alpha \in K_i ^\times$ such that $ \alpha^\sigma/\alpha \notin (K_i^\times)^2 $ and we have $ M_i^\prime =K_i(\sqrt{\alpha^\sigma}),\  L=K_i(\sqrt{\alpha}, \sqrt{\alpha^\sigma})  $ and $ M=K_i(\sqrt{\alpha\alpha^\sigma})$.  So we consider a system of representatives of the square class group of $ K_i $ and take all pairs $ (\alpha, \alpha^\sigma) $ of the system such that $ \alpha $ and $ \alpha^\sigma $ are independent modulo squares, thereby obtaining all $D_4$-extensions $L/\mathbb{Q}_2$.

An element $ x=u2^n \in \mathbb{Q}_2^\times $, where $u$ is a unit in $\Z_2$, is a square if and only if $n$ is even and $u\equiv 1$ mod 8. In the group $U$ of units of $\mathbb{Z}_2$, we have $ U=\{\pm 1 \}\times U_2 $, $ U_2\cong \mathbb{Z}_2$ and the set of squares in $ U_2 $ is $ U_3= \{ a\in \mathbb{Z}_2 \mid a\equiv 1$ mod $ 2^3 \} $.  Then $ U/U_3 \cong C_2\times C_2 $ with representatives $ \{\pm 1, \pm 5\} $ and $\Q_2^\times/(\Q_2^\times)^2\cong C_2\times C_2 \times C_2$ with representatives $\{\pm 1, \pm 5, \pm 2, \pm 10\}$, so there are exactly seven quadratic extensions of $ \mathbb{Q}_2 $.  Naito considers all cases and thereby constructs 18 $D_4$-extensions of $\Q_2$.

\subsection{A group-theoretic approach}

In cases in which the W-group of a field $F$ is known, this can provide a group-theoretic alternative to the direct construction method for determining the number of $D_4$-extensions of $F$.   The following proposition provides an illustration. 

\begin{prop}
\label{prop:D4Qp}
Let p be an odd prime.  Then $\Q_p$ has a $D_4$-extension if and only if $ p\equiv 3$ mod 4, and this extension is unique.
\end{prop}

\begin{proof}
Let $G=G_{\Q_p}(2)$.  By \cite[Proposition 2.1]{MS2}, the W-group of $\Q_p$, $\Gal(\Q_p^{(3)}/\Q_p) \cong G/G^4[G^2,G] =G^{[3]}$ and by \cite[Corollary 2.18]{MS2}, $\Q_p^{(3)}$ is the compositum of all quadratic, $C_4$- and $D_4$-extensions of $\Q_p$.  

$\Q_p$ has no $D_4$-extension if $ p\equiv 1$ mod 4, since in this case \cite[Example 4.2]{MS2} shows that $G^{[3]}\cong C_4\times C_4$, which has no quotient isomorphic to $D_4$.

If $ p\equiv 3$ mod 4, then $\{1,p,-1,-p\}$ is a set of representatives of $\Q_p^\times/(\Q_p^\times)^2$ and \cite[Example 4.3]{MS2} shows that $G^{[3]}=\langle\sigma_p, \sigma_{-1} \mid [\sigma_p,\sigma_{-1}]=\sigma_p^2 \rangle \cong C_4 \rtimes C_4$, with the semidirect product action given by $\sigma_{-1}^{-1} \sigma_p \sigma_{-1}= \sigma_p^{-1}$.  We have the group extension
\[
\begin{tikzcd}
1 \arrow{r} &\langle\sigma_p^2, \sigma_{-1}^2\rangle \cong C_2\times C_2 \arrow{r} &G^{[3]} \arrow{r} &\Gal(\Q_p^{(2)}/\Q_p) = \langle\bar{\sigma}_p, \bar{\sigma}_{-1}\rangle \arrow{r} &1.
\end{tikzcd}
\] 
So the existence of a $D_4$-extension of $\Q_p$ is equivalent to the existence of a subgroup $H \subset \langle\sigma_p^2, \sigma_{-1}^2\rangle$ such that  
\[
\begin{tikzcd}
1 \arrow{r} &H \cong C_2 \arrow{r} &G^{[3]} \arrow{r} &D_4\cong C_4\rtimes C_2 \arrow{r} &1
\end{tikzcd}
\] 
and
\[
\begin{tikzcd}
1 \arrow{r} &\langle\sigma_p^2, \sigma_{-1}^2\rangle /H \arrow{r} &D_4 \arrow{r} &\langle\bar{\sigma}_p, \bar{\sigma}_{-1}\rangle \arrow{r} &1
\end{tikzcd}
\] 
are group extensions.

Since $D_4$ is non-abelian, $\sigma_p^2=[\sigma_p,\sigma_{-1}]\neq 1$ in $D_4$, and since $D_4$ must have one generator of order 2, $\sigma_p^2 \sigma_{-1}^2 \neq 1$ in $D_4$.  Hence, the only possibility is $H=\langle\sigma_{-1}^2\rangle$, so $\Q_p$ has exactly one $D_4$-extension.
\end{proof}

\subsection{Quaternion algebras}
\label{sec:quat alg Qp}

Often, of course, one is dealing with a field $F$ for which the W-group is not known and the goal of counting $D_4$-extensions of $F$ may be to shed light on the structure of that group.  We now describe a technique of enumerating these extensions based on the theory of quaternion algebras, using the example of the field $\Q_p$.  

Recall that for $p$ odd, $\Q_p^\times/(\Q_p^\times)^2\cong C_2\times C_2$ with representatives $\{1,p,u,up\}$ where $u$ is not a square mod $p$.  So we consider the quaternion algebras $ (a,b)_{\Q_p} $ where $a,b \in \{1,p,u,up\} $ are independent modulo squares.  By Proposition \ref{prop:d4ext}, there exists a $D_4^{a,b}$-extension of $\Q_p$ if and only if $ (a,b)_{\Q_p}=1 \in \Br_2(\Q_p) $.

If $ p\equiv 1$ mod 4, then by \cite[Theorem VI.2.2]{Lam}, $(p,u)_{\Q_p}$ is a division algebra.  Since -1 is a square in $\Q_p$, we have $(p,u)_{\Q_p} \cong (p,up)_{\Q_p} \cong (u,up)_{\Q_p}$, so $\Q_p$ has no $D_4$-extension.   

If $ p\equiv 3$ mod 4, we can take $u=-1$.  In this case $(p,-p)_{\Q_p}$ splits and by the non-degeneracy of the Hilbert symbol, we see that $ \{p,-p\} $ is the only choice for $ \{a,b\} $.  Now suppose $L_1=F(\sqrt{a},\sqrt{b})(\sqrt{y})$ and $L_2=F(\sqrt{a},\sqrt{b})(\sqrt{z})$ are two $D_4^{a,b}$-extensions of a field $F$.  It follows from Proposition \ref{prop:d4ext} that there exists an $f\in F$ such that $z=fy$.  When $F=\Q_p$ with $ p\equiv 3$ mod 4, we have $\sqrt{f}\in \Q_p^{(2)}=\Q_p(\sqrt{p}, \sqrt{-p})$.  So \[L_2=\Q_p(\sqrt{p},\sqrt{-p})(\sqrt{z})=\Q_p(\sqrt{p},\sqrt{-p})(\sqrt{fy})=\Q_p(\sqrt{p},\sqrt{-p})(\sqrt{y})=L_1.
\]   
Hence there exists only one $D_4$-extension $L/\Q_p$ in this case.

The diagram below shows the lattice of subfields of this extension.

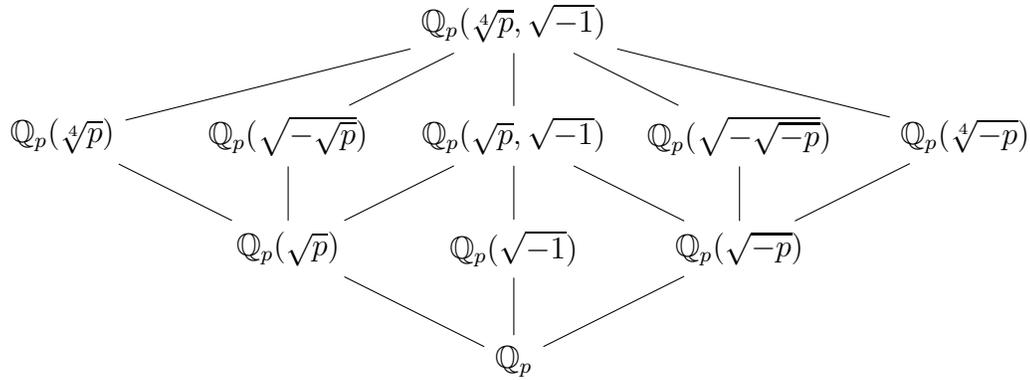
\begin{center}
\begin{tikzpicture}[scale=1.5]
\node (f) at (0,0) {$\Q_p$};
\node (f1) at (-2,1) {$\Q_p(\sqrt{p})$};
\node (f2) at (0,1) {$\Q_p(\sqrt{-1})$};
\node (f3) at (2,1) {$\Q_p(\sqrt{-p})$};
\node (k1) at (-4,2) {$\Q_p(\sqrt[4]{p})$};
\node (k2) at (-2,2) {$\Q_p(\sqrt{-\sqrt{p}})$};
\node (k3) at (0,2) {$\Q_p(\sqrt{p},\sqrt{-1})$};
\node (k4) at (2,2) {$\Q_p(\sqrt{-\sqrt{-p}})$};
\node (k5) at (4,2) {$\Q_p(\sqrt[4]{-p})$};
\node (k) at (0,3) {$\Q_p(\sqrt[4]{p},\sqrt{-1})$};
\path [-]
(f) edge (f1)
(f) edge (f2)
(f) edge (f3)
(f1) edge (k1)
(f1) edge (k2)
(f1) edge (k3)
(f2) edge (k3)
(f3) edge (k3)
(f3) edge (k4)
(f3) edge (k5)
(k1) edge (k)
(k2) edge (k)
(k3) edge (k)
(k4) edge (k)
(k5) edge (k);
\end{tikzpicture}
\end{center}

Recall that for $p=2$, $\{\pm 1, \pm 5, \pm 2, \pm 10\}$ is a set of representatives of $\Q_2^\times/(\Q_2^\times)^2$.  Then by Proposition \ref{prop:d4ext}, there exists a $D_4^{a,b}$-extension of $\Q_2$ for each $\{a,b\}\subset \{-1, \pm 5, \pm 2, \pm 10\}$ such that $a\neq b$ and the quaternion algebra $(a,b)_{\Q_2}=1\in \Br_2(\Q_2)\cong \{\pm 1\}$.  These are the following pairs: 
\[
\begin{aligned}
&\{-1,2\},\ \{-1,5\},\ \{-1,10\},\\
&\{2,-2\},\ \{5,-5\},\ \{10,-10\},\\
&\{-2,-5\},\ \{-2,-10\},\ \{5,-10\}.\\  
\end{aligned}
\]
Consider, for example, the pair $\{2,-2\}$.  This pair yields a $D_4^{2,-2}$-extension $L/\Q_2$ such that $\Q_2(\sqrt{2}, \sqrt{-2})\subset L$ and $\Gal(L/\Q_2(\sqrt{-1}))\cong C_4$.  

However, in this case, $L$ is not uniquely determined.  We have $y_1:=5\sqrt{-2}\in \Q_2(\sqrt{-2})$ with $Ny_1=2\cdot 5^2=2\in \Q_2^\times/(\Q_2^\times)^2$, so $L_1:=\Q_2(\sqrt{-2},\sqrt{2})(\sqrt{5\sqrt{-2}})$ is a $D_4^{-2,2}$-extension of $\Q_2$.  Similarly, $y_2:=-\sqrt{-2}\in \Q_2(\sqrt{-2})$ has $Ny_2=2$, so $L_2:=\Q_2(\sqrt{-2},\sqrt{2})(\sqrt{-\sqrt{-2}})=\Q_2(\sqrt{-2},\sqrt{2})(\sqrt[4]{-2})$ is a $D_4^{-2,2}$-extension of $\Q_2$.  

Suppose $L_1=L_2$.  Then $\sqrt{5}=(\sqrt{-2})^{-1}(\sqrt[4]{-2})(\sqrt{5\sqrt{-2}})\in L_1=L_2$, implying that this dihedral extension of degree 8 contains $\Q_2(\sqrt{-2},\sqrt{2},\sqrt{5})/\Q_2$, an elementary abelian extension of degree 8.  Hence $L_1\neq L_2$.  Now let $f\in \Q_2\setminus \Q_2^2$.  If $f\in \{ 5(\Q_2^\times)^2 \cup -5(\Q_2^\times)^2 \cup 10(\Q_2^\times)^2 \cup -10(\Q_2^\times)^2  \}$, then 
\[
\begin{aligned}
\Q_2(\sqrt{2}, \sqrt{-2})(\sqrt{fy_1})&=\Q_2(\sqrt{2}, \sqrt{-2})(\sqrt{-\sqrt{-2}})=L_2;\\
\Q_2(\sqrt{2}, \sqrt{-2})(\sqrt{fy_2})&=\Q_2(\sqrt{2}, \sqrt{-2})(\sqrt{5\sqrt{-2}})=L_1.
\end{aligned}
\]
Otherwise, $f\in \{ 2(\Q_2^\times)^2 \cup -2(\Q_2^\times)^2 \cup -(\Q_2^\times)^2$, in which case
\[
\begin{aligned}
\Q_2(\sqrt{2}, \sqrt{-2})(\sqrt{fy_1})&=\Q_2(\sqrt{2}, \sqrt{-2})(\sqrt{5\sqrt{-2}})=L_1;\\
\Q_2(\sqrt{2}, \sqrt{-2})(\sqrt{fy_2})&=\Q_2(\sqrt{2}, \sqrt{-2})(\sqrt{-\sqrt{-2}})=L_2.
\end{aligned}
\]

Hence, $\Q_2$ has exactly two $D_4^{-2,2}$-extensions.  An analogous argument shows that there are exactly two $D_4^{a,b}$-extensions for each of the nine pairs $\{a,b\}$ such that $(a,b)_{\Q_2}=1$, so $\Q_2$ has 18 $D_4$-extensions.

\subsection{Complex characters and M\"{o}bius functions}
\label{sec:mobius}

In this section we turn to the more general case of finite Galois $p$-extensions of a local field $K$, where $K$ is a finite extension of the $p$-adic numbers, $\Q_p$.  Such a field is sometimes referred to as a local number field.  We describe an interesting method of counting these extensions using M\"{o}bius functions and complex characters, due to M. Yamagishi \cite{yam}. 

Let $K$ be a field and $G$ a finite group.  Let
\[
\nu(K,G):=|\{G\text{-extensions of}\ K\}|.
\]

Let $\sG$ be a fixed group.  Define
\[
\begin{aligned}
\alpha_{\sG}(G):&=|\{\text{homomorphisms}\ \sG\to G \}|\\
\beta_{\sG}(G):&=|\{\text{surjective homomorphisms}\ \sG\surj G\}|.
\end{aligned}
\]
For any subgroup $H$ of $G$, assume that $\alpha_{\sG}(H)$ is finite.  Then
\[
\alpha_{\sG}(G)=\sum_{H\leq G}\beta_{\sG}(H),
\]
so by the M\"{o}bius inversion formula
\[
\beta_{\sG}(G)=\sum_{H\leq G}\mu_G(H)\alpha_{\sG}(H),
\]
where $\mu_G$ is the M\"{o}bius function on the partially ordered set of all subgroups of $G$.

There is a 1-1 correspondence between the set of $G$-extensions of $K$ and the set of surjective homomorphisms from the absolute Galois group of $K$, $G_K\surj G$, modulo automorphisms of $G$.  Let $p$ be a prime.  If $G$ is a $p$-group, then $G_K$ can be replaced by $G_K(p)$, the Galois group of the maximal $p$-extension of $K$.  Also, if $K$ is a finite extension of the field $\Q_p$ of $p$-adic numbers, it is well-known that $K$ has only  finitely many algebraic extensions of given degree (inside a fixed algebraic closure of $K$).  Hence we have

\begin{thm}
\label{thm:nuKG} 
Let p be a prime, K a finite extension of $\Q_p$, and G a finite p-group.  Let the notation be as above, with $\sG=G_K(p)$.  Then
\[
\nu(K,G)=\frac{1}{|\rm{Aut}(G)|}\sum_{H\leq G}\mu_G(H)\alpha_{\sG}(H).
\] 
\end{thm}

\begin{proof}
See \cite[Theorem 1]{yam}.
\end{proof}

\begin{ex}
\label{ex:shaf}

Let $p$ be a prime.  Suppose that $G$ is a $p$-group and $K$ is a finite extension of $\Q_p$ of degree $n=[K:\Q_p]$ which does not contain a primitive $p$-th root of unity.  I. R. Shafarevich \cite{Sha} showed that the Galois group $G_K(p)$ of the maximal $p$-extension of $K$, is a free pro-$p$-group of rank $n+1$ and he gave an explicit formula for $\nu(K,G)$ in this case:
\[
\nu(K,G)=\frac{1}{|\rm{Aut}(G)|}(\frac{|G|}{p^d})^{n+1}\prod_{i=0}^{d-1}(p^{n+1}-p^i),
\]
where $d$ is the minimal number of generators of $G$.

This formula can also be obtained from Theorem \ref{thm:nuKG} as follows.  Let $\sG=G_K(p)$.  Then $\alpha_{\sG}(H)=|H|^{n+1}$ for any $p$-group $H$.  By Lemma \ref{lem:muG}, we need consider only those subgroups $H\leq G$ such that $G^p[G,G]\leq H$.  The number of such subgroups is $\binom{d}{i}_p$, where
\[
i=\dim_{\F_p}\frac{H}{G^p[G,G]}.
\]
Since $[G:H]=p^{d-i}$, Lemma \ref{lem:muG} together with the identity $\binom{d}{i}_p=\binom{d}{d-i}_p$ gives
\[
\begin{aligned}
\sum_{H\leq G}\mu_G(H)\alpha_{\sG}(H)&=\sum_{G^p[G,G]\leq H\leq G}\mu_G(H)\alpha_{\sG}(H)\\
&=\sum_{i=0}^{d}\binom{d}{i}_p(-1)^{d-i}p^{\frac{1}{2}(d-i)(d-i-1)}\frac{|G|^{n+1}}{p^{(d-i)(n+1)}}\\
&=\sum_{i=0}^{d}\binom{d}{d-i}_p(-1)^{d-i}p^{\frac{1}{2}(d-i)(d-i-1)}\frac{|G|^{n+1}}{p^{(d-i)(n+1)}}\\
&=\sum_{i=0}^{d}\binom{d}{i}_p(-1)^ip^{\frac{1}{2}i(i-1)}\frac{|G|^{n+1}}{p^{i(n+1)}}\\
&=\frac{|G|^{n+1}}{p^{d(n+1)}}\sum_{i=0}^{d}\binom{d}{i}_p(-1)^ip^{\frac{1}{2}i(i-1)}(p^{n+1})^{d-i}
\end{aligned}
\]
Induction on $d$ together with the identity $\binom{d+1}{i}_p=\binom{d}{i-1}_p+p^i\binom{d}{i}_p$ shows that
\[
\prod_{i=0}^{d-1}(p^{n+1}-p^i)=\sum_{i=0}^{d}\binom{d}{i}_p(-1)^ip^{\frac{1}{2}i(i-1)}(p^{n+1})^{d-i},
\]
and the result follows. 
\end{ex}

If $\sG$ is finitely presented as:
\[
\sG=\langle x_1,x_2,\ldots,x_n\mid r_1=r_2=\cdots=r_m=1 \rangle,
\]
where each $r_i=r_i(x_1,x_2,\ldots,x_n)$ is a finite word in the symbols $x_1,x_2,\ldots,x_n$, then
\[
\alpha_{\sG}(G)=|\{(g_1,g_2,\ldots,g_n)\in G^n \mid r_i(g_1,g_2,\ldots,g_n)=1,\ i=1,2,\ldots,m \}|.
\]
In particular, $\alpha_{\sG}(G)$ and $\beta_{\sG}(G)$ are finite.  By the column orthogonality relations of irreducible characters
\[
\begin{aligned}
r_i(g_1,g_2,\ldots,g_n)=1&\Longleftrightarrow \sum_{\chi}\chi(1)\chi(r_i(g_1,g_2,\ldots,g_n))=|G|,\\
r_i(g_1,g_2,\ldots,g_n)\neq1&\Longleftrightarrow \sum_{\chi}\chi(1)\chi(r_i(g_1,g_2,\ldots,g_n))=0,
\end{aligned}
\] 
where $\chi$ runs over all irreducible complex characters of $G$.  Hence
\begin{equation}
\label{eq:alphaG}
\alpha_{\sG}(G)=\frac{1}{|G|^m}\sum_{(g_1,g_2,\ldots,g_n)\in G^n}\prod_{i=1}^{m}\sum_{\chi}\chi(1)\chi(r_i(g_1,g_2,\ldots,g_n)).
\end{equation}

Now consider the case in which $K$ is a finite extension of the $p$-adic field $\Q_p$ of degree $n$ and assume that $K$ contains a primitive $p$-th root of unity $\zeta_p$.  Then  the Galois group $G_K(p)$ of the maximal $p$-extension of $K$ is a Demushkin group of rank $n+2$.  

\begin{defn}
\label{def:Demushkin}
A \textit{Demushkin group} is a pro-$p$ group $G$ which satisfies the following three conditions:
\begin{enumerate}
\item $\dim_{\F_p} H^1(G,\F_p)<\infty,$ 
\item $\dim_{\F_p} H^2(G,\F_p)=1,$
\item  the cup product $H^1(G,\F_p)\times H^1(G,\F_p)\to H^2(G,\F_p)$ is a non-degenerate bilinear form.
\end{enumerate} 
\end{defn}

From the first two conditions, we see that a Demushkin group is a finitely generated pro-$p$ group having a single relation among a minimal set of generators.  These groups have been completely classified by S.P. Demushkin, J.-P. Serre, and J. Labute (see \cite{La1}).  Let $q$ be the maximal power of $p$ such that $\zeta_q \in K$.  By the classification theorem of Demushkin groups \cite{La1}, there exist generators $x_1,x_2,\ldots,x_{n+2}$ such that the unique relation $r$ takes one of the following forms:
\begin{enumerate}
\item[(i)] if $q\not=2$ ($n$ is even in this case), then 
\begin{equation}
 \label{eq:r1} r=x_1^q[x_1,x_2][x_3,x_4]\cdots[x_{n+1},x_{n+2}];
\end{equation}
\item[(ii)] if $q=2$ and $n$ is odd, then 
\begin{equation}
 \label{eq:r2} r=x_1^2x_2^4[x_2,x_3][x_4,x_5]\cdots[x_{n+1},x_{n+2}];
 \end{equation}
\item[(iii)] if $q=2$ and $n$ is even, then either
\begin{equation}
 \label{eq:r3} r=x_1^{2+2^f}[x_1,x_2][x_3,x_4]\cdots[x_{n+1},x_{n+2}],\\
\end{equation}
where $f$ is an integer $\geq 2$ or $\infty$, or 
\begin{equation}
\label{eq:r4} r=x_1^2[x_1,x_2]x_3^{2^f}[x_3,x_4]\cdots[x_{n+1},x_{n+2}],
\end{equation}
where $f$ is an integer $\geq 2$.
\end{enumerate} 

Substituting the explicit forms of $r$ into equation \ref{eq:alphaG} and using the identity
\[
\sum_{b,c\in G}\chi(a[b,c])=(\frac{|G|}{\chi(1)})^2\chi(a),\ \text{for all}\ a\in G,
\] 
Yamagishi proves

\begin{lem}
\label{lem:alphaG}
Let p be a prime, K be a finite extension of $\Q_p$ containing a primitive p-th root of unity $\zeta_p$, and G a finite p-group.  Let $\sG=G_K(p)$.  Then
\[
\alpha_{\sG}(G)=
\begin{cases}
|G|^n\sum_{\chi} \frac{1}{\chi(1)^n}\sum_{g\in G}\chi(g^{q-1})\chi(g)&\ (\text{Case}\ \ref{eq:r1})\\
|G|^{n-1}\sum_{\chi} \frac{1}{\chi(1)^{n-1}}\sum_{g,h\in G}\chi(g^2h^3)\chi(h)&\ (\text{Case}\ \ref{eq:r2})\\
|G|^n\sum_{\chi} \frac{1}{\chi(1)^n}\sum_{g\in G}\chi(g^{2^f+1})\chi(g)&\ (\text{Case}\ \ref{eq:r3})\\
|G|^{n-1}\sum_{\chi} \frac{1}{\chi(1)^{n-1}}\sum_{g,h\in G}\chi(g)\chi(gh^{2^f-1})\chi(h)&\ (\text{Case}\ \ref{eq:r4}),
\end{cases}
\]
where $n=[K:\Q_p]$, q is the maximal power of p such that $\zeta_q\in K$, and $\chi$ runs over all irreducible complex characters of G.
\end{lem} 

\begin{proof}
See \cite[Lemma 1.8]{yam}.
\end{proof}

Using this result, Yamagishi then derives a formula for $\nu(K,G)$ in the special cases in which $G$ is a non-abelian group of order $p^3$ or is a dihedral or generalized quaternion group of order $2^m$, $m\geq 3$. 

\begin{ex}
\label{ex:Q2}

Let $K$ be a finite extension of $\Q_2$ which does not contain a primitive 4-th root of unity and assume $n=[K:\Q_2]$ is odd.  Then $\sG:=G_K(2)$ is a Demushkin group with $q=2$ in which the unique relation among a minimal set of generators takes the form given in (\ref{eq:r2}).  Let $G=D_4$ be the dihedral group of order 8.
\[
D_4:=\langle r,s\mid r^4=s^2=1,\ srs=r^{-1} \rangle,
\]
with lattice of subgroups

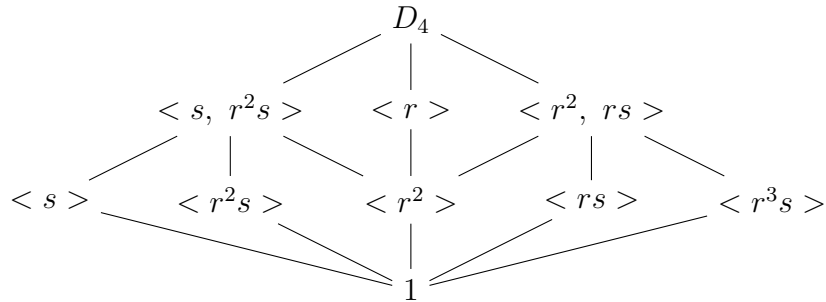
\begin{center}
\begin{tikzpicture}[scale=1.2]
\node (k) at (0,0) {$1$};
\node (k1) at (-4,1) {$<s>$};
\node (k2) at (-2,1) {$<r^2s>$};
\node (k3) at (0,1) {$<r^2>$};
\node (k4) at (2,1) {$<rs>$};
\node (k5) at (4,1) {$<r^3s>$};
\node (f1) at (-2,2) {$<s,\ r^2s>$};
\node (f2) at (0,2) {$<r>$};
\node (f3) at (2,2) {$<r^2,\ rs>$};
\node (f) at (0,3) {$D_4$};
\path [-]
(f) edge (f1)
(f) edge (f2)
(f) edge (f3)
(f1) edge (k1)
(f1) edge (k2)
(f1) edge (k3)
(f2) edge (k3)
(f3) edge (k3)
(f3) edge (k4)
(f3) edge (k5)
(k1) edge (k)
(k2) edge (k)
(k3) edge (k)
(k4) edge (k)
(k5) edge (k);
\end{tikzpicture}
\end{center}
The Frattini subgroup $\Phi(D_4)=<r^2>$, so for $H\leq G=D_4$, Lemma \ref{lem:muG} gives
\[
\mu_G(H)=
\begin{cases}
1\ &\text{if}\ H=D_4\\
-1\ &\text{if}\ H=<s,r^2s>,\ <r>\ \text{or} <r^2,rs>\\
2\ &\text{if}\ H=<r^2>\\
0\ &\text{otherwise}
\end{cases}
\]

The conjugacy classes of $D_4$ are $\{1\},\ \{r^2\},\ \{s,r^2s\},\ \{r,r^3\}$ and $\{rs,r^3s\}$.  The character table is shown below.

\begin{center}
\begin{tabular}{l|ccccc}
$D_4$&1&$s$&$rs$&$r$&$r^2$\\
\hline\\
$\chi_1$&1&1&1&1&1\\
$\chi_{\rho_2}$&1&-1&-1&1&1\\
$\chi_{\rho_3}$&1&1&-1&-1&1\\
$\chi_{\rho_4}$&1&-1&1&-1&1\\
$\chi_{\sigma_1}$&2&0&0&0&-2
\end{tabular}
\end{center}
From Lemma \ref{lem:alphaG} and the above character table, we obtain
\[
\alpha_{\sG}(D_4)=8^{n+1}(4+\frac{1}{2^n}).
\]
If $H\leq G$ is abelian, then in order to calculate $\alpha_{\sG}(H)$, one need consider only maps 
\[
\sG/[\sG, \sG]\cong \Z_2^{n+1}\times \Z_2/q\Z_2 \to H,
\]
which gives
\[
\alpha_{\sG}(H)=|H|^{n+1}\cdot |\{h\in H\mid h^q=1 \}|.
\]
Hence
\[
\alpha_{\sG}(H)=
\begin{cases}
4^{n+1}\cdot 4\ &\text{if}\ H=<s,r^2s>\ \text{or} <r^2,rs>\\
4^{n+1}\cdot 2\ &\text{if}\ H=<r>\\
2^{n+1}\cdot 2\ &\text{if}\ H=<r^2>.
\end{cases}
\]
So for $G=D_4$, Theorem \ref{thm:nuKG} gives
\[
\begin{aligned}
\nu(K,G)&=\frac{1}{|\rm{Aut}(G)|}\sum_{H\leq G}\mu_G(H)\alpha_{\sG}(H)\\
&=\frac{1}{2^3}(8^{n+1}(4+\frac{1}{2^n})-4^{n+1}\cdot 4-4^{n+1}\cdot 4-4^{n+1}\cdot 2+2(2^{n+1}\cdot 2))\\
&=2^n(2^{n+1}-1)^2.
\end{aligned}
\] 
In the case $K=\Q_2$, we once again obtain $\nu(\Q_2,D_4)=18$.
\end{ex}

\subsection{Cup products in cohomology} 
\label{sec:cup}

In this section, we again consider the case in which $K$ is a finite extension of the field $\Q_p$ of $p$-adic numbers and for a finite group $G$, we use the notation $\nu(K,G)$ to denote the number of $G$-extensions of $K$.  We describe a technique based on Galois cohomology to count the number of $D_4$-extensions of $K$.

In \cite{MT1} J. Min\'a\v{c} and N. D. T\^an use triple Massey products to compute $\nu(K,G)$ for $G=\U_4(\F_p)$, the group of unipotent four by four matrices over $\F_p$.  Closely following their approach, but using cup products instead of triple Massey products, a similar method can be employed to calculate $\nu(K,\U_3(\F_p))$.  

Let $G$ be a profinite group and $p$ a prime. We consider the finite field $\F_p$ as  a trivial discrete $G$-module. Let $\sC^\bullet=(C^\bullet(G,\F_p),\delta,\cup)$ be the differential graded algebra of inhomogeneous continuous cochains of $G$ with coefficients in $\F_p$ \cite[\S I.2]{NSW}.  The cohomology groups are written $H^i(G,\F_p)$. We denote by $Z^1(G,\F_p)$ the subgroup of $C^1(G,\F_p)$ consisting of all 1-cocycles. Since $G$ acts trivially on the coefficients $\F_p$, $Z^1(G,\F_p)=H^1(G,\F_p)={\rm Hom}(G,\F_p)$.
  
\begin{defn}
A {\it weak embedding problem} $\sE:=\sE(G,f\colon U\to \bar{U},\varphi\colon G\to \bar{U})$ for a profinite group $G$ is a diagram 
\[
\sE:=
\xymatrix
{
{} & G \ar[d]^{\varphi}\\
U \ar[r]^f & \bar U
}
\] 
consisting of  profinite groups $U$ and $\bar U$ and homomorphisms $\varphi \colon G\to \bar U$, $f\colon U\to \bar U$ with $f$ being surjective.  If $\varphi$ is also surjective, we call $\sE$ an {\it embedding problem}. 

A {\it weak solution} of $\sE$ is  a homomorphism $\psi\colon G\to U$ such that $f\psi=\varphi$.  If $\psi$ is surjective, the solution is said to be \textit{proper}.  We call $\sE$ a {\it finite} weak embedding problem if $U$ is finite. The {\it kernel} of $\sE$ is defined to be $M:=\ker(f)$.  We denote by ${\rm Sol}(\sE)$ the set of weak solutions of $\sE$. 
\end{defn}

\begin{ex}
A proper solution of the embedding problem
\[
\xymatrix
{
{} & G_F^{[3,2]} \ar[d]^{\varphi}\\
C_4 \ar[r]^f & C_2
}
\]
corresponds to a $C_4$-extension of $F$. 
\end{ex}

Suppose $\sE(G,f\colon U\to \bar{U},\varphi\colon G\to \bar{U})$ is a weak embedding problem with abelian kernel $M$. The conjugation action of $U$ on $M$ is trivial while restricting to $M\subseteq U$. Hence this induces a $\bar{U}$-module structure on $M$. We consider $M$ to be a $G$-module via $\varphi$ and the conjugation action of $\bar{U}$ on $M$. We denote this $G$-module by $M_\varphi$.

\begin{lem}
\label{lem:Sol}
Let $\sE(G,f,\varphi)$ be a weak embedding problem with finite abelian kernel $M$ which has a weak solution. Then ${\rm Sol}(\sE)$ is a principal homogeneous space over the group of 1-cocycles $Z^1(G,M_\varphi)$.

In particular, any weak solution $\theta$ of $\sE$ induces a bijection 
\[
{\rm Sol}(\sE) \simeq Z^1(G,M_\varphi).
\]
\end{lem}
\begin{proof} See \cite[Proposition 3.5.11]{NSW}.
\end{proof}

The group $\U_n(\F_p)$ of unipotent, $n\times n$ matrices over $\F_p$ is the multiplicative group of all upper-triangular $n\times n$ matrices over $\F_p$ which agree with the identity matrix along the diagonal.  Let $Z_n(\F_p)$ be the subgroup of $\U_n(\F_p)$ consisting of matrices with all off-diagonal entries being zero except at position $(1,n)$, together with the identity matrix.  Then $Z_n(\F_p)$ lies in the center of $\U_n(\F_p)$ and is isomorphic to the additive group of $\F_p$.  The quotient group $\bar\U_n(\F_p)= \U_n(\F_p)/Z_n(\F_p)$ can be identified with the group of all upper-triangular unipotent $n\times n$ matrices over $\F_p$ with the $(1,n)$ entry omitted.

A representation $\rho\colon G\to \U_n(\F_p)$ is given by a component array $\rho_{ij},\ 1\leq i\leq n,\ i<j\leq n$, of set maps $G\to \F_p$ which satisfy the identities
\[
\rho_{ij}(g_1g_2)=\rho_{ij}(g_1)+\rho_{ij}(g_2)+\sum_{k=i+1}^{j-1}\rho_{ik}(g_1)\rho_{kj}(g_2),\quad g_1,g_2 \in G.
\]
The maps of the form $\rho_{i,i+1}$, called the \textit{near-diagonal components} of $\rho$, are then group homomorphisms $G\to \F_p$, and hence cohomology classes in $H^1(G, \F_p)$.  Similarly, a representation $\rho\colon G\to \bar\U_n(\F_p)$ has near-diagonal components in $H^1(G, \F_p)$.

Dwyer \cite{Dwy} demonstrated a close connection between $n$-fold Massey products of elements in $H^1(G,\F_p)$ and representations $\rho\colon G\to \U_{n+1}(\F_p)$.  In particular, for the case $n=2$, if $\rho\colon G\to \bar\U_3(\F_p)$ is a group homomorphism given by the components $-\rho_1, -\rho_2$, it follows from \cite[Theorem 2.4]{Dwy} that $\rho$ can be lifted to a group homomorphism $G\to \U_3(\F_p)$ if and only if the cup product $\rho_1 \cup \rho_2 =0$ in $H^2(G,\F_p)$.  

\begin{lem}
\label{lem:linear independence} Let $G$ be a pro-$p$-group.
 Let $\chi_1,\ldots,\chi_n$ be elements in $H^1(G,\F_p)$. Then the homomorphism \[\varphi:=(\chi_1,\ldots,\chi_n)\colon G\to \F_p\times\cdots\times \F_p\] is surjective if and only if $\chi_1,\ldots,\chi_n$ are $\F_p$-linearly independent in $H^1(G,\F_p)$.
\end{lem}
\begin{proof}
We set $H:=\F_p\times\cdots\times \F_p$. Then $\varphi\colon G\to H$ is surjective if and only if the induced homomorphism $\varphi^*\colon H^1(H,\F_p)\to H^1(G,\F_p)$ is injective (\cite[Proposition 1.6.14 (ii)]{NSW}). We have an  (non-canonical) isomorphism
\[
H\to H^1(H,\F_p), a=(a_1,\ldots,a_n)\mapsto \chi_a,
\]
where $\chi_a$ is defined by $\chi_a(h_1,\ldots,h_n)=\sum_{i=1}^n a_ih_i$. Then for each $a=(a_1,\ldots,a_n)\in H$, 
\[
(\varphi^*(\chi_a))(g)=\sum_{i=1}^n a_i\chi_i(g),\; \forall g\in G.
\]
Therefore $\varphi^*$ is injective if and only if $\chi_1,\ldots,\chi_n$ are $\F_p$-linearly independent. 
\end{proof}

Now consider the following exact sequence of finite groups
\begin{equation}
\label{eq:U3 exact sequence}
1 \longrightarrow \F_p \longrightarrow \U_3(\F_p) \xrightarrow{(a_{12},a_{23})} \F_p\times \F_p \longrightarrow 1,
\end{equation}
where $a_{ij}\colon\U_3(\F_p)\to\F_p$ is the map sending a matrix to its $(i,j)$-coefficient.

Let ${\rm CP}(G,\F_p)$ be the set of $(x,y)\in H^1(G,\F_p)\times H^1(G,\F_p)$ such that $x\cup y=0$ and $x,\ y$ are $\F_p$-linearly independent in $H^1(G,\F_p)$.  For two profinite groups $G$ and $H$, let  ${\rm Epi}(G,H)$ be the set of all continuous surjective homomorphisms from $G$ to $H$. 

\begin{prop}
\label{prop:counting Epi}
Let the notation be as above. Assume that both ${\rm CP}(G,\F_p)$ and $Z^1(G,\F_p)$ are finite.  
Then 
\[
|{\rm Epi}(G,\U_3(\F_p))|=\sum_{\varphi\in {\rm CP}(G,\F_p)} |Z^1(G,\F_p)|. 
\]
 \end{prop}
 \begin{proof}
 Let $\varphi =(x,y)\in {\rm CP}(G,\F_p)$.  Then $\varphi$ induces a homomorphism $G\to \bar\U_3(\F_p) \cong \F_p \times \F_p,\ g\mapsto (x(g),y(g))$ which, by Lemma~\ref{lem:linear independence} is surjective, since $x,y$ are $\F_p$-linearly independent in $H^1(G,\F_p)$.  This gives an embedding problem
 \[
 \xymatrix
 {
 {} & G \ar[d]^{\varphi}\\
 \U_3(\F_p) \ar[r]^f & \bar\U_3(\F_p)
 }
 \] 
 with kernel $\F_p$ which, by \cite[Theorem 2.4]{Dwy}, has a solution since $x\cup y=0$.  Hence by Lemma~\ref{lem:Sol}, the embedding problem has $ |Z^1(G,\F_p)| $ solutions.  Since the kernel has order $p$, each solution is a proper solution, so the result follows.
 \end{proof}

\begin{lem}
\label{lem:G mod Zassenhaus subgroup} Let $G$ be a profinite group, and let $G(p)$ be its maximal pro-$p$-quotient. Then ${\rm Epi}(G,\U_3(\F_p))\cong {\rm Epi}(G(p),\U_3(\F_p))$.
\end{lem}
\begin{proof}
This follows from the fact that $\U_3(\F_p)$ is a finite $p$-group.
\end{proof}

Assume that $K$ is a finite extension of $\Q_p$.  Recall that if $K$ contains a primitive $p$th root of unity, then the group $G:=G_K(p)$ is a Demushkin group, which is a pro-$p$ group having the following properties:
\begin{enumerate}
\item $\dim_{\F_p} H^1(G,\F_p)<\infty,$ 
\item $\dim_{\F_p} H^2(G,\F_p)=1,$
\item  the cup product $H^1(G,\F_p)\times H^1(G,\F_p)\to H^2(G,\F_p)$ is a non-degenerate bilinear form.
\end{enumerate}

Note that since $a\cup b =-b\cup a$ for $a,b\in H^1(G,\F_p)$, the bilinear form  $(\cdot,\cdot)\colon H^1(G,\F_p)\times H^1(G,\F_p)\to H^2(G,\F_p) \cong \F_p$ induced by the cup product is skew-symmetric. 
Let $d(G) = \dim_{\F_p} H^1(G, \F_p)$.  Then $G$ has a minimal presentation $G=S/R$ where $S$ is a free pro-$p$-group of rank $d = d(G)$ on generators $x_1,x_2,\ldots,x_d$, and $R= \langle r\rangle$ is the closed normal subgroup of $S$ generated by an element $r\in S^p[S,S]$.   Let $q=q(G)$ be the maximal power of $p$ such that $\zeta_q \in K$ (and by convention, $p^\infty=0$).  Recall from section \ref{sec:mobius} that the relation $r$ takes one of the following forms:
\begin{enumerate}
\item[(i)] if $q\neq 2$ ($d$ is even in this case), then 
\begin{equation}
 \label{eq:D1} r=x_1^q[x_1,x_2][x_3,x_4]\cdots[x_{d-1},x_d];
\end{equation}
\item[(ii)] if $q=2$ and $d$ is odd, then 
\begin{equation}
 \label{eq:D2} r=x_1^2x_2^{2^f}[x_2,x_3][x_4,x_5]\cdots[x_{d-1},x_d],
 \end{equation}
 where $f$ is an integer $\geq 2$ or $\infty$;
\item[(iii)] if $q=2$ and $d$ is even, then either
\begin{equation}
 \label{eq:D3} r=x_1^{2+2^f}[x_1,x_2][x_3,x_4]\cdots[x_{d-1},x_d],\\
\end{equation}
where $f$ is an integer $\geq 2$ or $\infty$, or 
\begin{equation}
\label{eq:D4} r=x_1^2[x_1,x_2]x_3^{2^f}[x_3,x_4]\cdots[x_{d-1},x_d],
\end{equation}
where $f$ is an integer $\geq 2$.
\end{enumerate} 

\begin{prop}
\label{prop:Demushkin-linear algebra}
 Let $G$ be a Demushkin group and let 
\[
(\cdot,\cdot)\colon H^1(G,\F_p)\times H^1(G,\F_p)\stackrel{\cup}{\to} H^2(G,\F_p)\cong \F_p
\]
be the non-degenerate skew-symmetric bilinear form induced by the cup product.  Let $d=d(G)$ and $q=q(G)$.  
\begin{enumerate}
\item If $q\not=2$, there exists an $\F_p$-basis $v_1,v_2,\ldots,v_d$ of $H^1(G,\F_p)$ such that $(v_i,v_i)=0$ for every $1\leq i\leq d$.
\item If $q=2$, there exists an $\F_p$-basis $v_1,v_2,\ldots,v_d$ of $H^1(G,\F_p)$ such that $(v_1,v_1)=1$, and that $(v_i,v_i)=0$ for every $2\leq i\leq d$.
\end{enumerate}
\end{prop}

\begin{proof} 
Let
\[
\begin{tikzcd}
1 \arrow{r} &R \arrow{r} &S \arrow{r} &G \arrow{r} &1,
\end{tikzcd}
\]
be a minimal presentation of $G$, with minimal system of generators $x_1,x_2,\ldots,x_d$.  

By \cite[Proposition 3.9.12]{NSW}, there exists an $\F_p$-basis  $v_1,v_2,\ldots,v_d$ of $H^1(S,\F_p)= H^1(G,\F_p)$ such that $v_i(x_j)=\delta_{ij}$, where $\delta$ is the Kronecker delta function.  
  
Suppose $q=p^k\neq 2$.  If $p>2$ then $(v_i,v_i)=0$ for every $i=1,2,\ldots,d$ since a skew-symmetric  bilinear form on a vector space over a field of characteristic $\neq 2$ is alternating.  From \cite[Proposition 3.9.13]{NSW} applied to the descending 2-central series $(S^{(i)})$ of $S$, if the defining relation $r$ of $G$ is such that
\[
r\equiv \prod_{j=1}^{d} x_j^{2a_j} \cdot \prod_{1\leq l<m\leq d} (x_l,x_m)^{a_{lm}} \mod{S^{(3)}}, \quad a_j, a_{lm}\in \F_2,
\]
then $(v_j,v_j)=a_j$.  If $q=2^k$ with $k\geq 2$, then $r$ takes the form shown in equation (\ref{eq:D1}) above, so $(v_i,v_i)=a_i=0$ for all $i=1,2\ldots,d$.  This establishes part 1.

Now suppose $q=2$.  Then $r$ takes the form shown in equation (\ref{eq:D2}), (\ref{eq:D3}) or (\ref{eq:D4}) above.  In each of these cases, $(v_1,v_1)=a_1=1$ and $(v_i,v_i)=a_i=0$ for all $i=2,3,\ldots,d$, which establishes part 2.    
\end{proof}

When $G$ is a Demushkin group, the following result, based on linear algebra, provides a means of calculating $|{\rm CP}(G,\F_p)|$ from $d(G)$ and $q(G)$.
  
\begin{lem}
\label{lem:LA}
Let $V$ be an $\F_p$-vector space of dimension $d\geq 3$ with  with basis $v_1,v_2,\ldots,v_d$. Let $(\cdot,\cdot)\colon V\times V\to \F_p$ be a non-degenerate skew-symmetric bilinear form on $V$. Let $N$ be the number of pairs $(x,y)\in V\times V$ such that $(x,y)=0$ and that $x,y$ are $\F_p$-linearly independent.
\begin{enumerate}

\item If $(v_i,v_i)=0$ for every $1\leq i\leq d$, then 
\[
N=(p^d-1)(p^{d-1}-p).
\]
\item If $(v_1,v_1)=1$ and $(v_i,v_i)=0$ for every $2\leq i\leq d$, then
\[
N=(2^{d-1}-1)(2^{d-1}-2) +2^{d-1}(2^{d-1}-1).
\]
\end{enumerate}
\end{lem}
\begin{proof}
 For each $y\in V\setminus \{0\}$, let 
\[
 y^\perp =\{ x\in V\mid (x,y)=0\}.
\]
Then $y^\perp$ is an $\F_p$-vector space, and $\dim y^\perp=\dim V -1= d-1$ since the bilinear form $(\cdot,\cdot)$ is non-degenerate. Let 
\[
\begin{aligned}
 C(y):&=\{x\in V \mid (x,y)=0\ \text{and}\ x,y \text{ are $\F_p$-linearly independent}\}\\
&=\{x\in y^\perp \mid x,y \text{ are $\F_p$-linearly independent}\}.
\end{aligned}
\]
\\
{\bf Case 1.}  Let $y=\sum_{i\in I} a_i v_i$, $a_i\in \F_p$, $I\subseteq \{1,2,\ldots,d\}$.  Then 
\[
\begin{aligned}
 (y,y)&=(\sum_{i\in I} a_i v_i, \sum_{j\in I} a_j v_j)=\sum_{i,j\in I}a_ia_j(v_i,v_j)\\
&=\sum_{i\in I} a_i^2(v_i,v_i)+\sum_{\substack{ i<j,\\ i,j\in I}}a_ia_j((v_i,v_j)+(v_j,v_i))\\
&=0,
\end{aligned}
\]
since $(\cdot,\cdot)$ is skew-symmetric and, by assumption, $(v_i,v_i)=0$ for $1\leq i\leq d$. So $y\in y^\perp$.  Hence 
\[
 |C(y)|=(p^{d-1}-p).
\]
This gives
\[
 N=\sum_{y\in V\setminus\{0\}} |C(y)|=(p^d-1)(p^{d-1}-p).
\]
\\
\\
{\bf Case 2.} In this case $p=2$. Let $y=\sum_{i\in I}v_i$, $I\subseteq \{1,2,\ldots,d\}$.  There are two possibilities to consider:
\\
(i) $1\not\in I$. Then $(v_i,v_i)=0$ for all $i\in I$, so by the same argument as Case 1 above, $y\in y^\perp$. Hence
\[
 |C(y)|=(2^{d-1}-2).
\]
\\
(ii) $1\in I$. Let $I^\prime:=I\setminus\{1\}$.  Then $(v_i,v_i)=0$ for all $i\in I^\prime$, so
\[
\begin{aligned}
 (y,y)&=(v_1+\sum_{i\in I^\prime} v_i, v_1+\sum_{i\in I^\prime} v_i)\\
&=(v_1,v_1)+ (v_1,\sum_{i\in I^\prime} v_i)+(\sum_{i\in I^\prime} v_i,v_1)+(\sum_{i\in I^\prime} v_i,\sum_{i\in I^\prime} v_i)\\
&=(v_1,v_1)\\
&=1,
\end{aligned}
\]
since $(\cdot,\cdot)$ is skew-symmetric, and $(\sum_{i\in I^\prime} v_i,\sum_{i\in I^\prime} v_i)=0$ by part (i).  So $y\notin y^\perp$.  Hence, 
\[
 |C(y)|=|y^\perp\setminus\{0\}|=(2^{d-1}-1).
\]
Combining the two possibilities gives
\[
\begin{aligned}
 N&=\sum_\text{$y$ in case (i)} |C(y)|+ \sum_\text{$y$ in case (ii)} |C(y)|\\
&=(2^{d-1}-1)(2^{d-1}-2) +2^{d-1}(2^{d-1}-1).
\end{aligned}
\]
\end{proof}

Let $K$ be a finite extension of degree $n$ of $\Q_p$ and assume that $K$ contains a primitive $p$-th root of unity.  Then the Galois group $G:=G_K(p)$ of the maximal $p$-extension of $K$ is a Demushkin group of rank $n+2$.  Recall that the number of $\U_3(\F_p)$-extensions of $K$ is given by
\[
\nu(K,\U_3(\F_p))=\frac{|{\rm Epi}(G_K,\U_3(\F_p))|}{|{\rm Aut}(\U_3(\F_p))|},
\]
where $G_K$ is the absolute Galois group of $K$.  Since $|Z^1(G,\F_p)|=|H^1(G,\F_p)|=p^{n+2}$, Proposition \ref{prop:counting Epi} together with Lemma \ref{lem:G mod Zassenhaus subgroup} gives
\[
|{\rm Epi}(G_K,\U_3(\F_p))|=|{\rm Epi}(G,\U_3(\F_p))|=\sum_{\varphi\in {\rm CP}(G,\F_p)}|Z^1(G,\F_p)| =|{\rm CP}(G,\F_p)|\cdot p^{n+2}.
\]
Since $G$ is a Demushkin group, the cup product $H^1(G,\F_p)\times H^1(G,\F_p)\stackrel{\cup}{\to} H^2(G,\F_p)\cong \F_p$ is a non-degenerate skew-symmetric bilinear form.  If $q=q(G)$ is the $q$-invariant of $G$, then by Lemma \ref{lem:LA}
\[
|{\rm CP}(G,\F_p)| =
\begin{cases}
(p^{n+2}-1)(p^{n+1}-p) & \text{if $p>2$},\\
(2^{n+2}-1)(2^{n+1}-2) & \text{if $p=2$ and $q>2$},\\
(2^{n+1}-1)(2^{n+1}-2)+2^{n+1}(2^{n+1}-1) &\text {if $p=2$ and $q=2$}.
 \end{cases}
\]  
Note also that 
\[
|{\rm Aut}(\U_3(\F_p))|= 
\begin{cases}
p^3(p^2-1)(p-1) &\text{if $p>2$},\\
8 &\text{if $p=2$}. 
\end{cases}
\]
Therefore
\[
\nu(K,\U_3(\F_p))=\frac{|{\rm CP}(G,\F_p)|\cdot p^{n+2}}{|{\rm Aut}(\U_3(\F_p))|}=
 \begin{cases}
\displaystyle
\frac{p^n(p^{n+2}-1)(p^n-1)}{(p^2-1)(p-1)} & \text{if $p>2$},\\
2^n(2^n-1)(2^{n+2}-1) & \text{if $p=2$ and $q>2$},\\
\displaystyle
2^n(2^{n+1}-1)^2 &\text {if $p=2$ and $q=2$}.
 \end{cases}
\]

\section{Formally Real Pythagorean Fields}
\label{sec:pyth}

\begin{defn}
A field $F$ is called \textit{pythagorean} if every sum of two squares (hence any number of squares) in $F$ is a square.  For any field $F$, the set of elements of $F$ that can be expressed as a sum of squares will be denoted $\sum F^2$.  If $F$ is pythagorean then $\sum F^2=F^2$.
\end{defn}

\begin{defn}
A field $F$ is \textit{formally real} if $F$ satisfies the following (equivalent) conditions:
\begin{enumerate}
\item -1 is not a sum of squares in $F$.
\item For any $n\in \N$, the quadratic form $n\langle 1 \rangle = \langle 1,\ldots,1 \rangle$ is anisotropic over $F$.
\end{enumerate}
Otherwise, $F$ is said to be \textit{nonreal}.
\end{defn}

If $F$ is a nonreal pythagorean field then for any $a\in F$, there exist $x,\ y,\ z\in F$ such that
\[
a=x^2-y^2=x^2+z^2y^2 \in F^2,
\]
so $F$ is quadratically closed.  Hence, our interest will be in formally real pythagorean fields.

\begin{defn}
Let $F$ be a field.  A subset $P$ of $F$ is called a \textit{preordering} of $F$ if 
\[
P+P\subseteq P,\quad P\cdot P \subseteq P,\quad -1\notin P,\quad \sum F^2\subseteq P.
\]
A preordering of a field $F$ is an \textit{ordering} if in addition
\[
P\cup -P=F,\quad P\cap -P=0.
\]
An \textit{ordered field} is a pair consisting of a field $F$ and an ordering $P$ of $F$.  
If $P$ is an ordering of $F$, the elements of $P^\times =P\setminus \{0\}$ are called \textit{positive}, the elements of $-P^\times$ are called \textit{negative} (with respect to $P$).  An element $b\in \fx$ is said to be \textit{totally positive} if it is positive with respect to all orderings on $F$.  The set of all orderings of $F$ will be denoted $X_F$.
  
\end{defn}

Artin and Schreier, in the 1920's, developed much of the algebraic theory of formally real fields and studied the relationship between formally real fields and fields with orderings.  We have the following important results.

\begin{thm}[Artin-Schreier Criterion, {\cite{AS}}]
A field $F$ is formally real if and only if $F$ possesses at least one ordering.
\end{thm}

\begin{thm}[ {\cite[Chapter 3, Theorem 1.6]{scharlau}}]
\label{thm:preordering}
Let $F$ be a formally real field and $P$ a preordering of $F$.  Then $P=\cap R$, where the intersection is taken over all orderings $R$ containing $P$.
\end{thm}

\begin{thm}[Artin's Theorem, {\cite{Artin}}]
For a field $F$ of characteristic $\neq 2$, an element $b\in \fx$ is totally positive if and only if $b\in \sum F^2$.
\end{thm}

We now wish to consider formally real pythagorean fields with a view toward counting their dihedral extensions.  We begin with

\begin{lem}
\label{lem:DF}
Let $F$ be a formally real pythagorean field with set of orderings $X_F$ and let $a\in \fx\setminus \fxs$.  Then
\[
D_F\langle1,-a\rangle = \bigcap_{P\in X_F, -a\in P} P^\times.
\] 
\end{lem}

\begin{proof}
Consider the set $F^2-aF^2 = \{x^2-ay^2 \mid x,y \in F\}$.  Since $F$ is pythagorean, $\sum F^2= F^2$, so this set is closed under addition and multiplication.  If $-1 = x^2-ay^2$, then $y\neq 0$ so $a=(x/y)^2+(1/y)^2 \in F^2$, a contradiction.  Hence $F^2-aF^2$ is a preordering of $F$.  Also, if $P\in X_F$, then since $F^2 \subseteq P$,
we have $F^2-aF^2 \subseteq P$ if and only if $-a\in P$.  So by Theorem~\ref{thm:preordering},
\[
D_F\langle1,-a\rangle \cup \{0\} = F^2-aF^2 = \bigcap_{P\in X_F, -a\in P} P.
\]
\end{proof}

\begin{lem}
\label{lem:orderings}
Let $n\in \N$ and let $F$ be a formally real pythagorean field with set of orderings $X_F$.  If $| \fx/\fxs | =2^n$, then $n\leq | X_F | \leq 2^{n-1}$.
\end{lem}

\begin{proof}
For any $P\in X_F$, $P^\times$ is a subgroup of index 2 in $\fx$.  The map
\[
\fx \to \prod_{P\in X_F} \fx/P^\times,\quad a\mapsto (a \mod{P})_{P\in X_F}
\]
has kernel $\bigcap_{P\in X_F} P^\times $, which is $\sum \fxs = \fxs$ by Artin's Theorem and the fact that $F$ is pythagorean.  So we have an injective map
\[
\fx/\fxs \hookrightarrow \prod_{P\in X_F} \{\pm 1\}.
\]
Hence $2^n= | \fx/\fxs | \leq 2^{|X_F|}$, so $n\leq |X_F|$.

Since $F$ is formally real, we can choose a basis $\{-1,a_1,\ldots,a_{n-1}\}$ of $\fx/\fxs$.  For any $P\in X_F$, $-1\notin P$ and for each $i=1,\ldots,n-1$, we have exactly one of $a_i\in P$ or $-a_i\in P$.  Each ordering $P$ could be labelled accordingly, so $F$ can have at most $2^{n-1}$ orderings.  
\end{proof}

\subsection{Extensions of SAP fields}
\label{sec:sap}

In this section, we develop a method for counting $D_4$-extensions of a formally real pythagorean field having the minimal number of orderings.

\begin{defn}
A formally real pythagorean field $F$ with finite square class group $\fx /\fxs$ is said to have the {\it Strong Approximation Property} (or to be \textit{SAP}) if for any subset $\{P_1, \ldots P_s\}$ of orderings of $F$, there exists $a \in F^\times$ such that 
\[
a \in \bigcap_{i=1}^s P_i, \quad a \notin P \  \textnormal {for\ all} \ P \neq P_i,\ i=1, \ldots ,s.
\] 
\end{defn}

Recall from sections \ref{sec:cohom}, \ref{sec:csa} and \ref{sec:quad} that for a field $F$ with char$(F)\neq 2$, we have the maps
\[
\fx/\fxs\cong H^1(G_F,\F_2)=H^1(G_F(2),\F_2),\quad a\mapsto (a),
\]
and
\[
\fx/\fxs \times \fx/\fxs \to H^2(G_F,\F_2)\cong \Br_2(F),\quad (a,b)\mapsto (a)\cup(b)\mapsto (a,b)_F.
\] 

\begin{lem}
\label{lem:basis}
Let $F$ be a formally real pythagorean SAP field with finite square class group of cardinality $2^n$.  There exists a basis $\sB =\{a_1,\ldots,a_n\}$ for $\fx/\fxs$ such that for all $a,b \in \fx/\fxs$, $(a)\cup (b)=0$ if and only if there is no common basis element $a_i \in \sB$ entering the expressions for both $a$ and $b$. 
\end{lem}

\begin{proof}
Let $F$ be a formally real pythagorean SAP field with $|F^\times/(F^\times)^2 | = 2^n$ and let $X_F$ be the set of orderings of $F$.  Since $F$ is SAP, for each $P_i \in X_F$ we can choose 
\[
a_i \in  \bigcap_{P \in X_F,\ P \neq P_i}P \setminus P_i.
\]  
Suppose $a_{i_1}\cdots a_{i_r}=1$ in $ F^\times/(F^\times)^2 $, where $i_j\neq i_k$ if $j\neq k$.  Then since $a_{i_2},\ldots, a_{i_r} \in P_{i_1}$, we have
$a_{i_1}=a_{i_2}\cdots a_{i_r}\in P_{i_1}$, a contradiction.  Hence the $ a_i $'s are independent mod $ (F^\times)^2 $ which implies $|X_F| \leq n$.  Then, by Lemma \ref{lem:orderings}, $|X_F| = n$, so $\sB = \{a_1,\ldots, a_n\} $ is a basis of $ F^\times/(F^\times)^2 $.  


Now for all $i \neq j$, we have
\[
\begin{array}{rl}
D_F<1,-a_i>:&=\{x^2-a_iy^2\mid x, y\in F\}\setminus \{0\}\\
&= \bigcap_{P\in X_F ,\ -a_i\in P}P^\times\\
&= \bigcap_{P\in X_F ,\ a_i\notin P}P^\times\\
&=P_i^\times.\\
\end{array}
\]
Hence $a_j \in D_F<1,-a_i>$, which implies $(a_i)\cup (a_j)=0$ in $H^2(G_F(2), \F_2)$.

Working modulo squares, any given $a,b \in \fx$ can be expressed in the basis $\{a_1, \ldots ,a_n\}$.  Let $c=\prod a_i$, where the product is taken over all elements $a_i$ which occur in the expression for both $a$ and $b$.  From the bilinearity of the cup product and the fact that $(a_i)\cup (a_j)=0$ if $i\neq j$, we have $(a)\cup (b)= (c)\cup (c)$.  The quaternion algebra $(c,c)_F\cong (c,-1)_F$ and $(c,-1)_F$ splits if and only if $c$ is a sum of squares in $F$. Since $F$ is pythagorean, the result follows.
\end{proof}

Now let $K=F(\sqrt{a}, \sqrt{b}),\ a,b \in \fx$ be a $V_4$-extension and recall that $K/F$ embeds into a $D_4$-extension $L/F$ if and only if $(a)\cup (b)=0$, and in that case, the possible extensions are
\[
L=K(\sqrt{f\gamma}),\ \textnormal {where}\ f\in \fx ,\ \gamma \in F(\sqrt{a})\textnormal{ with } N_{F(\sqrt{a})/F}(\gamma)=b.
\]

\begin{lem}
Let $\sS = \{ (a,b)\in \fx/ \fxs \times \fx/ \fxs \mid (a) \cup (b)=0,\ a,b\neq 1\}$.  Then 
\[
|\sS |= 3^n-2^{n+1}+1.
\]
\end{lem}

\begin{proof}
Suppose $a$ is expressed as a product of elements of the basis $\sB $ of Lemma \ref{lem:basis}, and similarly for $b$.  That is, $a=a_{i_1}\cdots a_{i_r},\ b=a_{j_1}\cdots a_{j_s}$.  Since $(a)\cup (b)=0$, there is no element of $\sB$ common to the expression of both $a$ and $b$.  So $2\leq k=r+s \leq n$.  Now choose a subset $\sA$ of $\sB$ of size $k$ and consider all subsets $\sC$ of $\sA$ such that $\sC \neq \varnothing $ and $\sC \neq \sA$.  There are $2^k-2$ such subsets $\sC$.   We take $a$ to be the product of the elements of $\sC$ and $b$ the product of the elements of $\sA \setminus \sC$.  This gives $ \sum_{k=2}^{n}\binom{n}{k}(2^k-2)$ ordered pairs $(a,b)$.  Using the binomial identity $(1+x)^n=\sum_{k=0}^{n}\binom{n}{k}x^k$ gives
\[
\begin{aligned}
|\sS|&=\sum_{k=2}^{n}\binom{n}{k}(2^k-2)\\
&=\sum_{k=2}^{n}\binom{n}{k}2^k -2\sum_{k=2}^{n}\binom{n}{k}\\
&=\sum_{k=0}^{n}\binom{n}{k}2^k -2\sum_{k=0}^{n}\binom{n}{k} +1\\
&=3^n-2^{n+1}+1.
\end{aligned}
\]
\end{proof}

Putting everything together we have

\begin{thm}
Let $F$ be a formally real pythagorean SAP field with $|F^\times/(F^\times)^2 | = 2^n,\ n\geq 2$ and let $\sN$ be the number of $D_4$-extensions of $F$.  Then 
\[
\sN= 2^{n-3}(3^n-2^{n+1}+1).
\]
\end{thm}

\begin{proof}
Using the notation of the previous lemma and preceding discussion, the number of biquadratic extensions $K=F(\sqrt{a}, \sqrt{b})$ which embed into a $D_4$-extension $L/F$ is given by the number of unordered pairs $\{a,b\}$ such that $(a) \cup (b)=0$.  For each such pair, there is a $1-1$ correspondence between $\{L/F\mid K/F \subset L/F,\ Gal(L/F)\cong D_4 \}$ and $\{f\mid f\in \fx/ (\fxs \cup a\fxs \cup b\fxs \cup ab\fxs) \}$.  Hence
\[
\begin{aligned}
\sN &=(\frac{1}{2}|\sS|)(2^{n-2})\\
&=2^{n-3}(3^n-2^{n+1}+1).
\end{aligned}
\]   
\end{proof}

\begin{examples}
We have the following results for the first few values of $n$.
\[
\begin{array}{rl}
n=2\quad &\sN=1\\
n=3\quad &\sN=12\\
n=4\quad &\sN=100\\
n=5\quad &\sN=720\\
n=6\quad &\sN=4816\\
n=7\quad &\sN=30912\\
n=8\quad &\sN=193600\\
\end{array}
\] 
\end{examples}

Now consider the extension field $K=F(\sqrt{-1})$ and its quadratic closure $K(2)$.  We will show that Lemma \ref{lem:basis} is useful not only in allowing us to count extensions of $F$, but also in elucidating the structure of the subgroup $G_K(2)$ of $G_F(2)$.  First, we recall the following lemma due to Bass and Tate.
    
\begin{lem}
\label{lem:K2}
Let $p$ be a prime.  If $E$ is a field which has no nontrivial finite extensions of degree less than $p$ and $L/E$ is an extension of degree $p$, then $K_2(L)$ is generated by the symbols (e,l) with $e\in E^\times$, $l \in L^\times$.
\end{lem}

\begin{proof}
See \cite[Lemma 8.6]{Srin}
\end{proof}

\begin{prop}
\label{prop:GK2}
Let $F$ be a formally real pythagorean SAP field with $|\fx /\fxs|=2^n$ and let $K=F(\sqrt{-1})$.  Then $G_K(2)$ is a free pro-2 group of rank $2^{n-1}$. 
\end{prop}

\begin{proof}
Let $a\in \fx$ and $b\in \kx$.  The corestriction map
\[
cor: H^2(G_K(2),\F_2) \longrightarrow H^2(G_F(2),\F_2)
\]
is given by $ cor ((a)\cup (b))=(a)\cup (N_{K/F}(b)) $.  Since $F$ is pythagorean, $ N_{K/F}(b)=(1) $ and $ cor ((a)\cup (b))=0 $.  By Lemma \ref{lem:K2} and Merkurjev's Theorem, $H^2(G_K(2),\F_2)$ is generated by the cup products $((f)\cup (k))$ with $f\in \fx,\ k\in \kx$.  Hence $cor$ is the zero map.

From Arason's long exact sequence \cite{arason} we obtain
\[
\begin{tikzcd}
0 \arrow{r} &\fx /\fxs  \arrow{r}{\bullet \cup (-1)} &H^2(G_F(2),\F_2) \arrow{r}{res} &H^2(G_K(2),\F_2) \arrow{r} &0.
\end{tikzcd}
\]

Once again working modulo squares, any given $e,f \in \fx$ can be expressed in the basis $\sB=\{a_1, \ldots ,a_n\}$ of Lemma \ref{lem:basis}.  Let $c=\prod a_i$, where $a_i$ enters the expression for both $e$ and $f$.  Then by the bilinearity of the cup product and Lemma \ref{lem:basis} we have
\[
\begin{array}{rl}
res((e)\cup (f))&=res((c)\cup (c))\\
&=res((-1)\cup (c))\\
&=0.
\end{array}
\]
So by Merkurjev's Theorem, $res:H^2(G_F(2),\F_2) \twoheadrightarrow H^2(G_K(2),\F_2)$ is the zero map.  Hence $ G_K(2) $ is a free pro-2 group.

From the short exact sequence 
\[
\begin{tikzcd}
0 \arrow{r} &\{\fxs \bigcup -\fxs \}  \arrow{r} &\fx /\fxs \arrow{r} &\kx /\kxs \\ 
\arrow{r}{0} &N(\kx)/\fxs \arrow{r} &0
\end{tikzcd}
\]
we see that $|\kx /\kxs|=2^{n-1}$.  So the rank of $G_K(2)$ is $n-1$.  
\end{proof} 

In the next section we go on to consider the group $G_F(2)$ for the case in which $F$ is a formally real pythagorean SAP field and also the case in which $F$ is a superpythagorean field.

\subsection{The group $G_F(2)$}
\label{sec:pythGF2}

\begin{thm}
\label{thm:SAP}
Let $F$ be a field with $|F^\times/(F^\times)^2|=2^{d+1},\ d\geq 0$. Then $F$ is a formally real pythagorean SAP field if and only if $G_F(2)\cong C_2*\cdots *C_2$, the free product of $d+1$ copies of $C_2$.
\end{thm}

\begin{proof} 
Suppose $F$ is a formally real pythagorean SAP field with space of orderings $X_F$.  Then $|X_F|=d+1$ and by \cite[Theorem 17.4]{Lam83}, a decomposition of $X_F$ into its connected components is given by $X_F=\bigoplus_{i=1}^{d+1} X_i$ where $|X_i|=1$ for each $i=1,\ldots,d+1$.  It then follows from \cite{Mi} that $G_F(2)$ is isomorphic to the free product of $d+1$ copies of $C_2$.

Conversely, suppose that $G_F(2)\cong C_2*\cdots*C_2$ ($d+1$ copies). Then $G_F(2)$ is generated by $d+1$ involutions, so $F$ is pythagorean and formally real.  Hence $\langle 1,1 \rangle_F$ is not universal.  In this case, R. Ware \cite{War1} showed that $G_F(2)$ determines the Witt ring $W(F)$ of $F$.  By \cite[Theorem 3.8]{MS2} and \cite[Corolllary 2.10]{MS1}, the Witt ring $W(F)$ determines the space of orderings $X_F$ of $F$. It then follows from the ``only if'' part that $|X_F|=d+1$, so $F$ is an SAP field.
\end{proof}

\begin{cor}
\label{cor:SAP}
Let $F$ be any pythagorean field, and let $K$ be a pythagorean SAP field. Assume that  $|F^\times/(F^\times)^2|=|K^\times/(K^\times)^2|=2^{d+1}$ . Then there exists an epimorphism $G_K(2)\cong C_2*\cdots*C_2 \twoheadrightarrow G_F(2)$.
\end{cor}

\begin{proof}
By Lemma \ref{lem:orderings}, $F$ has at least $d+1$ orderings, and we can choose $d+1$ involutions $\sigma_1,\ldots,\sigma_{d+1}$ in $G_F(2)$ which minimally generate $G_F(2)$.  The statement then follows from the previous theorem.
 
\end{proof}

We now wish to look at the case in which $F$ is a superpythagorean field.  Recall that a formally real pythagorean field $F$ with $|F^\times/(F^\times)^2|=2^{d+1}<\infty$ is called superpythagorean if $F$ admits exactly $2^d$ orderings.  We consider the group $G:=\Z_2^d \rtimes C_2= H \rtimes \langle x \rangle$, where the semidirect product action of $C_2$ on $H:=\Z_2^d$ is given by $xyx=y^{-1}$, for all $y\in H$.

\begin{prop}
Let $F$ be a pythagorean field with  $|F^\times/(F^\times)^2|=2^{d+1}$, $d\geq 0$. Then there exists an epimorphism $G_F(2)\twoheadrightarrow G=\Z_2^d\rtimes C_2$.
\end{prop}
\begin{proof}
Choose any ordering $P$ in $F$ and an $\F_2$-basis $[a_1],\ldots, [a_d]$ of $P^\times/(F^\times)^2$. By \cite{Be1} there exists a field $E$, called the Euclidean closure of $F$ with respect to $P$, such that $F(2)=E(\sqrt{-1})$, $E$ is a formally real field and $E^2\cap F=P$. For each $a_i,\ i=1,\ldots,d$, there exists a sequence 
\[
\sqrt{a_i}, \sqrt[4]{a_i},\ldots, \sqrt[2^n]{a_i},\ldots,
\]
such that $\sqrt[2^n]{a_i}\in E^\times$ for all $n\in \N$. Indeed, by induction on $n$, we may assume that $\sqrt[2^n]{a_i}\in E^\times$. Then since $E^\times=(E^\times)^2\cup -(E^\times)^2$, we can choose $\sqrt[2^{n+1}]{a_i}\in (E^\times)^2$. Now let 
\[
\tilde{M}:=\bigcup_{n=1}^\infty F(\sqrt[2^n]{a_1},\ldots,\sqrt[2^n]{a_d}).
\]
Then $\tilde{M}$ is formally real since $\tilde{M}$ is a subfield of $E$. For each $n\in \N$,  $F(\sqrt{-1})$ contains a primitive $2^n$-th root of unity $\zeta_{2^n}$ (see \cite[Chapter II, Theorem 8]{Be}) and we may also assume that $\zeta_{2^{n+1}}^2=\zeta_{2^n}$.
Let $M:=\tilde{M}(\sqrt{-1})$. Then $M/F$ is a Galois extension.

We now show that ${\rm Gal}(M/F{\sqrt{-1}})$ is isomorphic to $\Z_2^d$.  This follows from Kummer theory. Let $\tau_1, \ldots, \tau_d$ be elements in ${\rm Gal}(M/F(\sqrt{-1})$ such that  for each $i=1,\ldots,d$,
\[
\tau_i(\sqrt[2^n]{a_i})=\zeta_{2^n} \sqrt[2^n]{a_i}\ \text {and }\tau_i(\sqrt[2^n]{a_j})=\sqrt[2^n]{a_j},\ \forall j\not=i.
\]
Then ${\rm Gal}(M/F(\sqrt{-1}))=\prod_{i=1}^d\langle \tau_i\rangle \cong \Z_2^d$. 

The restriction of a nontrivial element of ${\rm Gal}(E(\sqrt{-1})/E)$ to $M$ gives a nontrivial element $\sigma\in {\rm Gal}(M/\tilde{M})$. Thus we have a splitting 
\[
{\rm Gal}(M/F)\cong {\rm Gal}(M/F{\sqrt{-1}})\rtimes \langle \sigma \rangle,
\]
where $ \langle \sigma \rangle\cong C_2$, and the action of $C_2$ on  ${\rm Gal}(M/F{\sqrt{-1}})$ is by involution.

The natural projection 
\[
G_F(2)={\rm Gal}(F(2)/F)\to {\rm Gal}(M/F)\cong \Z_2^d\rtimes C_2
\]
gives the desired epimorphism.
\end{proof}

\begin{cor}
\label{cor:superPytha}
Let $F$ be a a field with $|F^\times/(F^\times)^2|=2^{d+1}$. Then $F$ is a superpythagorean field if and only if $G_F(2)$ is isomorphic to the group $G=\Z_2^d \rtimes C_2$.
\end{cor}
\begin{proof}
 Assume that $F$ is a superpythagorean field with $|F^\times/(F^\times)^2|=2^{d+1}$. Let the notation be as in the previous proposition. Then ${\rm Gal}(M/F)\cong G=\Z_2^d \rtimes C_2$. On the other hand, from \cite[Example 3.8, (ii)]{Wa} (see also \cite[Chapter III, Theorem 1]{Be}), we know that ${\rm Gal}(M/F)$ is equal to $G_F(2)$. Hence $G_F(2)\cong \Z_2^d \rtimes C_2$.

The converse direction is proved in a similar  fashion to the proof of the "if" part in Theorem~\ref{thm:SAP}.
\end{proof}
\begin{cor}
\label{cor:superPy}
Let $F$ be any Pythagorean field, and let $K$ be a superpythagorean field. Assume that  $|F^\times/(F^\times)^2|=|K^\times/(K^\times)^2|=2^{d+1}$ . Then there exists an epimorphism $G_F(2) \twoheadrightarrow G_K(2)\cong \Z_2^d\rtimes C_2$.
\end{cor}
\begin{proof}
This follows from the previous proposition and corollary.
\end{proof}

We will consider these groups further in sections \ref{sec:freeprod} and \ref{sec:semidirect} when we look at dimensions of Zassenhaus filtration subquotients.

\chapter{Dimensions of Zassenhaus Filtration Subquotients}
\label{ch:dimensions}
\footnotetext{A version of this chapter is to appear in the Israel Journal of Mathematics \cite{MRT}.}

Central filtrations of profinite groups have a close connection with Galois theory.  In 1947, Shafarevich \cite{Sha} observed that for certain fields not containing primitive $p$-th roots of unity, one could show the Galois groups of their maximal $p$-extensions were free pro-$p$ groups by looking at the cardinality of filtration quotients.  

Early work by Witt \cite{witt} established a correspondence between free Lie rings and the higher commutator groups of free groups.  This idea has subsequently been very fruitful in the study and classification of pro-$p$ groups, one example being the important work of Labute on Demushkin groups \cite{La1} and mild pro-$p$ groups \cite{La3}. 

Recall that for a  group $G$ and a prime number $p$, the descending central series $(G_n)$ of $G$ is defined inductively by
\[
G_1=G,\quad G_{n+1}=[G_n,G]
\]
and the Zassenhaus ($p$-)filtration $(G_{(n)})$ of $G$ is defined inductively by
\[
G_{(1)}=G, \quad G_{(n)}=G_{(\lceil n/p\rceil)}^p\prod_{i+j=n}[G_{(i)},G_{(j)}],
\]
where $\lceil n/p \rceil$ is the least integer which is greater than or equal to $n/p$. 

Given a free Lie ring $L$ on $d$ generators and a free group $S$ on $d$ generators, Witt showed that there is an isomorphism between the additive group of the homogeneous elements of degree $n$ in $L$ and the multiplicative group $S_n/S_{n+1}$.

Our focus in this chapter will be primarily on the Zassenhaus filtration.  We will develop a method for determining the $\F_p$-dimension of subquotients of this filtration in the case of finitely generated pro-$p$ groups and derive an explicit formula for these subquotient dimensions for various families of groups, including free pro-$p$ groups, Demushkin groups and free pro-2 products of finitely many copies of the cyclic group of order 2.  Galois theory provides much of the underlying motivation as many of these groups are realizable as Galois groups of maximal $p$-extensions of certain fields, including local fields and formally real pythagorean fields.

In section \ref{sec:hilbert series} we define, for a finitely generated pro-$p$ group $G$,
\[
c_n(G):=\dim_{\F_p}(G_{(n)}/G_{(n+1)})
\]
and note that $c_n(G)$ is finite for every $n\geq 1$.  We show in Lemma \ref{lemHPfree} that the numbers $c_n(G)$ are sufficient to characterize finitely generated free pro-$p$ groups in the family of all finitely generated pro-$p$ groups.  In Remarks \ref{rmk:fgGalois} and \ref{rmks:SAP}, we observe that in some interesting cases, the two numbers $c_1(G)$ and $c_2(G)$ alone are sufficient to determine $G$.  We also observe that if $G$ is a free pro-$p$ group or a Demushkin group, the minimal number of topological generators of $G_{(n)}$ can be calculated from the dimensions $c_n(G)$.  In section \ref{sec:free} an interesting connection between these dimensions and the Kernel Unipotent Conjecture is also explored.

\section{The Hilbert-Poincar\'e Series}
\label{sec:hilbert series}

The Hilbert-Poincar\'e series is an important tool which allows us to study filtrations of profinite groups from the group algebra standpoint. 

\begin{defn}
Let $R$ be a unital commutative ring and $V=\bigoplus_{i=0}^\infty V_n$ a graded free $R$-module.  $V$ is called {\it locally finite} if ${\rm rank}_R(V_n)<\infty$ for all $n\geq 0$. For such a graded free $R$-module $V$, the {\it Hilbert-Poincar\'e series} $P_V(t)\in \Z[[t]]$ of $V$ is the formal power series
\[
P_V(t)=\sum_{n=0}^\infty {\rm rank}_R(V_n) t^n.
\]
\end{defn}

We recall also the following definitions from the theory of Lie algebras.

\begin{defn}
A Lie algebra $L$ over a commutative ring $R$ is an $R$-module equipped with a bilinear composition $(x,y)\to [xy]$ that satisfies the two conditions
\[
[xx]=0\quad \textnormal{and}\quad [[xy]z]+[[yz]x]+[[zx]y]=0.
\]
By an $R$-algebra we mean an associative ring with identity, containing $R$ as a subring. Any $R$-algebra $A$ defines a Lie algebra $A_L$ having the same $R$-module structure as that of $A$ with the Lie product given by $[xy]:=xy-yx$.  A `Lie subalgebra of $A$' means a Lie subalgebra of $A_L$.

Given any Lie algebra $L$ over $R$, we can construct the \textit{universal enveloping algebra} $U(L)$ of $L$ as follows.  Form the tensor algebra $T(L)$ for the $R$-module $L$, $T(L)=R\oplus L \oplus L\otimes L \oplus \cdots$ and let $U(L)=T(L)/I$, where $I$ is the ideal in $T(L)$ generated by all elements of the form
\[
[xy]-x\otimes y + y\otimes x,\quad x,y\in L.
\]    
If $u$ is the restriction to $L$ of the canonical homomorphism of $T(L)$ onto $U(L)$, then $u$ is a homomorphism of the Lie algebra $L$ into $U(L)_L$.  

The pair $(U(L),u)$ has the following universal property.  If $A$ is any $R$-algebra and $g$ is a homomorphism of $L$ into $A_L$, then there exists a unique $R$-algebra homomorphism $\hat{g}:U(L)\to A$, such that the following diagram of Lie algebra homomorphisms commutes
\[
\begin{tikzcd}
L \arrow {rd} [swap] {g} \arrow {r} {u} & U(L)_L \arrow {d} {\hat{g}}\\
{} & A_L
\end{tikzcd}
\]
\end{defn}

\begin{defn}
Let $k$ be a field of characteristic $p$.  Let $A$ be a $k$-algebra and let L be a Lie subalgebra of $A$.  Then $L$ is said to be \textit{restricted} if for each element $a\in L$, $a^p\in L$.

More generally, a Lie algebra $L$ over $k$, with an additional unary operation $[p]$, is called a \textit{restricted Lie algebra} if there exist a $k$-algebra $A$ and a Lie algebra monomorphism $\theta:L\to A_L$ such that $\theta(a^{[p]})=\theta(a)^p$ for all $a\in L$.  In this case $A$ is called a \textit{restricted enveloping algebra} of $L$.  A Lie algebra homomorphism between two restricted Lie algebras is called \textit{restricted} if it preserves the operation $[p]$.  

A restricted enveloping algebra $U$ of $L$ is \textit{universal} if it has the following universal property:  for any restricted Lie algebra homomorphism $\varphi: L\to B_L$, where $B$ is a $k$-algebra, there exists a unique $k$-algebra homomorphism $\hat{\varphi}:U\to B$ such that the following diagram of restricted Lie algebra homomorphisms commutes
\[
\begin{tikzcd}
L \arrow {rd} [swap] {\varphi} \arrow {r} {\theta} & U_L \arrow {d} {\hat{\varphi}}\\
{} & B_L
\end{tikzcd}
\]  
\end{defn}

Now let $G$ be a finitely generated pro-$p$ group.  Recall that $(I^n(G))_{n\geq 0}$ is the filtration of the completed group algebra $\F_p[[G]]$ of $G$ over $\F_p$ by powers of the augmentation ideal, where $I^0(G)= \F_p[[G]]$.  There are two graded $\F_p$-algebras associated to $G$ and $\F_p[[G]]$ respectively which are defined by
\[
{\rm gr}(G):= \bigoplus_{n\geq 1}  G_{(n)}/G_{(n+1)} \quad  \text{ and } \quad {\rm gr}(\F_p[[G]]): =\bigoplus_{n\geq 0} I^n(G)/I^{n+1}(G).
\]

Since $G$ is finitely generated, it follows from \cite[Lemma 7.10 and Theorem 7.11]{Ko} that the graded algebras ${\rm gr}(\F_p[[G]])$ and ${\rm gr}(G)$ are locally finite.  We define $a_n(G):= \dim_{\F_p} I^n(G)/I^{n+1}(G)$ and $c_n(G):=\dim_{\F_p} G_{(n)}/G_{(n+1)}$.

The following theorem is a consequence of a beautiful theory of Jennings and Lazard \cite[Chapters 11 and 12]{DDMS}, viewing the Zassenhaus filtration subgroups $G_{(n)}$ as dimension subgroups. (See also \cite{Qu}.)

\begin{thm}[Jennings-Lazard]
\label{thm:JL}
 Let the notation be as above.
\begin{enumerate}
\item[(i)] The graded algebra ${\rm gr}(G)$ is a restricted Lie algebra.
\item[(ii)] The graded algebra ${\rm gr}(\F_p[[G]])$ is a universal restricted enveloping algebra of ${\rm gr}(G)$.
\item[(iii)] We have
 \begin{equation}
  \label{eq:fundamental}  P_{{\rm gr}(\F_p[[G]])}(t) = 
 \sum_{n=0}^\infty a_n(G) t^n=\prod_{n=1}^\infty \left(\frac{1-t^{np}}{1-t^n}\right)^{c_n(G)}.
\end{equation}
\end{enumerate}
\end{thm}

\begin{proof}
(i)  See \cite[Theorem 12.8(i)]{DDMS}. 

(ii) See \cite[Theorem 12.8(iii)]{DDMS}.

(iii) See \cite[Theorem 12.16]{DDMS} (see also \cite[Proposition 2.3]{Er}).
\end{proof} 

We have a similar result relating the descending central series of $G$ to the filtration $(J^n(G))_{n\geq 0}$ of the completed group algebra $\Z_p[[G]]$ by powers of the augmentation ideal.  There are two graded $\Z_p$-algebras associated to $G$ and $\Z_p[[G]]$ respectively which are defined by
\[
{\rm gr}_\gamma(G)= \bigoplus_{n\geq 1}  G_{n}/G_{n+1} \quad \text{ and } \quad  {\rm gr}(\Z_p[[G]]) =\bigoplus_{n\geq 0} J^n(G)/J^{n+1}(G).
\]
\begin{lem}
\label{lem:integral version}
Let $G$ be a finitely generated pro-$p$-group. Assume that the graded algebra ${\rm gr}_\gamma(G)=\bigoplus_{n\geq1} G_n/G_{n+1}$ is torsion free. Let $e_n(G)={\rank}_{\Z_p} G_n/G_{n+1}$.
\begin{enumerate}
 \item[(i)] The graded algebra ${\rm gr}(\Z_p[[G]])$ is a universal enveloping algebra of ${\rm gr}_\gamma(G)$.
\item[(ii)] $J^n(G)/J^{n+1}(G)$ is a free module over $\Z_p$ of finite rank $d_n(G)$,  and   
\[
 P_{{\rm gr}(\Z_p[[G]])}(t) = 
 \sum_{n=0}^\infty d_n(G) t^n=\prod_{n=1}^\infty \frac{1}{(1-t^n)^{e_n(G)}}.
\]
\end{enumerate}
\end{lem}
\begin{proof} (i) This follows from \cite[Theorem 1.3]{Har} and Corollary \ref{cor:finitegen}. 

 (ii) This follows from (a) and \cite[Proposition 2.5]{La3}.
\end{proof}

\begin{ex}
\label{ex:HPseriesCp}
If $G=C_p$ is the cyclic group of order $p$ then since ${\rm gr}(C_p)=C_p$ and  ${\rm gr}(\F_p[[C_p]])={\rm gr}(\F_p[C_p])\cong \F_p[x]/(x^p) $, we have
\[ P_{{\rm gr}(\F_p[[C_p]])}(t)= 1+t+\cdots+t^{p-1}.
\]
\end{ex}

The following lemma is an important technical tool which relies on a fundamental result of Lichtman and also on a simple but remarkable formula which can be traced back to the work of Lemaire in \cite[Chapter 5]{Le}. 

\begin{lem} 
\label{lem:free product}
Let $G_1$ and $G_2$ be two finitely generated pro-$p$-groups. Let $G=G_1*G_2$ be the free product of $G_1$ and $G_2$ in the category of  pro-$p$-groups. Then
\[
P_{{\rm gr}(\F_p[[G]])}(t) = (P_{{\rm gr}(\F_p[[G_1]])}^{-1}(t)+P_{{\rm gr}(\F_p[[G_2]])}^{-1}(t)-1)^{-1}.
\]
\end{lem}

\begin{proof} By \cite[Theorem 1]{Li}, the graded $\F_p$-algebra ${\rm gr}(\F_p[[G]])$ is a free product (i.e., a categorical coproduct) of ${\rm gr}(\F_p[[G_1]])$ and ${\rm gr}(\F_p[[G_2]])$. The statement then follows from \cite[Equation (1.2), page 56]{PP}.
\end{proof}

\begin{ex}
If $G=C_2*\cdots*C_2$ is a free product of $d+1$ copies of $C_2$ the cyclic group of order 2, then by the previous example,  Lemma~\ref{lem:free product} and induction on $d$ it follows that 
\[ P_{{\rm gr}(\F_p[[G]])}(t)= \frac{1+t}{1-dt}.
\]
\end{ex}

Our aim is to use these results to develop a formula for $c_n(G)$ for various families of pro-$p$ groups $G$.  We proceed as follows.  Given a power series $P(t)=1+\sum_{n\geq 1}a_nt^n\in \Z[[t]]$, we define $c_n,\ n=1,2,\ldots$ by
\[
 P (t)=1+\sum_{n\geq 1}a_n t^n=\prod_{n=1}^\infty \left(\frac{1-t^{np}}{1-t^n}\right)^{c_n}.
\]
Now write $\log P(t)= \sum_{n\geq 1} b_n t^n$. We will derive a formula for $c_n$ using the values $b_1,\ldots,b_n$. 

Taking logarithms and using $\log(\dfrac{1}{1-t})=\sum \limits_{\nu=1}^\infty \dfrac{1}{\nu}t^\nu$ gives
\[
 \sum_{n=1}^\infty b_n t^n =\sum_{m=1}^\infty c_m \sum_{\nu=1}^\infty \frac{1}{\nu} (t^{m\nu}-t^{mp\nu}).
\]
Equating the coefficients of $t^n$, we have
\[
b_n=\sum_{m\nu=n} \frac{1}{\nu}c_m - \sum_{mp\nu=n}\frac{1}{\nu}c_m.
\]
Hence
\[
 nb_n= \sum_{m\mid n} m c_m -\sum_{mp\mid n} mp c_m.
\]

We now define a new sequence $w_n,\ n=1,2,\ldots$ by
\[
w_n=\frac{1}{n}\sum_{m\mid n} \mu(n/m) mb_m,
\]
where $\mu$ is the M\"obius function: for a positive integer $d$, 
\[
 \mu(d)=
\begin{cases}
(-1)^r & \text{ if $d$ is a product of $r$ distinct prime numbers},\\
0 &\text{ otherwise}.
\end{cases}
\]
Then by the M{\"o}bius inversion formula,
\[
nb_n= \sum_{m\mid n} m w_m.
\]

\begin{rmk}
\label{rmk:wn}
 From the definition of $w_n$ we see that
 \[
 P (t)=1+\sum_{n\geq 1}a_n t^n=\prod_{n=1}^\infty \frac{1}{(1-t^n)^{w_n}}.
\]
\end{rmk}

\begin{lem}
\label{lem:coprime}
If $(n,p)=1$ then $c_n=w_n$. 
\end{lem}
\begin{proof}
 Assume that $(n,p)=1$. Then we have
\[nb_n=\sum_{m\mid n} mc_m.\]
Hence by the M{\"o}bius inversion formula, 
\[
 c_n=\frac{1}{n}\sum_{m\mid n} \mu(n/m) mb_m=w_n.
\qedhere
\]
\end{proof}

\begin{lem} 
\label{lem:not coprime}
If $p$ divides $n$, then 
\[
c_n = c_{n/p}+w_n.
\]
\end{lem}
\begin{proof}
The proof is by induction on $n$. Clearly $c_p-c_1=\dfrac{pb_p-b_1}{p}=w_p$, hence the statement is true for $n=p$. Assume now that $n>p$ and $p\mid n$.  Assume also that the statement is true for every $m$ such that $p\mid m\mid n$, $m\not=n$. 

Then
\[
\begin{aligned}
 nb_n &=\sum_{m\mid n} m c_m -\sum_{pm\mid n} pm c_m\\
&=\sum_{m\mid n} m c_m -\sum_{p\mid m\mid n} m c_{m/p}\\
&= \sum_{m\mid n, (m,p)=1} m c_m +\sum_{p\mid m\mid n} m (c_m-c_{m/p})\\
&= \sum_{m\mid n, (m,p)=1} m w_m +\sum_{p\mid m\mid n,m\not= n} m w_m + n(c_n-c_{n/p})\\
&=\sum_{m\mid n, m\not= n} m w_m +n(c_n-c_{n/p}). 
\end{aligned}
\]
Combining this with 
\[
 nb_n=\sum_{m\mid n} mw_m,
\]
gives $c_n-c_{n/p}=w_n$.  Hence the statement is true for all $n$.
\end{proof}

\begin{prop}
\label{prop:key}
 If $n=p^k m $ with $(m,p)=1$, then 
\[
c_n =w_m +w_{pm}+\cdots + w_{p^km}.
\]
\end{prop}
\begin{proof}
This follows from the previous two lemmas.
\end{proof}  
\begin{thm}
\label{thm:general}
Let $G$ be a finitely generated pro-$p$-group. Write
\[ \log P_{{\rm gr}(\F_p[[G]])}(t)=\sum_{n\geq  1}b_nt^n\in \Q[[t]],\] and define $w_n(G)$ by
\[
w_n(G):=\frac{1}{n}\sum_{m\mid n} \mu(n/m) mb_m.
\]
Let  $n=p^k m $ with $(m,p)=1$. Then 
\[
c_n(G) =w_m(G) +w_{pm}(G)+\cdots + w_{p^km}(G).
\]

\end{thm}
\begin{proof} This follows from Theorem~\ref{thm:JL} and Proposition~\ref{prop:key}.
\end{proof}

The following proposition points out that in certain cases there is a close relationship between the quantities $c_n(G):=\dim_{\F_p} G_{(n)}/G_{(n+1)}$ and $e_n(G):={\rank}_{\Z_p} G_n/G_{n+1}$.

\begin{prop}
\label{prop:wn}
Let $G$ be a finitely generated pro-$p$-group and keep the same notation as in Lemma~\ref{lem:integral version} and Theorem~\ref{thm:general}. Assume that the graded algebra ${\rm gr}_\gamma(G)=\bigoplus_{n\geq1} G_n/G_{n+1}$ is torsion free. The following are equivalent.
\begin{enumerate}
 \item[(i)] ${\rm rank}_{\Z_p} J^n(G)/J^{n+1}(G)=\dim_{\F_p} I^n(G)/I^{n+1}(G)$ for all $n\geq 1$.
\item[(ii)] $w_n(G)={\rm rank}_{\Z_p} G_n/G_{n+1}$ for all $n\geq 1$. 
\end{enumerate}
\end{prop}
\begin{proof} 

 (i) $\Rightarrow$ (ii): Assume that  ${\rm rank}_{\Z_p} J^n(G)/J^{n+1}(G)=\dim_{\F_p} I^n(G)/I^{n+1}(G)$ for all $n$. Then by Theorem~\ref{thm:JL}, Remark~\ref{rmk:wn} and Lemma~\ref{lem:integral version}, we have
\[
 P_{{\rm gr}(\F_p[[G]])}(t)=\prod_{n=1}^\infty \frac{1}{(1-t^n)^{w_n(G)}}= P_{{\rm gr}(\Z_p[[G]])}(t)=\prod_{n=1}^\infty \frac{1}{(1-t^n)^{e_n(G)}}.
\]
Therefore $w_n(G)=e_n(G)$ for all $n\geq 1$. 

(ii) $\Rightarrow$ (i): Assume that $w_n(G)=e_n(G)$ for all $n\geq 1$. Then by Theorem~\ref{thm:JL}, Remark~\ref{rmk:wn} and Lemma~\ref{lem:integral version}, we have
\[
   P_{{\rm gr}(\F_p[[G]])}(t)=P_{{\rm gr}(\Z_p[[G]])}(t).
\]
Therefore  ${\rm rank}_{\Z_p} J^n(G)/J^{n+1}(G)=\dim_{\F_p} I^n(G)/I^{n+1}(G)$ for all $n\geq 1$.
\end{proof}

\begin{rmk}
 We shall see in the next sections that both a free finitely generated pro-$p$-group and a Demushkin group with a relation of the form $r=[x_1,x_2]\cdots [x_{d-1},x_d]$ satisfy the equivalent statements in Proposition~\ref{prop:wn}. 
\end{rmk}

\section{Free Pro-$p$ Groups}
\label{sec:free}

Throughout this section we assume that $S$ is a free pro-$p$-group on a finite set of  generators $x_1,\ldots,x_d$.
Recall that the  Magnus homomorphism from the  completed group algebra $\F_p[[S]]$  to the $\F_p$-algebra $\F_p\langle\langle X_1,\ldots,X_d\rangle\rangle$ of formal power series in $d$ non-commuting variables $X_1,\ldots,X_d$ over $\F_p$ is given by 
\[
\psi\colon \F_p[[S]] \to \F_p \langle\langle X_1,\ldots,X_d\rangle\rangle, x_i\mapsto 1+X_i.
\]
The $\F_p$-algebra $\F_p\langle\langle X_1,\ldots,X_d\rangle\rangle$ is equipped with a natural valuation $v$ given by
\[
v(\sum a_{i_1,\ldots,i_k}X_{i_1}\cdots X_{i_k})=\inf\{k\mid a_{i_1,\ldots,i_k}\not=0\}\in \Z_{\geq 0}\cup \{\infty\},
\]
making it a compact topological $\F_p$-algebra and by Theorem~\ref{thm:magnusiso} the Magnus homomorphism is a (topological) isomorphism.

\begin{lem} 
\label{lemHPfree}
A finitely generated pro-$p$ group $S$ is free of rank $d$ if and only if the Hilbert-Poincar\'e series
\[
P_{{\rm gr}(\F_p[[S]])}(t)=\frac{1}{1-dt}.
\]
 \end{lem}
 
\begin{proof}
$ (\Rightarrow) $
 Via the Magnus homomorphism, the augmentation ideal $I(S)$ is mapped to the ideal $I=(X_1,\ldots,X_d)$ of $\F_p\langle\langle X_1,\ldots,X_d\rangle\rangle$. Hence 
\[
 a_n(S):= \dim_{\F_p}(I^n(S)/I^{n+1}(S))=\dim_{\F_p}(I^n/I^{n+1}),
\]
which is equal to the number of non-commutative monomials of degree $n$ in $d$ variables $X_1,\ldots,X_d$. Hence $a_n(S)=d^n$. The result then follows. 

$ (\Leftarrow) $
Let $S$ be a finitely generated free pro-$p$ group of rank $d$ and suppose $G$ is a finitely generated pro-$p$ group with 
\[ 
P_{{\rm gr}(\F_p[[G]])}(t)=\frac{1}{1-dt}.
\]
Then
\[
\log(\frac{1}{1-dt})=\sum \limits_{\nu=1}^\infty \dfrac{1}{\nu}(dt)^\nu,
\]
so by Theorem~\ref{thm:general},
\[
w_n(G)=w_n(S):=\frac{1}{n}\sum_{m\mid n} \mu(m) d^{n/m}
\]
and $c_n(G)=c_n(S)$ for all $n\geq 1$.  Since $c_1(G)=w_1(G)=d$, which is equal to the minimal number of topological generators of $G$, there exists a minimal presentation of $G$:
\[
1 \to R\to S\to G\to 1.
\]
Then for all $n\geq 1$, $c_n(G)=c_n(S)$ implies $|S/S_{(n)}|=|G/G_{(n)}|$ and hence the natural epimorphism 
\[
S/S_{(n)} \surj G/G_{(n)}
\]
is an isomorphism. This implies that $R\subseteq S_{(n)}$ for all $n\geq 1$, so by \cite[Theorem 7.11]{Ko}, $R=1$.  Hence $G\cong S$.
\end{proof}

Defining $w_n(S)$ by
\[
w_n(S)=\frac{1}{n}\sum_{m\mid n} \mu(m) d^{n/m},
\]
Theorem~\ref{thm:general} immediately implies the following result.

\begin{prop}
\label{prop:cn free}
If $n=p^k m $ with $(m,p)=1$, then 
\[
c_n(S) =w_m(S) +w_{pm}(S)+\cdots + w_{p^km}(S).
\qed
\]
\end{prop}

\begin{rmk}
\label{rmk:wn free}
Let $(S_n)$ be the lower central series of $S$.  Then by Witt's result, $S_{n}/S_{n+1}$ is a free $\Z_p$-module of finite rank $w_n(S)$.
\end{rmk}

\begin{rmks}
\label{rmk:fgGalois}
(1) If a finitely generated pro-$p$-group $G$ is known to be realizable as the Galois group of a maximal $p$-extension of a field $F$ containing a primitive $p$-th root of unity, then we need only $c_1(G)=c_1(S)$ and $c_2(G)=c_2(S)$, for some finitely generated free pro-$p$-group $S$, to establish that $G$ is isomorphic to $S$.

Indeed, as $c_1(S)=c_1(G)$ we have a short exact sequence
\[
1 \to R\to S\stackrel{\pi}{\to} G\to 1.
\]
Since $c_1(S)=c_1(G)$ and $c_2(S)=c_2(G)$, we have $|S/S_{(3)}|=|G/G_{(3)}|$. Thus the natural epimorphism 
\[
S/S_{(3)} \surj G/G_{(3)}
\]
is in fact an isomorphism. Hence by \cite[Theorem C]{EM} (see also \cite[Theorem D]{CEM} for the case $p=2$) we see that $\pi\colon S\to G$ is an isomorphism.

(2) Observe that the numbers $c_n(S)$, $n=1,2,\ldots$, also detect the minimal number  of generators of $S_{(n)}$. Indeed by the pro-$p$ version of Schreier's formula, for each open subgroup $T$ of $S$ we have the following expression for the minimal number of generators $d(T)$ of $T$:
\[
 d(T)=[S:T](d(S)-1)+1.
\]
Therefore
\[
 d_n(S):=d(S_{(n)})= p^{\sum_{i=1}^{n-1}c_i(S)}(d-1)+1. 
\]
\end{rmks}

\begin{ex} Let $S$ be a free pro-$p$-group of finite rank $d$. We have
\[
\begin{aligned}
 c_1(S)&=d,\\
c_2(S)&= \begin{cases}
\frac{d^2-d}{2} \text{ if } p\not=2,\\
\frac{d^2+d}{2} \text{ if } p=2, 
\end{cases}\\
c_3(S) &=\begin{cases}
\frac{d^3-d}{3} \text{ if } p\not=3,\\
\frac{d^3+2d}{3} \text{ if } p=3,
\end{cases}\\
c_4(S)&= \begin{cases}
\frac{d^4-d^2}{4} \text{ if } p\not=2,\\
\frac{d^4+d^2+2d}{4} \text{ if } p=2, 
\end{cases}\\
c_5(S)&=\begin{cases}
\frac{d^5-d}{5} \text{ if } p\not=5,\\
\frac{d^5+4d}{5} \text{ if } p=5. 
\end{cases}
\end{aligned}
\]
We can look at this example in more detail.  For any minimal presentation 
\[
1 \to R\to S\to G\to 1,
\]
\[
d=c_1(S)=\dim_{\F_p}\frac{S}{S^p[S,S]}=\dim_{\F_p}\frac{G}{G^p[G,G]}=c_1(G),
\]
so $c_1(G)$ is an important invariant which gives the minimal number of generators of $G$.

If $p\neq 2$ then
\[
\begin{aligned}
c_2(S)&=\dim_{\F_p}\frac{S_{(2)}}{S_{(3)}}\\
&=\dim_{\F_p}\frac{S^p[S,S]}{S^p[[S,S],S]}\\
&=\dim_{\F_p}<{\overline{[x_i,x_j]}\mid 1\leq i<j\leq d}>\\
&=\binom{d}{2}\\
&=\frac{d^2-d}{2}\\
&=w_2(S).
\end{aligned}
\]
Recall that
\[
\gr (\F_p[[S]])=\F_p \oplus \frac{I(S)}{I^2(S)} \oplus \frac{I^2(S)}{I^3(S)} \oplus \cdots
\]
is the universal restricted enveloping algebra of the restricted Lie algebra
\[
\gr (S)=\F_p \oplus \frac{S}{S_{(2)}} \oplus \frac{S_{(2)}}{S_{(3)}} \oplus \cdots 
\]
This leads to the equation
\[
\begin{aligned}
&(1+t+t^2+\cdots +t^{p-1})^{c_1(S)}\\
&\cdot (1+t^2+t^{2\cdot 2}+\cdots +t^{2(p-1)})^{c_2(S)} \\
&\cdot (1+t^3+t^{3\cdot 2}+\cdots +t^{3(p-1)})^{c_3(S)} \\
&\cdots \\
&=1+a_1(S)t+a_2(S)t^2+a_3(S)t^3+\cdots\\
&=1+dt+d^2t^2+d^3t^3+\cdots
\end{aligned}
\]
Equating coefficients of $t$ gives $c_1(S)=d$, which reflects the fact that $\overline{X}_1=\overline{{x_1-1}},\ldots , \overline{X}_d=\overline{{x_d-1}}$ is a basis of $I(S)/I^2(S)$.  Equating coefficients of $t^2$ gives $d^2=d+\binom{d}{2}+c_2(S)$, so $c_2(S)=\frac{d^2-d}{2}$.

If $p=2$ then $p-1<2$, so the above equation gives $d^2=\binom{d}{2}+c_2(S)$ or $c_2(S)=\frac{d(d+1)}{2}$, which reflects the fact that, in this case, a basis of $S_{(2)}/S_{(3)}$ also contains the squares of generators. 

Similarly, considering the case $p\geq 5$ and looking at coefficients of $t^3$, we find
\[
\begin{aligned}
&(1+dt+(\binom{d}{2}+d)t^2+(\binom{d}{3}+d(d-1)+d)t^3+\ldots )\\
&\cdot (1+\frac{d^2-d}{2}t^2+\sO (t^4)) \\
&\cdot (1+c_3(S)t^3+\sO (t^6)) \\
&\cdots \\
&=1+dt+d^2t^2+d^3t^3+\cdots
\end{aligned}
\]
which gives $c_3(S)=\frac{d^3-d}{3}$.
\end{ex}

We can give an explicit $\F_p$-basis for $S_{(n)}/S_{(n+1)}$, for each $n$ in terms of Hall commutators, which we now describe.  We note that an $\F_p$-basis for $S_{(n)}/S_{(n+1)}$ is also given in \cite{Ga}.

\begin{defn} Let $S$ be the free group generated by $\{x_1,\ldots,x_d\}$.  The set $C_n$ of {\it Hall commutators of weight $n$} together with a total order $<$ is inductively defined as follows:
\begin{enumerate}
 \item $C_1=\{x_1,\ldots,x_d\}$ with the ordering $x_1>\cdots >x_d$.
\item Assume $n>1$ and that the Hall commutators have been defined and simply ordered for all weights $<n$ so that commutators of weight $k$ are greater than all commutators of weight $<k$.  Then $C_n$ is the set of all commutators $[c_1,c_2]$ where $c_1\in C_{n_1},c_2\in C_{n_2}$ such that $n_1+n_2=n$, $c_1>c_2$ and if $c_1=[c_3,c_4]$ then we also require that $c_2\geq c_4$. The set $C_n$ is ordered lexicographically, i.e., $[c_1,c_2]<[c_1^\prime,c_2^\prime]$ if and only if $c_1<c_1^\prime$, or  $c_1=c_1^\prime$ and $c_2<c_2^\prime$.  
\end{enumerate}
\end{defn}

The following theorem was proved by M. Hall in the discrete case. The extension of his theorem to the pro-$p$ case follows from Corollary~\ref{cor:finitegen}.

\begin{thm}[{\cite[Theorem 4.1]{Ha}}] The Hall commutators of weight $n$ represent a basis of $S_n/S_{n+1}$ as a free $\Z_p$-module.  In particular, $w_n(S)=|C_n|$. 
\end{thm}

The following theorem relating the Zassenhaus filtration to the descending central series is due to Lazard.

\begin{thm}[Lazard]
For each $n$, one has
\[
G_{(n)}=\prod_{ip^j\geq n}G_i^{p^j}.
\]
\end{thm}
\begin{proof}
See \cite[Theorem 11.2]{DDMS}.
\end{proof}

\begin{cor}
Let us write $n=p^km$ with $(m,p)=1$. Then a basis of the $\F_p$-vector space $S_{(n)}/S_{(n+1)}$ can be represented by the following set
\[
C_m^{p^k} \bigsqcup C_{pm}^{p^{k-1}}\bigsqcup \cdots  \bigsqcup C_{p^{k-1}m}^p\bigsqcup C_n.
\]
\end{cor}
\begin{proof}
 By Lazard's theorem, the above set defines a set of generators for the $\F_p$-vector space $S_{(n)}/S_{(n+1)}$. Now by Theorem~\ref{thm:general} and a counting argument, we see that this set defines a basis for the $\F_p$-vector space $S_{(n)}/S_{(n+1)}$. 
\end{proof}

With this basis in mind, we revisit the calculation of $c_3(S)$ for $p\neq 3$.  As pointed out in \cite{Ga}, $C_3=\{[[x_i,x_j],x_k]\mid 1\leq i<j\leq d, k\leq j \}$ and 
\[
\begin{aligned}
|C_3|&=2\binom{d+1}{3}\\
&=\frac{d^3-d}{3}.
\end{aligned}
\]

We now consider an interesting, purely group theoretical corollary of our formula for $c_n(S)$ which is closely related to the Kernel $n$-Unipotent Conjecture formulated by J. Min\'a\v{c} and N. D. T\^an in \cite{MT}.  Recall that $\U_n(\F_p)$ is the group of all upper-triangular unipotent $n \times n$ matrices with entries in $\F_p$.

\begin{defn}
Let $G$ be a pro-$p$ group and let $n\geq 1$ be an integer.  We say that $G$ has the \textit{kernel n-unipotent property} if
\[
G_{(n)}=\bigcap \ker(\rho:G\to \U_n(\F_p) ),
\]
where $\rho$ runs through the set of all representations (continuous homomorphisms) $G\to \U_n(\F_p)$.
\end{defn}

\begin{conj}[Kernel $n$-Unipotent Conjecture]
Let F be a field containing a primitive p-th root of unity and let $G=G_F(p)$.  Let $n\geq 3$ be an integer.  Then G has the kernel n-unipotent property.
\end{conj}

\begin{lem}
\label{lem:coefficient}
Let $n$ be a positive integer. If $\U_{n+1}(\F_p)_{(n)}=1$, then $c_n(S)=0$ for every free pro-$p$-group $S$.
\end{lem}

\begin{proof} Let $S$ be a free pro-$p$-group.
Assume that $\U_{n+1}(\F_p)_{(n)}=1$.  Then for  any (continuous) representation $\rho: S\to \U_{n+1}(\F_p)$, we have $\rho(S_{(n)})\subseteq \U_{n+1}(\F_p)_{(n)}=1$. Hence 
\[
S_{(n+1)}\subseteq S_{(n)}\subseteq \bigcap \ker(\rho\colon S\to \U_{n+1}(\F_p)),
\]
where $\rho$ runs over the set of all representations (continuous homomorphisms) $G\to \U_{n+1}(\F_p)$. On the other hand, we know that $S$ has the kernel $n$-unipotent property for all $n$ (see \cite{Ef1}, and  also \cite{Ef2}, \cite{MT}). This means that we have
\[
S_{(n+1)}=\bigcap \ker(\rho\colon S\to \U_{n+1}(\F_p)).
\]
Therefore, $S_{(n+1)}=S_{(n)}$, so $c_n(S)=0$.
\end{proof}

\begin{cor}
\label{cor:unipotent}
Let $n$ be a positive integer. Then 
$ \U_{n+1}(\F_p)_{(n)}\simeq \F_p$ and
\[
n=\max\{h\mid \U_{n+1}(\F_p)_{(h)}\neq 1 \}.
\]
\end{cor}

\begin{proof} We first observe that if $S$ is a free pro-$p$-group of rank $d>1$, then all numbers $w_n(S)$, $n=1,2,\ldots$, are positive.  Therefore from Proposition~\ref{prop:cn free} we see that $c_n(S)\not=0$ for all $n\in \N$. Hence by Lemma~\ref{lem:coefficient}, $\U_{(n+1)}(\F_p)_{(n)}\not=1$.

On the other hand, it is well-known that $\U_{n+1}(\F_p)_{(n+1)}=1$. Hence $\U_{n+1}(\F_p)_{(n)}\subseteq Z(\U_{n+1}(\F_p))\simeq \F_p$, where  $Z(\U_{n+1}(\F_p))$ is the center of $\U_{n+1}(\F_p)$. Therefore 
\[
\U_{n+1}(\F_p)_{(n)}=Z(\U_{n+1}(\F_p))\simeq \F_p,
\]
and the second assertion is also clear.
\end{proof}

\section{Free Products of Cyclic Groups}
\label{sec:freeprod}

Let $d$ be a non-negative integer. Let $G=C_p*\cdots * C_p$ be a free product in the category of pro-$p$-groups of  $d+1$ copies of $C_p$, where $C_p$ is the cyclic group of order $p$.  By Example \ref{ex:HPseriesCp},  Lemma~\ref{lem:free product} and induction on $d$, the Hilbert-Poincar\'e series of ${\rm gr}(\F_p[[G]])$ is
\[ 
P_{{\rm gr}(\F_p[[G]])}(t)= \frac{1+t+\cdots +t^{p-1}}{1-dt-\cdots -dt^{p-1}}.
\] 
Due to the close connection with formally real pythagorean fields, we will focus on the case in which $p=2$.

\subsection{Free products of cyclic groups of order 2}
 
Let $d$ be a non-negative integer. Let $G=C_2*\cdots * C_2$ be a free product in the category of pro-$2$-groups of  $d+1$ copies of $C_2$, where $C_2$ is the group of order 2. 

Recall from Theorem \ref{thm:SAP} that $G$ plays an important role as the maximal pro-$2$ quotient of the absolute Galois group of a formally real pythagorean SAP field.  Also it is interesting to observe that if $G$ is such a Galois group, then $G$ is already determined by its quotient $G/G_{(3)}$. More precisely, assume that $H$ is another pro-2-group which is realizable as the Galois group of the maximal 2-extension of a field $F$, and that $H/H_{(3)}\simeq G/G_{(3)}$, then  $H\simeq G$. (See \cite{MS1,MS2,Mi}.)

The Hilbert-Poincar\'e series of ${\rm gr}(\F_2[[G]])$ is
\[
P_{{\rm gr}(\F_2[[G]])}(t)= \frac{1+t}{1-dt}.
\]
We have
\[
\log P_{{\rm gr}(\F_2[[G]])}(t)= \log(\frac{1}{1-dt})-\log(\frac{1}{1+t})=\sum_{n\geq 1} \frac{1}{n} (d^n-(-1)^n)t^n.
\]
Now we define the sequence $w_n(G), n=1,2,\ldots$ by
\[
w_n(G)=\frac{1}{n}\sum_{m\mid n} \mu(n/m) (d^m-(-1)^m).
\]
\begin{prop}
\label{prop:cn free product of C2} 
If $n=2^k m $ with $(m,2)=1$, then 
\[
c_n(G) =w_m(G) +w_{2m}(G)+\cdots + w_{2^km}(G).
\qed
\]
\end{prop}

\subsection{Free products of cyclic groups of order 2 as semidirect products}

We again let $G=C_2*\cdots * C_2$ be the free product in the category of pro-2 groups of $d+1$ copies of $C_2$.  Our goal in this subsection is to show that $G$ is isomorphic to a semidirect product $H\rtimes C_2$ of a free pro-2 group $H$ and $C_2$.  We also provide a relation between $G_{(n)}$ and $H_{(n)}$.

Define the numbers $\epsilon_n$, $n=1,2,\ldots$ by
\[
\epsilon_n=\frac{1}{n}\sum_{m\mid n} \mu(n/m)(-1)^m.
\]
Then by  the M{\"o}bius inversion formula,
\[
\label{eq:epsilon}
\tag{*}(-1)^n= \sum_{m\mid n} m \epsilon_m.
\]
\begin{lem}
\label{lem:epsilon}
We have $\epsilon_1=-1$, $\epsilon_2=1$ and $\epsilon_n=0$ for $n\geq 3$.
\end{lem}
\begin{proof} The equation (\ref{eq:epsilon}) determines $\epsilon_n$, $n\in \N$, uniquely. But $\epsilon_1=-1$,  $\epsilon_2=1$ and $\epsilon_n=0$ for $n\geq 3$ work as for these numbers 
\[
\sum_{m\mid n} m \epsilon_m
=\begin{cases}
-1 &\;\text{if $n$ is odd},\\
-1+2=1 &\; \text{if $n$ is even}.
\qedhere
\end{cases}
\]
\end{proof}

Now write 
\[G=C_2*C_2*\cdots *C_2=\langle x_0\mid x_0^2\rangle*\langle x_1\mid x_1^2\rangle *\cdots * \langle x_d\mid x_d^2\rangle.\] 
For  ease of notation, we consider $x_0,x_1,\ldots, x_d$ as elements of $G$. We consider a continuous homomorphism $\varphi:G\to C_2= \langle x\mid x^2\rangle$ defined by $x_i\mapsto x$ for all  $i=0,1,\ldots,d$. For each $i=1,\ldots,d$, we set $y_i=x_0x_i\in G$ and let $H$ be the closed subgroup of $G$ generated by $y_1,\ldots,y_d$.
\begin{lem} Let the notation be as above.
\begin{enumerate} 
\item[(i)] $\ker\varphi=H$.
\item[(ii)] $G\simeq H\rtimes C_2$, where the action of $C_2$ on $H$ is given by $xy_ix=y_i^{-1}$.
\item[(iii)] $H$ is a free pro-$2$ group of rank $d$.
\end{enumerate}
\end{lem}
\begin{proof}
(i) For each $i=1,\ldots, d$, $y_i\in \ker\varphi$, hence $H\subseteq \ker\varphi$. Now consider any element $\gamma\in \ker\varphi$. For each open neighborhood $U$ of $\gamma$ in $G$, there exists an element $g=x_{i_1}\cdots x_{i_r}\in U$, $i_1,\ldots,i_r\in \{1,\ldots,d\}$ such that $1=\varphi(g)=x^r$.   Hence $r=2s$ is even. Since $x_0 y_ix_0= y_i^{-1}$, we obtain
\[
g= x_0 y_{i_1}\cdots x_0 y_{i_r}=y_{i_1}^{-1}y_{i_2}\cdots y_{i_{r-1}}^{-1}y_{i_r}.
\]
Thus $g\in H$. Therefore $\gamma\in H$ and $H=\ker\varphi$.

(ii) This follows by observing that $\psi\colon C_2=\langle x\mid x^2 \rangle \to G$ which maps $x$ to $x_0$, is a section of $\varphi$.

(iii) By Theorem \ref{thm:SAP}, we can view $G$ as the Galois group $G_F(2)$ of the maximal 2-extension of a formally real pythagorean SAP field $F$ having square class group of cardinality $2^{d+1}$.  There is a bijective correspondence between the set  $\{x_0,x_1,\ldots,x_d\}$ and the set of orderings $X_F=\{P_0,P_1,\ldots,P_d\}$ of $F$ given by
\[
P_i=\{f\in \fx\mid x_i(\sqrt{f})=\sqrt{f}\}.
\]
Let $K$ be the fixed field of $H$.  For all $i=0,\ldots,d$, we have $-1\notin P_i$, so $x_i(\sqrt{-1})=-\sqrt{-1}$.  Hence for all $i=1,\ldots,d$, $y_i=x_0x_i$ acts trivially on $\sqrt{=1}$.  So $F(\sqrt{-1})\subseteq K$.  Since $[G:H]=2$, we have $K=F(\sqrt{-1})$ and $H=G_K(2)$.  Then by Proposition \ref{prop:GK2}, $H$ is a free pro-2 group of rank $d$.  
\end{proof}

The following proposition and corollary are remarkable properties of the pair $\{H,G\}$.  
\begin{prop}
\label{prop:comparison}
 We have $c_1(H)=d=c_1(G)-1$ and $c_n(H)=c_n(G)$ for all $n\geq 2$.
\end{prop}
\begin{proof} 
It is clear that $c_1(H)=w_1(H)=d$ and $c_1(G)=w_1(G)=d+1$. Hence $c_1(H)=d=c_1(G)-1$. We shall show that $c_n(H)=c_n(G)$ for any $n\geq 2$.

We note that 
\[ w_n(H)-w_n(G)=\frac{1}{n}\sum_{m\mid n}\mu(n/m)(-1)^m =\epsilon_n.
\] 
By Lemma~\ref{lem:epsilon}, one has $w_2(H)=w_2(G)+1$ and $w_n(H)=w_n(G)$  for every $n\geq 3$.

If $n>1$ is odd, then 
\[
c_n(H)=w_n(H)=w_n(G)=c_n(G).
\]
If $n$ is even, then by writing $n=2^km$ with $m$ odd, we have
\[
\begin{aligned}
c_n(H)&= w_m(H)+w_{2m}(H) +w_{4m}(H)+\cdots+w_{2^km}(H)\\
&=w_m(G)+w_{2m}(G) +w_{4m}(G)+\cdots+w_{2^km}(G)=c_n(G).
\end{aligned}
\]
(Note that we always have $w_m(H)+w_{2m}(H)=w_m(G)+w_{2m}(G)$ for every $m\geq 1$ odd.)
\end{proof}

\begin{cor}
\label{cor:quotient}
 Let $n\geq 2$ be an integer.
\begin{enumerate}
\item[(i)] $H_{(n)}=H\cap G_{(n)}$.
\item[(ii)] $G/G_{(n)}\simeq H/H_{(n)}\rtimes C_2$, where the action of $C_2$ on $H$ is given by $\bar{x}\bar{y}_i\bar{x}=\bar{y}_i^{-1}$.
\end{enumerate}

\end{cor}
\begin{proof}
(i) Clearly $H_{(n)}\subseteq H\cap G_{(n)}$. We proceed by induction on $n$ that $H_{(n)}=H\cap G_{(n)}$. First consider the case $n=2$. We have an exact sequence
\[
1\to H/H\cap G_{(2)} \to G/G_{(2)}\to C_2\to 1.
\]
This implies that $[H: H\cap G_{(2)}]= [G:G_{(2)}]/2=2^d=[H:H_{(2)}]$. Hence $H_{(2)}=H\cap G_{(2)}$. Assume that $H_{(n)}=H\cap G_{(n)}$ for some $n\geq 2$. Then from the exact sequence
\[
1\to H/H\cap G_{(n)} \to G/G_{(n)}\to C_2\to 1,
\]
we obtain $[H:H_{(n)}]=[H:H\cap G_{(n)}]=[G:G_{(n)}]/2$. From a similar exact sequence we obtain
\[
\begin{aligned}
{[H:H\cap G_{(n+1)}]}&=\frac{1}{2}[G:G_{(n+1)}] = \frac{1}{2}[G:G_{n}] [G_{(n)}:G_{(n+1)}]\\
&= [H:H_{(n)}] [H_{(n)}:H_{(n+1)}]=[H:H_{(n+1)}].
\end{aligned}
\]
Here the equality $[G_{(n)}:G_{(n+1)}]=[H_{(n)}:H_{(n+1)}]$ follows from Proposition~\ref{prop:comparison}.
Therefore $H_{(n+1)}=H\cap G_{(n+1)}$.

(ii) This follows from (i).
\end{proof}

\section{Another Semidirect Product}
\label{sec:semidirect}

In this section we consider an example in which $G$ is the semidirect product $G:\Z_2^d \rtimes C_2=H \rtimes \langle x \rangle$, where the action of $C_2$ on $H=\Z_2^d$ is given by $xyx=y^{-1}$, for all $y\in H$.  Recall from Corollary \ref{cor:superPytha} that this group is realizable as the maximal pro-2 quotient of the absolute Galois group of a superpythagorean field.

\begin{lem} Let $G=H\rtimes \langle x \rangle=\Z_2^d\rtimes C_2$ be as above. Let $n\geq 2$ be an integer, and let $s=\lceil \log_2n \rceil$. Then $G_{(n)}=H^{2^s}$.
\end{lem}
\begin{proof} We proceed by induction on $n$. We first observe that $[y,x]=y^{-1}x^{-1} y x=(y^{-1})^2$ and $(yx)^2=1$, for every $y\in H$. Hence 
\[ 
G_{(2)}=G^2[G,G]=G^2=H^2,
\]
so the lemma is true for $n=2$. Now assume that the lemma is true for $j$ with  $2\leq j<n$. Then
\[
\begin{aligned}
 G_{(n)}&=G_{(\lceil n/2 \rceil)}^2\prod_{i+j=n}[G_{(i)},G_{(j)}]\\
&=G_{(\lceil n/2\rceil)}^2[G,G_{(n-1)}]\\
&=(H^{2^{s-1}})^2=H^{2^s}.
 \end{aligned}
\]
Here we use the fact that $G_{(n-1)}\subseteq H^{2^{s-1}}$,  and hence $[G,G_{(n-1)}]\subseteq H^{2^s}$.
\end{proof}
An immediate  consequence of the above lemma is the following result.
\begin{cor}
\label{cor:cn superPy}
 Let $n\geq 1$ be an integer. We have
\[
c_n(G)=
\begin{cases} 
d+1 &\text{ if $n=1$},\\
d &\text{ if $n=2^s$ for some $1\leq s\in \Z$},\\
1 & \text{ if $n$ is not a power of 2}.
\end{cases}
\]
\end{cor}
\begin{cor}
\label{cor:superPy}
 We have
\[
 P_{{\rm gr}(\F_2[[G]])}(t)=\frac{1+t}{(1-t)^d}\prod_{i=1}^{\infty}\frac{1}{1-t^{2i+1}}. 
 \]
\end{cor}
\begin{proof}
We write $\log  P_{{\rm gr}(\F_2[[G]])}= \sum_{n\geq 1} b_n(G) t^n $, and let 
\[
w_n(G)=\frac{1}{n}\sum_{m\mid n} \mu(n/m) mb_m(G).
\]
By Lemma~\ref{lem:coprime}, if $n$ is odd then $w_n(G)=c_n(G)$. In particular, $w_1(G)=c_1(G)=d+1$, and $w_{2i+1}(G)=c_{2i+1}(G)=1$ for $i\geq 1$. 

By Lemma~\ref{lem:not coprime}, $w_2(G)=c_2(G)-c_1(G)=d-(d+1)=-1$.

We claim that $w_n(G)=0$ if $ n$ is even and $n\geq 4$. Indeed, if $n=2^s$ with $s\geq 2$, then by Lemma~\ref{lem:not coprime}, 
\[w_{2^s}(G)=c_{2^s}(G)-c_{2^{s-1}}(G)=d-d=0.\] 
Now if $n=2m$, where $m$ is not a power of $2$, then also by Lemma~\ref{lem:not coprime}, 
\[w_{2m}(G)=c_{2m}(G)-c_{m}(G)=1-1=0.\]
The corollary then follows from Remark~\ref{rmk:wn}.
\end{proof}

\begin{rmks}
\label{rmks:SAP}
It is interesting that $c_1(G)$ and $c_2(G)$ can be sufficient to determine $G$ itself within some large families of pro-$p$-groups. The example of free pro-$p$ groups was mentioned in Remarks \ref{rmk:fgGalois} (1). Here are two other instances.

(1) Suppose that $K$ is a formally real pythagorean SAP field with $|K^\times/(K^\times)^2|=2^{d+1}$. Then $G_K(2)=C_2*\cdots *C_2$, the free product of $d+1$ copies of $C_2$. 
By Proposition~\ref{prop:cn free product of C2}, one has
\[
c_1(G_K(2))=d+1 \text { and } c_2(G_K(2))=\frac{d(d+1)}{2}.
\]

Now let $F$ be a formally real pythagorean field $F$ with $|F^\times/(F^\times)^2|< \infty$. We assume that $c_1(G_F(2))=d+1$ and that $c_2(G_F(2))=d(d+1)/2$ for some integer $d\geq 0$. 

\begin{claim}
$F$ is an SAP field with exactly $d+1$ orderings.
\end{claim}

\begin{proof}
Since $c_1(G_F(2))=d+1$, we see that $G_F(2)$ has $d+1$ minimal generators, and therefore $|F^\times/(F^\times)^2|=2^{d+1}$. 
Now choose any formally real pythagorean SAP field $K$ with $|K^\times/(K^\times)^2|=2^{d+1}$. 
By Corollary~\ref{cor:SAP}, there exists an epimorphism $\varphi\colon G_K(2)\twoheadrightarrow G_F(2)$. 
 We have
\[
\begin{aligned}
|G_K(2)/{G_K(2)}_{(3)}| &=c_1(G_K(2))+c_2(G_K(2))\\
&=d+\frac{d(d+1)}{2}\\
&=c_1(G_F(2))+c_2(G_F(2))\\
&=|G_F(2)/{G_F(2)}_{(3)}|.
\end{aligned}
\]
This implies that the induced epimorphism $G_K(2)/{G_K(2)}_{(3)}\twoheadrightarrow G_F(2)/{G_F(2)}_{(3)}$ is an isomorphism. By \cite[Theorem D]{CEM}, $\varphi \colon G_K(2)\to G_F(2)$ is an isomorphism. This implies that $F$ is a SAP field by Theorem~\ref{thm:SAP}.
\end{proof}

So, quite remarkably, within the family of formally real pythagorean fields with finitely many square classes, the numbers $c_1(G_F(2))$ and $c_2(G_F(2))$ above suffice to characterize SAP fields $F$. 

(2) Suppose that  $K$ is a superpythagorean field with $|K^\times/(K^\times)^2|=2^{d+1}<\infty$. By Corollary~\ref{cor:cn superPy}, one has
\[
c_1(G_K(2))=d+1 \; \text { and } c_2(G_K(2))=d.
\]

Now let $F$ be a formally real pythagorean field $F$ with $|F^\times/(F^\times)^2|< \infty$. We assume that $c_1(G_F(2))=d+1$, $c_2(G_F(2))=d$ for some integer $d\geq 0$.
\begin{claim}
$F$ is a superpythagorean field.
\end{claim}

\begin{proof}
Choose any superpythagorean field $K$ with $|K^\times/(K^\times)^2|=2^{d+1}$. By Corollary~\ref{cor:superPy}, we have an epimorphism $\varphi\colon G_F(2)\twoheadrightarrow G_K(2)$. 
Then
\[
\begin{aligned}
|G_F(2)/{G_F(2)}_{(3)}| &=c_1(G_F(2))+c_2(G_F(2))\\
&=d+1+d\\
&=c_1(G_K(2))+c_2(G_K(2))\\
&=|G_K(2)/{G_K(2)}_{(3)}|.
\end{aligned}
\]
This implies that the induced epimorphism $G_F(2)/{G_F(2)}_{(3)}\twoheadrightarrow G_K(2)/{G_K(2)}_{(3)}$ is an isomorphism. By \cite[Theorem D]{CEM}, $\varphi \colon G_F(2)\to G_K(2)$ is an isomorphism. This implies that $F$ is a superpythagorean field by Corollary~\ref{cor:superPytha}.
\end{proof}

So  within the family of formally real pythagorean fields with finitely many square classes, the numbers $c_1(G_F(2))$ and $c_2(G_F(2))$ above, also suffice to characterize superpythagorean fields $F$. 
\end{rmks}

\section{Demushkin Groups}
Recall that a pro-$p$-group $G$ is said to be a Demushkin group if
\begin{enumerate}
\item $\dim_{\F_p} H^1(G,\F_p)<\infty,$ 
\item $\dim_{\F_p} H^2(G,\F_p)=1,$
\item  the cup product $H^1(G,\F_p)\times H^1(G,\F_p)\to H^2(G,\F_p)$ is a non-degenerate bilinear form.
\end{enumerate}
By the work of \cite{De1,De2}, \cite{Se1} and \cite{La1}, we now have a complete classification of Demushkin groups. 

Let $G$ be a Demushkin group of rank $d=\dim_{\F_p} H^1(G,\F_p)$. Let $c_n=c_n(G)$. Then by \cite[Theorem 5.1 (g)]{La3} (see also \cite{Fo,Ga,LM}), the Hilbert-Poincar\'e series 
\[
P_{{\rm gr}(\F_p[[G]])}(t)= \frac{1}{1-dt+t^2}.
\]
We write $1-dt+t^2=(1-at)(1-bt)$ so that $a+b=d$ and $ab=1$. Then
\[
\log P_{{\rm gr}(\F_p[[G]])}(t)= \log(\frac{1}{1-at})+\log(\frac{1}{1-bt})=\sum_{n\geq 1} \frac{1}{n}(a^n+b^n).
\]
We define the sequence $w_n(G), n=1,2,\ldots$ by
\[
w_n(G)=\frac{1}{n}\sum_{m\mid n} \mu(m) (a^{n/m}+b^{n/m})=\frac{1}{n}\sum_{m\mid n} \mu(n/m) (a^m+b^m).
\]
\begin{rmk} The numbers $w_n(G)$ are given by the formula
\[
w_n(G)=\frac{1}{n}\sum_{m\mid n} \mu(n/m) \left[ \sum_{0\leq i\leq [m/2]} (-1)^i \frac{m}{m-i} {m-i \choose i} d^{m-2i} \right].
\]
(See \cite[Proof of Proposition 4]{La2}.)
 \end{rmk}

\begin{prop} If $n=p^k m $ with $(m,p)=1$, then 
\[
c_n(G) =w_m(G) +w_{pm}(G)+\cdots + w_{p^km}(G).
\qedhere
\]
\end{prop}
\begin{ex} Let $G$ be a Demushkin pro-$p$-group of finite rank $d$. We have
\[
\begin{aligned}
 c_1(G)&=d,\\
c_2(G)&= \begin{cases}
\frac{d^2-d-2}{2} \text{ if } p\not=2,\\
\frac{d^2+d-2}{2} \text{ if } p=2, 
\end{cases}\\
c_3(G) &=\begin{cases}
\frac{d^3-4d}{3} \text{ if } p\not=3,\\
\frac{d^3-d}{3} \text{ if } p=3,
\end{cases}\\
c_4(G)&= \begin{cases}
\frac{d^4-5d^2+4}{4} \text{ if } p\not=2,\\
\frac{d^4-3d^2+2d}{4} \text{ if } p=2, 
\end{cases}\\
c_5(G)&=\begin{cases}
\frac{d^5-5d^3+4d}{5} \text{ if } p\not=5,\\
\frac{d^5-5d^3+9d}{5} \text{ if } p=5. 
\end{cases}
\end{aligned}
\]

Observe that our numbers $c_n(G)$, $n=1,2,\ldots$, also detect the minimal numbers  of generators of $G_{(n)}$. Indeed by the remarkable result of I. V. Ando\v{z}skii and independently by J. Dummit and J. Labute for each open subgroup $T$ of the Demushkin group $G$, we have the following expression for the minimal number of generators $d(T)$ of $T$:
\[
 d(T)=[G:T](d(G)-2)+2.
\]
(See \cite[Theorem 3.9.15]{NSW}.)
Therefore
\[
 d_n(G):=d(G_{(n)})= p^{\sum_{i=1}^{n-1}c_i(G)}(d-2)+2. 
\]
\end{ex}

From now on we assume that $G=F/\langle r\rangle$, where $F$ is a free pro-$p$-group on generators $x_1,x_2,\ldots,x_d$, and 
\[
 r=[x_1,x_2][x_3,x_4]\cdots [x_{d-1},x_d].
\]
Then we extract from \cite{La2} the following fact.
\begin{lem}
\label{lem:wn Demushkin} For every $n$, $w_n(G)=\rank_{\Z_p} G_n/G_{n+1}$.
\end{lem}
\begin{proof} This follows from  \cite[Theorem and proof of Proposition 4]{La2} and Corollary \ref{cor:finitegen}. 
\end{proof}

\begin{cor} Assume that for each $n$, $B_n$ represents a $\Z_p$-basis of $G_n/G_{n+1}$. 
Let us write $n=p^km$ with $(m,p)=1$. Then a basis of the $\F_p$-vector space $G_{(n)}/G_{(n+1)}$ can be represented by the following set
\[
B_m^{p^k} \bigsqcup B_{pm}^{p^{k-1}}\bigsqcup \cdots \bigsqcup B_{p^{k-1}m}^p\bigsqcup B_n.
\]
\end{cor}

\section{Some Other Groups}
\subsection{Free products of a finite number of Demushkin groups and free pro-$p$-groups}
Let $G$ be a free pro-$p$ product of $r$ Demushkin groups of ranks $d_1,\ldots, d_r$ and of a free pro-$p$-group of rank $e$. 
The Hilbert-Poincar\'e series of ${\rm gr}(\F_p[[G]])$ is
\[
P_{{\rm gr}(\F_p[[G]])}(t)=\frac{1}{1-(d_1+\cdots +d_r+e)t+ rt^2}=: \frac{1}{(1-at)(1-bt)}.
\]
We define the sequence $w_n(G), n=1,2,\ldots$ by
\[
w_n(G)=\frac{1}{n}\sum_{m\mid n} \mu(m) (a^{n/m}+b^{n/m})=\frac{1}{n}\sum_{m\mid n} \mu(n/m) (a^m+b^m).
\]
\begin{prop} If $n=p^k m $ with $(m,p)=1$, then 
\[
c_n(G) =w_m(G) +w_{pm}(G)+\cdots + w_{p^km}(G).
\qedhere
\]
\end{prop}

\subsection{A free product of a cyclic group of order 2  and a free pro-$2$-group}
We first consider the case of $p=2$  because this is  the case of interest in Galois theory (of $2$-extensions), and because this case is a bit simpler than the general case of any prime $p$. This latter case will be covered in the next subsection.

Let $G=C_2* S$ be a free pro-$2$ product of the cyclic group $C_2$ of order 2 and a free  pro-$2$-group of rank $d$.
The Hilbert-Poincar\'e series of ${\rm gr}(\F_2[[G]])$ is
\[
P_{{\rm gr}(\F_2[[G]])}(t)=(\frac{1}{1+t}-dt)^{-1}=\frac{1+t}{1-dt- dt^2}=: \frac{1+t}{(1-at)(1-bt)}.
\]
We define the sequence $w_n(G), n=1,2,\ldots$ by
\[
w_n(G)=\frac{1}{n}\sum_{m\mid n} \mu(n/m) (a^m+b^m-(-1)^m).
\]
\begin{prop} If $n=2^k m $ with $(m,2)=1$, then 
\[
c_n(G) =w_m(G) +w_{pm}(G)+\cdots + w_{2^km}(G).
\qedhere
\]
\end{prop}

\subsection{A free product of a cyclic group of order $p$  and a free pro-$p$-group}
Let $G=C_p* S$ be a free pro-$p$ product of the cyclic group $C_p$ of order $p$ and a free  pro-$p$-group of rank $d$. We shall find a formula for $c_n(G)$. 
The Hilbert-Poincar\'e series of ${\rm gr}(\F_p[[G]])$ is
\[
P_{{\rm gr}(\F_p[[G]])}(t)=\frac{1+t+\cdots+t^{p-1}}{1-dt- dt^2-\cdots-dt^p}=: \frac{(1-\xi_1 t)\cdots(1-\xi_{p-1}t)}{(1-a_1t)\cdots (1-a_pt)}.
\]
We define the sequence $w_n(G), n=1,2,\ldots$ by
\[
w_n(G)=\frac{1}{n}\sum_{m\mid n} \mu(n/m)( a_1^m+\cdots+a_p^m-(\xi_1^m+\cdots+\xi_{p-1}^m)).
\]
\begin{prop} If $n=p^k m $ with $(m,p)=1$, then 
\[
c_n(G) =w_m(G) +w_{pm}(G)+\cdots + w_{p^km}(G).
\qedhere
\]
\end{prop}
\begin{rmk}
Note that
\[
\xi_1^n+\cdots+\xi_{p-1}^n
=\begin{cases}
-1 \text{ if } (n,p)=1,\\
p-1 \text{ if } p\mid n.
\end{cases}
\]
We shall compute $a_1^n+\cdots+a_{p}^n$.
From
\[
  \frac{1}{(1-a_1t)\cdots(1-a_pt)}=\frac{1}{1-(dt+dt^2+\cdots+dt^p)},
\]
taking logarithms of both sides, we obtain
\[
\begin{aligned}
& \sum_{n\geq 1}\frac{1}{n}(a_1^n+\cdots+a_p^n)t^n= \sum_{n\geq 1} \frac{1}{n} (dt+dt^2+\cdots +dt^p)^n\\
 &= \sum_{n\geq 1} \frac{1}{n}\sum_{\substack{k_1+\cdots+k_p=n,\\k_i \geq 0}} \binom{n}{k_1,\ldots,k_p} (dt)^{k_1} (dt^2)^{k_2}\cdots (dt^p)^{k_p}\\
&= \sum_{M}\sum_{\substack{k_1+2k_2+\cdots+pk_p=M,\\k_i \geq 0}}\\
&\hspace*{60pt}
\left[\frac{1}{M-k_2-\cdots-(p-1)k_p} \binom{M-k_2-\cdots-(p-1)k_p}{k_1,\ldots,k_p} d^{M-k_2-\cdots-(p-1)k_p}\right] t^M.
\end{aligned}
\]
Finally comparing the coefficients of $t^n$ gives us the required formula for $a_1^n+\cdots+a_{p}^n$,
\[
\begin{aligned}
 &a_1^n+\cdots+a_p^n\\
&=\sum_{\substack{k_1+2k_2+\cdots+pk_p=n,\\k_i \geq 0}}\frac{n}{n-k_2-\cdots-(p-1)k_p} \binom{n-k_2-\cdots-(p-1)k_p}{k_1,\ldots,k_p} d^{n-k_2-\cdots-(p-1)k_p}.
 \end{aligned}
\]
\end{rmk}

\chapter{Relations in Pro-$p$ Galois Groups}
\label{ch:relations}

As we have seen, Demushkin groups are pro-$p$ groups which play an important role in Galois theory.  These groups have a single defining relation among a finite set of generators and appear as Galois groups of maximal $p$-extensions of local fields containing a primitive $p$-th root of unity.  The study of relations in Demushkin groups has a long and interesting history and we saw in section \ref{sec:mobius} that the unique relation among generators of a Demushkin group can take one of only four rather special forms.  One example of a specific Demushkin relation is
\[
r=x_1^p[x_1,x_2][x_3,x_4]\cdots[x_{d-1},x_d].
\]
The question arises as to whether other pro-$p$ Galois groups have similar restrictions on the shape of relations.  That is, if $G$ is a pro-$p$ group which is realizable as the Galois group $G_F(p)$ of the maximal $p$-extension of a field $F$, must the relations in $G$ take on only certain forms?  

In this chapter we begin to explore that question and will see that even considering only small abelian extensions of $F$ can provide insight into necessary restrictions on relations in pro-$p$ Galois groups.  Further work to extend these results and thereby provide a better understanding of the structure of absolute Galois groups is in progress.

\section{The Case $\zeta _{p^2}\in F$ }   

In this section we assume that $p$ is a prime and $F$ is a field containing a primitive $p^2$-th root of unity $\zeta _{p^2}$.

\begin{thm}
\label{thm:x1x2}
Let $p$ be a prime and $F$ a field containing a primitive $p^2$-th root of unity $\zeta _{p^2}$.  There is no relation $r=x_1^px_2^ps \in R$, where $ S=<x_1, x_2,\ldots >$ is a free pro-$p$ group, $ s\in [S,S] $ and 
\[
\begin{tikzcd}
1 \arrow{r} &R \arrow{r} &S \arrow{r}{\pi} &G_F(p) \arrow{r} &1
\end{tikzcd}
\]
is a minimal presentation of $G_F(p)$.
\end{thm}

\begin{proof}
Suppose such an $r$ exists and consider the field $L=F( \sqrt[p^2]{a} \mid a\in \fx )$.  By Kummer theory, $\Gal(L/F)$ is the Pontrjagin dual $ (\fx/ (\fx)^{p^2})^*$ of $ \fx/ (\fx)^{p^2}\cong (\bigoplus_{J} C_{p})\bigoplus (\bigoplus_{I} C_{p^2}) $ \cite[p. 17, Theorem 6]{Kap}.  Since $\zeta _{p^2}\in F$, every cyclic extension $F(\sqrt[p]{a})/F $ of degree $p$ embeds into a cyclic extension $K=F(\sqrt[p^2]{a})/F $ of degree $p^2$, so $J=\phi$ and 
\[
\begin{array}{rl}
\Gal(L/F) &\cong (\bigoplus_{I} C_{p^2})^*\\
&= \prod_{I} C_{p^2}.\\
\end{array}
\]
Let $\sigma_i$ be the image of $\pi(x_i),\ i=1,2$ under the restriction map
\[
\begin{tikzcd}
S \arrow{r}{\pi} &G_F(p) \arrow{r}{res} &\Gal(L/F). 
\end{tikzcd}
\]
Then
\[
\begin{array}{rl}
1 &=res (\pi (r))\\ 
&=res(\pi(x_1)^p \pi(x_2)^p \pi(s))\\
&=(res\ \pi(x_1))^p (res\ \pi(x_2))^p\\
&=\sigma_1^p \sigma_2^p\\
&=(\sigma_1 \sigma_2)^p.
\end{array}
\]
However, since 
\[
\begin{tikzcd}
1 \arrow{r} &R \arrow{r} &S \arrow{r}{\pi} &G_F(p) \arrow{r} &1
\end{tikzcd}
\]
is a minimal presentation of $G_F(p)$, it follows that $\pi(x_1)$ and $\pi(x_2)$ are independent modulo the Frattini subgroup $\Phi (G_F(p))$.  So $\sigma_1 \sigma_2$ has order $p^2$ in $\Gal(L/F)$, which is a contradiction.
\end{proof}

\begin{thm}
Let $p$ be a prime and $F$ a field containing a primitive $p^2$-th root of unity $\zeta _{p^2}$.  There is no relation $r=x_1^ps \in R$, where $ S=<x_1, x_2,\ldots >$ is a free pro-$p$ group, $ s\in [S,S] $ and 
\[
\begin{tikzcd}
1 \arrow{r} &R \arrow{r} &S \arrow{r}{\pi} &G_F(p) \arrow{r} &1
\end{tikzcd}
\]
is a minimal presentation of $G_F(p)$.
\end{thm}

\begin{proof}
As in the proof of Theorem \ref{thm:x1x2}, consider the field $L=F( \sqrt[p^2]{a} \mid a\in \fx )$ with Galois group $\Gal(L/F) \cong \prod_{I} C_{p^2}$.  If $\sigma$ is the image of $\pi(x_1)$ under the restriction map
\[
\begin{tikzcd}
S \arrow{r}{\pi} &G_F(p) \arrow{r}{res} &\Gal(L/F), 
\end{tikzcd}
\]
then
\[
\begin{aligned}
1 &=res (\pi (r))\\ 
&=res(\pi(x_1)^p \pi(s))\\
&=(res\ \pi(x_1))^p\\
&=\sigma^p.
\end{aligned}
\]
But since $\pi(x_1)\notin \Phi(G_F(p))$, $\sigma$ generates a cyclic subgroup of $\Gal(L/F)$ of order $p^2$, so this is a contradiction. 
\end{proof}

\section{The Case $\zeta _p\in F,\ \zeta _{p^2}\notin F$}

In this section we assume that $p$ is an odd prime and $F$ is a field which contains a primitive $p$-th root but not a primitive $p^2$-th root of unity.\\

In addition to the dihedral group $D_4$, other small nonabelian $p$-groups also play a fundamental role in the theory of Galois $p$-extensions.  The modular group $M_{p^3}$ is the unique nonabelian group of order $p^3$ and exponent $p^2$ given, in terms of generators and relations, by
\[
M_{p^3}=\langle x,y\mid x^{p^2}=y^p=1,\ yxy^{-1}=x^{1+p} \rangle=\langle x \rangle \rtimes \langle y \rangle.
\]

\begin{thm}
Let  p be an odd prime and let F be a field containing a primitive p-th root of unity $\zeta_p$, but no primitive $p^2$-th root of unity.  Let $ S=<x_1, x_2,\ldots >$ be a free pro-$p$ group such that 
\[
\begin{tikzcd}
1 \arrow{r} &R \arrow{r} &S \arrow{r}{\pi} &G_F(p) \arrow{r} &1
\end{tikzcd}
\]
is a minimal presentation of $G_F(p)$.  Then there is no relation of the form $r=x_1^ps \in R$, where $ s\in [S,S] $ is such that any commutator of the form $[x_i,x_j]$ appearing in the expression for $s$ has $i\neq 1$ and $j\neq 1$. 
\end{thm}

\begin{proof}
Suppose there is such a relation $r$.  For $k=2,3$, choose a primitive $p^k$-th root of unity such that $\zeta_{p^k}^p=\zeta_{p^{k-1}}$.  Consider the field $K=F(\zeta_{p^3})$.  Then $\pi(x_1)(\zeta_{p^2})=\zeta_{p^2}$; otherwise $\rho(\pi(x_1))$ would generate the entire Galois group $\Gal(K/F)\cong C_{p^2}$, contradicting
\[
\rho(\pi(x_1))^p=\rho(\pi(x_1)^p\pi(s))=\rho(r)=1,
\]
where $\rho: G_F(p) \surj \Gal(K/F)$ is the restriction map.

Since the presentation
\[
\begin{tikzcd}
1 \arrow{r} &R \arrow{r} &S \arrow{r}{\pi} &G_F(p) \arrow{r} &1
\end{tikzcd}
\]
is minimal, $\pi(x_1)\notin \Phi(G_F(p))$.  Hence there exists an element $a\in F\setminus F^p$ such that $\pi(x_1)(\sqrt[p]{a})\neq \sqrt[p]{a}$.  

The polynomial $x^{p^2}-a$ is irreducible over $F(\zeta_{p^2})$ so the extension $L:=F(\zeta_{p^2}, \sqrt[p^2]{a})$ is Galois over $F$.  Define automorphisms $\sigma, \tau \in \Gal(L/F)$ by
\[
\tau:\sqrt[p^2]{a}\mapsto \zeta_{p^2}\sqrt[p^2]{a}\quad \text{and}\quad \tau:\zeta_{p^2}\mapsto \zeta_{p^2};
\]
\[
\sigma: \sqrt[p^2]{a}\mapsto \sqrt[p^2]{a}\quad \text{and}\quad \sigma:\zeta_{p^2}\mapsto \zeta_p \zeta_{p^2}=\zeta_{p^2}^p \zeta_{p^2}=\zeta_{p^2}^{p+1}. 
\]
So 
\[
\begin{aligned}
&\sigma\tau\sigma^{-1}(\zeta_{p^2})=\zeta_{p^2}\quad \text{and}\\
&\sigma\tau\sigma^{-1}(\sqrt[p^2]{a})=\sigma\tau(\sqrt[p^2]{a})=\sigma(\zeta_{p^2}\sqrt[p^2]{a})=\zeta_{p^2}^{1+p}\sqrt[p^2]{a}=\tau^{1+p}(\sqrt[p^2]{a}).
\end{aligned}
\]
Thus, $\tau$ has order $p^2$, $\sigma$ has order $p$ and $\sigma\tau\sigma^{-1}=\tau^{1+p}$, so 
\[
\Gal(L/F)=\langle \sigma,\tau \mid \tau^{p^2}=\sigma^p=1,\ \sigma\tau\sigma^{-1}=\tau^{1+p} \rangle \cong M_{p^3}.
\]
Now let $res:G_F(p)\surj \Gal(L/F)$ be the restriction map.  Since $\pi(x_1)(\zeta_{p^2})=\zeta_{p^2}$ and $\pi(x_1)(\sqrt[p]{a})\neq \sqrt[p]{a}$, it follows that $res(\pi(x_1))$ generates $\Gal(L/F(\zeta_{p^2}))$.  Then, since no commutator $[x_i,x_j]$ in the expression for $s$ involves $x_1$, we have
\[
\begin{aligned}
1 &=res(\pi(r))\\
&=res(\pi(x_1)^p\pi(s))\\
&=res(\pi(x_1))^p\\
&=\tau^p,
\end{aligned}
\]
which is a contradiction, since $\tau$ has order $p^2$ in $\Gal(L/F)$.
\end{proof}

We see that for a field $F$ containing a $p$-th root of unity, the presence or absence of a $p^2$-th root of unity in $F$ is already enough to place certain restrictions on the relations among a set of generators of the Galois group of the maximal $p$-extension of $F$.  We look forward to the results of further research in this direction.

\chapter{Conclusion}

In our quest for a deeper understanding of absolute Galois groups we are led to the study of small quotients and central filtrations of these large and largely mysterious objects.  By developing techniques to count Galois $p$-extensions and calculate filtration subquotient dimensions we can, in turn, shed more light on these underlying structures.  Answers to these Galois theoretic problems also have important ramifications in other areas of mathematics.

We have described several known techniques for counting finite Galois $p$-extensions of local fields.  While actually constructing small 2-extensions of the $p$-adic integers is feasible, this would rapidly become unwieldy for local fields with larger square class groups.  A technique involving complex characters and M\"obius functions used by Yamagishi to enumerate the Galois extensions of a local field having a given Galois group $G$ requires knowledge of the character table of $G$ as well as its subgroups.  This also becomes a significant obstacle as the order of $G$ increases.  

With these limitations in mind, we have explored other approaches and shown that by using some deeper results from Galois cohomology and the theory of quadratic forms and quaternion algebras one can develop more efficient combinatorial techniques.  For example, we find that the number $\sN$ of Galois extensions of a formally real pythagorean SAP field with square class group of cardinality $2^n$ having Galois group the dihedral group of order 8 is given by the simple formula
\[
\sN = 2^{n-2}\sum_{k=2}^{n}(_k ^n)(2^{k-1}-1). 
\]

Our main results pertain to the Zassenhaus filtration of finitely generated pro-$p$ groups.  Building on a remarkable theory of Jennings and Lazard, we have developed a method for determining the $\F_p$-dimension of subquotients of this filtration and derived explicit formulas applicable to various families of groups, including free pro-$p$ groups, Demushkin groups and free pro-2 products of finitely many copies of the cyclic group of order 2.  

These subquotient dimensions, $c_n(G)$, provide an important contribution to our knowledge of group theory and Galois theory.  For example, in several significant cases they determine the minimal number of generators of the Zassenhaus subgroups of $G$,  and if $F$ is a pythagorean SAP field or a superpythagorean field, only two of these dimensions, $c_1(G_F(2))$ and $c_2(G_F(2))$ are needed in order to determine the Galois group of the maximal $p$-extension of $F$.  

They are also of considerable interest in current Galois theory research, such as that involving the Kernel Unipotent Conjecture.  If $G$ is isomorphic to the maximal pro-$p$ quotient $G_F(p)$ of the absolute Galois group $G_F$ of a field $F$ and if $F_{(n)}$ denotes the fixed field of $G_F(p)_{(n)}$, then $|\Gal(F_{(n)}/F)|=p^{\sum_{i=1}^{n-1}c_i(G)}$.  If the Kernel Unipotent Conjecture is true, we would obtain a characterization of $G_F(p)_{(n)}$, $n\geq 3$, as the intersection of the kernels of all Galois representations $\rho:G_F(p)\to \U_n(\F_p)$.  In the case when $G_F(p)$ is finitely generated, knowledge of $|\Gal(F_{(n)}/F)|$ would be useful in order to check whether the intersection of the kernels of given representations is in fact $G_F(p)_{(n)}$.

Another interesting area of research is the investigation of the shape of relations in pro-$p$ Galois groups, which is closely related to detailed knowledge of small quotients of these groups.  Further work is planned in this direction.

Certainly, much progress has been made since the time of \'Evariste Galois and much remains to be done.  The journey so far has been fascinating and we look forward to the many miles not yet travelled.

\end{spacing}

\newpage
\addcontentsline{toc}{chapter}{Bibliography}
\bibliographystyle{alpha}
\bibliography{thesis}

\begin{thebibliography}{MNQD77}

\bibitem[Ara75]{arason}
J.~K. Arason.
\newblock Cohomologische invarianten quadratischen \uppercase{F}ormen.
\newblock {\em J. Algebra}, 36:448--491, 1975.

\bibitem[Art27]{Artin}
E.~Artin.
\newblock \"{U}ber die {\uppercase{z}}erlegung definiter
  {\uppercase{f}}unktionen in {\uppercase{q}}adrate.
\newblock {\em Abh. Math. Sem. Univ. Hamburg}, 5:100--115, 1927.

\bibitem[AS27]{AS}
E.~Artin and O.~Schreier.
\newblock Algebraische {\uppercase{k}}onstruktion reeller
  {\uppercase{k}}\"{o}rper.
\newblock {\em Abh. Math. Sem. Univ. Hamburg}, 5:85--99, 1927.

\bibitem[Bec74]{Be1}
E.~Becker.
\newblock Euklidische {\uppercase{k}}{\"{o}}rper und euklidische
  {\uppercase{h}}{\"{u}}llen von {\uppercase{k}}{\"{o}}rpern.
\newblock {\em J. Reine Angew. Math.}, 268/269:41--52, 1974.

\bibitem[Bec78]{Be}
E.~Becker.
\newblock {\em Hereditarily-Pythagorean fields and orderings of higher level}.
\newblock Number~29 in Monografias de matem\'atica. IMPA Lecture Notes, 1978.

\bibitem[BG75]{Bend}
E.~A. Bender and J.~R. Goldman.
\newblock On the applications of \uppercase{M}\"{o}bius inversion in
  combinatorial analysis.
\newblock {\em The American Mathematical Monthly}, 82:789--803, 1975.

\bibitem[CEM12]{CEM}
S.~K. Chebolu, I.~Efrat, and J.~Min\'a\v{c}.
\newblock Quotients of absolute \uppercase{G}alois groups which determine the
  entire \uppercase{G}alois cohomology.
\newblock {\em Math. Ann.}, 352(1):205--221, 2012.

\bibitem[Dem61]{De1}
S.~P. Demushkin.
\newblock The group of the maximal $p$-extension of a local field.
\newblock {\em Izv. Akad. Nauk. SSSR Ser. Mat.}, 25:329--346, 1961.
\newblock (Russian).

\bibitem[Dem63]{De2}
S.~P. Demushkin.
\newblock On 2-extensions of a local field.
\newblock {\em Mat. Sibirsk Z.}, 4:951--955, 1963.
\newblock (Russian).

\bibitem[DSMS99]{DDMS}
J.~D. Dixon, M.~P. F.~Du Sautoy, A.~Mann, and D.~Segal.
\newblock {\em Analytic pro-$p$ groups}.
\newblock Number~61 in Cambridge Studies in Advanced Mathematics. Cambridge
  University Press, 2nd edition, 1999.

\bibitem[Dwy75]{Dwy}
W.~G. Dwyer.
\newblock Homology, {\uppercase{m}}assey products and maps between groups.
\newblock {\em J. Pure Appl. Algebra}, 6:177--190, 1975.

\bibitem[Efr14a]{Ef2}
I.~Efrat.
\newblock Filtrations of free groups as intersections.
\newblock {\em Arch. Math. (Basel)}, 103:411--420, 2014.

\bibitem[Efr14b]{Ef1}
I.~Efrat.
\newblock The \uppercase{Z}assenhaus filtration, \uppercase{M}assey products,
  and representations of profinite groups.
\newblock {\em Adv. Math.}, 263:389--411, 2014.

\bibitem[EM11a]{EM}
I.~Efrat and J.~Min\'a\v{c}.
\newblock Galois groups and cohomological functors.
\newblock {\em arXiv: 1103.1508}, 2011.
\newblock To appear in the Transactions of the American Mathematical Society.

\bibitem[EM11b]{EfMin}
I.~Efrat and J.~Min\'a\v{c}.
\newblock On the descending central sequence of absolute \uppercase{G}alois
  groups.
\newblock {\em American J. Math.}, 133:1503--1532, 2011.

\bibitem[Ers11]{Er}
M.~Ershov.
\newblock Kazhdan quotients of \uppercase{G}olod-\uppercase{S}hafarevich
  groups.
\newblock {\em Proc. Lond. Math. Soc. (3)}, 102(4):599--636, 2011.

\bibitem[For11]{Fo}
P.~Forr\'e.
\newblock Strongly free sequences and pro-$p$ groups of cohomological dimension
  2.
\newblock {\em J. Reine Angew. Math.}, 658:173--192, 2011.

\bibitem[Fr{\"{o}}85]{Fr}
A.~Fr{\"{o}}hlich.
\newblock Orthogonal representations of {\uppercase{g}}alois groups,
  {\uppercase{s}}tiefel-{\uppercase{w}}hitney classes and
  {\uppercase{h}}asse-{\uppercase{w}}itt invariants.
\newblock {\em J. Reine Angew. Math.}, 360:84--123, 1985.

\bibitem[G{\"{a}}r11]{Ga}
J.~G{\"{a}}rtner.
\newblock {\em Mild pro-$p$ groups with trivial cup product}.
\newblock PhD thesis, Universit\"at Heidelberg, 2011.

\bibitem[Hal50]{Ha}
M.~Hall.
\newblock A basis for free {\uppercase{l}}ie rings and higher commutators in
  free groups.
\newblock {\em Proc. Amer. Math. Soc.}, 1:575--581, 1950.

\bibitem[Har90]{Har}
D.~Haran.
\newblock Closed subgroups of $g(\mathbb{Q})$ with involutions.
\newblock {\em J. Algebra}, 129(2):393--411, 1990.

\bibitem[ILF97]{Ishk}
V.~V. Ishkhanov, B.~B. Lur\'{e}, and D.~K. Faddeev.
\newblock {\em The Embedding Problem in Galois Theory}, volume 165 of {\em
  Translations of Mathematical Monographs}.
\newblock American Mathematical Society, 1997.

\bibitem[JLY03]{JLY}
C.~U. Jensen, A.~Ledet, and N.~Yui.
\newblock {\em Generic Polynomials: Constructive Aspects of the Inverse Galois
  Problem}, volume~45 of {\em Mathematical Sciences Research Institute
  Publications}.
\newblock Cambridge University Press, 2003.

\bibitem[Kap69]{Kap}
I.~Kaplansky.
\newblock {\em Infinite Abelian Groups}.
\newblock University of Michigan Press, 1969.

\bibitem[Koc02]{Ko}
H.~Koch.
\newblock {\em Galois theory of \textit{p}-extensions}.
\newblock Springer Monographs in Mathematics. Springer-Verlag, 2002.

\bibitem[Lab66]{La1}
J.~Labute.
\newblock Classification of {\uppercase{d}}emushkin groups.
\newblock {\em Canadian J. Math.}, 19:106--132, 1966.

\bibitem[Lab70]{La2}
J.~Labute.
\newblock On the descending central series of groups with a single defining
  relation.
\newblock {\em J. Algebra}, 14:16--23, 1970.

\bibitem[Lab06]{La3}
J.~Labute.
\newblock Mild pro-$p$-groups and {\uppercase{g}}alois groups of $p$-extensions
  of $\mathbb{Q}$.
\newblock {\em J. Reine Angew. Math.}, 596:155--182, 2006.

\bibitem[Lam83]{Lam83}
T.~Y. Lam.
\newblock {\em Orderings, valuations and quadratic forms}.
\newblock Number~52 in CBMS Regional Conference Series in Mathematics. American
  Mathematical Society, 1983.

\bibitem[Lam05]{Lam}
T.~Y. Lam.
\newblock {\em Introduction to Quadratic Forms over Fields}.
\newblock American Mathematical Society, 2005.

\bibitem[Led05]{Ledet}
A.~Ledet.
\newblock {\em Brauer Type Embedding Problems}.
\newblock Fields Institute Monographs. American Mathematical Society, 2005.

\bibitem[Lem74]{Le}
J.-M. Lemaire.
\newblock {\em Alg\`ebres connexes et homologie des espaces de lacets}.
\newblock Number 422 in Lecture Notes in Mathematics. Springer-Verlag, 1974.

\bibitem[Lic80]{Li}
A.~I. Lichtman.
\newblock On \uppercase{L}ie algebras of free products of groups.
\newblock {\em J. Pure Appl. Algebra}, 18(1):67--74, 1980.

\bibitem[LM11]{LM}
J.~Labute and J.~Min\'a\v{c}.
\newblock Mild pro-2 groups and 2-extensions of $\mathbb{Q}$ with restricted
  ramification.
\newblock {\em J. Algebra}, 332:136--158, 2011.

\bibitem[Mas87]{Ma}
R.~Massy.
\newblock Construction de $p$-extensions galoisiennes d'un corps de
  caract\'{e}ristique diff\'{e}rent de $p$.
\newblock {\em J. Algebra}, 109(2):508--535, 1987.

\bibitem[Mer81]{Merk}
A.~Merkurjev.
\newblock On the norm residue symbol of degree 2.
\newblock {\em Soviet Math. (Doklady)}, 24:546--551, 1981.

\bibitem[Min86]{Mi}
J.~Min\'a\v{c}.
\newblock Galois groups of some 2-extensions of ordered fields.
\newblock {\em C. R. Math. Rep. Acad. Sci. Canada}, 8(2):103--108, 1986.

\bibitem[MNQD77]{MN}
R.~Massy and T.~Nguyen-Quang-Do.
\newblock Plongement d'une extension de degr\'{e} $p^2$ dans une surextension
  non ab\'{e}lienne de degr\'{e} $p^3$: \'{e}tude locale-globale.
\newblock {\em J. Reine Angew. Math.}, 291:149--161, 1977.

\bibitem[MRT15]{MRT}
J.~Min\'a\v{c}, M.~Rogelstad, and N.~D. T\^an.
\newblock Dimensions of \uppercase{Z}assenhaus filtration subquotients of some
  pro-$p$ groups.
\newblock {\em arXiv: 1405.6980v2}, 2015.
\newblock To appear in the Israel Journal of Mathematics.

\bibitem[MS90]{MS1}
J.~Min\'a\v{c} and M.~Spira.
\newblock Formally real fields, pythagorean fields, {\uppercase{c}}-fields and
  {\uppercase{w}}-groups.
\newblock {\em Math. Z.}, 205(4):519--530, 1990.

\bibitem[MS96]{MS2}
J.~Min\'a\v{c} and M.~Spira.
\newblock Witt rings and {\uppercase{g}}alois groups.
\newblock {\em Ann. of Math.}, 144(1):35--60, 1996.

\bibitem[MT13]{MT}
J.~Min\'a\v{c} and N.~D. T\^an.
\newblock Triple \uppercase{M}assey products and \uppercase{G}alois theory.
\newblock {\em ar\uppercase{X}iv: 1307.6624}, 2013.
\newblock To appear in the Journal of the European Mathematical Society.

\bibitem[MT14]{MT1}
J.~Min\'a\v{c} and N.~D. T\^an.
\newblock Counting \uppercase{G}alois $\mathbb{U}_4(\mathbb{F}_p)$-extensions
  using \uppercase{M}assey products.
\newblock {\em arXiv: 1408.2586}, 2014.

\bibitem[MT15]{MTE}
J.~Min\'a\v{c} and N.~D. T\^an.
\newblock The \uppercase{K}ernel \uppercase{U}nipotent \uppercase{C}onjecture
  and \uppercase{M}assey products on an odd rigid field.
\newblock {\em Adv. Math.}, 273:242--270, 2015.
\newblock (with an appendix by I. Efrat, J. Min\'a\v{c} and N. D. T\^an).

\bibitem[Nai95]{Naito}
H.~Naito.
\newblock Dihedral extensions of degree 8 over the rational {$p$}-adic fields.
\newblock {\em Proc. Japan Acad. Ser. A. Math. Sci.}, 71, 1995.

\bibitem[NSW08]{NSW}
J.~Neukirch, A.~Schmidt, and K.~Wingberg.
\newblock {\em Cohomology of Number Fields}.
\newblock Number 323 in A Series of Comprehensive Studies in Mathematics.
  Springer-Verlag, second edition, 2008.

\bibitem[PP05]{PP}
A.~Polishchuk and L.~Positselski.
\newblock {\em Quadratic algebras}.
\newblock Number~37 in University Lecture Series. American Mathematical
  Society, 2005.

\bibitem[Qui68]{Qu}
D.~G. Quillen.
\newblock On the associated graded ring of a group ring.
\newblock {\em J. Algebra}, 10:411--418, 1968.

\bibitem[Rot64]{Rota}
G.-C. Rota.
\newblock On the foundations of combinatorial theory.
\newblock {\em Z. Wahrscheinlichkeitstheorie}, 2:340--368, 1964.

\bibitem[Sch85]{scharlau}
W.~Scharlau.
\newblock {\em Quadratic and Hermitian Forms}.
\newblock Number 270 in A Series of Comprehensive Studies in Mathematics.
  Springer-Verlag, 1985.

\bibitem[Ser02]{Se2}
J.-P. Serre.
\newblock {\em Galois cohomology}.
\newblock Springer Monographs in Mathematics. Springer, 2002.
\newblock Corrected second printing.

\bibitem[Ser63]{Se1}
J.-P. Serre.
\newblock Structures de certain pro-$p$ groups.
\newblock In {\em S\'em. Bourbaki, expos\'e 252}, 1962/63.

\bibitem[Sha47]{Sha}
I.~R. Shafarevich.
\newblock On $p$-extensions.
\newblock {\em Math. Sb.}, 20:351--363, 1947.
\newblock (Russian).

\bibitem[Sri95]{Srin}
V.~Srinivas.
\newblock {\em Algebraic \uppercase{K}-theory}.
\newblock Birkh\"auser, second edition, 1995.

\bibitem[Wal61]{Wall}
G.~E. Wall.
\newblock Some applications of the \uppercase{E}ulerian functions of a finite
  group.
\newblock {\em J. Austral. Math. Soc.}, 2:35--59, 1961.

\bibitem[War78]{Wa}
R.~Ware.
\newblock When are {\uppercase{w}}itt rings group rings?
  {\uppercase{i}}{\uppercase{i}}.
\newblock {\em Pacific J. Math.}, 76(2):541--564, 1978.

\bibitem[War79]{War1}
R.~Ware.
\newblock Quadratic forms and profinite 2-groups.
\newblock {\em J. Alg.}, 58:227--237, 1979.

\bibitem[Wit37a]{Wit}
E.~Witt.
\newblock Theorie der quadratischen \uppercase{F}ormen in beliebigen
  \uppercase{K}\"orpern.
\newblock {\em J. Reine Angew. Math.}, 176:31--44, 1937.

\bibitem[Wit37b]{witt}
E.~Witt.
\newblock Treue {\uppercase{d}}arstellung {\uppercase{l}}iescher
  {\uppercase{r}}inge.
\newblock {\em J. Reine Angew. Math.}, 177:152--160, 1937.

\bibitem[Yam95]{yam}
M.~Yamagishi.
\newblock On the number of {\uppercase{g}}alois {$p$}-extensions of a local
  field.
\newblock {\em Proc. Amer. Math. Soc.}, 123(8):2373--2380, August 1995.

\bibitem[Zas40]{Zas}
H.~Zassenhaus.
\newblock Ein \uppercase{V}erfahren, jeder endlichen $p$-\uppercase{G}ruppe
  einen \uppercase{L}ie-\uppercase{R}ing mit der \uppercase{C}haracteristic $p$
  zuzuordnen.
\newblock {\em Abh. Mat. Sem. Univ. Hamburg}, 13:200--207, 1940.

\end{thebibliography}
\newpage
\addcontentsline{toc}{chapter}{Curriculum Vitae}

\pagestyle{plain}
\begin{center}
\textbf{{\large CURRICULUM VITAE}}
\end{center}
\vspace{.5in}
\begin{tabular}[t]{lll}
\textbf{Name:} & Michael L. Rogelstad& \\
& & \\
\textbf{Post-secondary} & The University of Western Ontario& \\
\textbf{Education and} & London, Ontario, Canada& \\
\textbf{Degrees:} & & \\
& M.D. (cum laude)& 1983\\
& & \\
& FRCSC Ophthalmology& 1988\\
& & \\
& M.Sc. Physics& 1996\\
& & \\
& B.Sc. Honors Mathematics& 2005\\
& & \\
& M.Sc. Mathematics& 2010\\
& & \\
& Ph.D. Mathematics& 2015\\
& & \\
\textbf{Honours and} & UWO Board of Governors Scholarship& 1977-1981\\ 
\textbf{Awards:} & & \\
& J.A.F. Stevenson Memorial Scholarship& 1982\\
& & \\
& Horner Medal in Ophthalmology& 1983\\
& & \\
& Angela Armitt Gold Medal& 2004\\
& & \\
& Western Graduate Research Scholarship& 2009\\
& & \\
& Ontario Graduate Scholarship& 2010-2012\\
& & \\
& NSERC Alexander Graham Bell& 2012-2014\\
& Canada Graduate Scholarship D& \\
&&\\
\textbf{Related Work} & Teaching Assistant& 2009-2013\\
\textbf{Experience:} & The University of Western Ontario & \\
&&\\
\textbf{Publications:}& &\\
\end{tabular}

\begin{itemize}
\item F. B. Yousif, J. B. A. Mitchell, M. Rogelstad, A. Le Paddelec, A. Canosa, and M. I. Chibisov (1994).  \textit{Dissociative recombination of $HeH^+$: a re-examination.}  Phys. Rev. A 49, 4610-4615.
\item F. B. Yousif, M. Rogelstad, and J. B. A. Mitchell.  \textit{Rydberg state formation in $H_3^+$ recombination,} in Proceedings of the Fourth U.S.-Mexico Symposium on Atomic and Molecular Physics, 343-351.  World Scientific, 1995.
\item M. L. Rogelstad, F. B. Yousif, T. J. Morgan, and J. B. A. Mitchell (1997).  \textit{Stimulated radiative recombination of $H^+$ and $He^+$.}  J. Phys. B: At. Mol. Opt. Phys. 30, 3913-3931.
\item J. Min\'a\v{c}, M. Rogelstad, and N. D. T\^an.  \textit{Dimensions of Zassenhaus filtration subquotients of some pro-p groups.}  arXiv: 1405.6980v2, 2015.  To appear in the Israel Journal of Mathematics.  
\end{itemize}

\end{document}